\documentclass[a4paper,12pt]{article}

\usepackage{xeCJK}
\usepackage{charter} 
\usepackage[english]{babel}
\usepackage{fontspec}
\babelprovide{chinese-traditional}
\usepackage[utf8]{inputenc} 
\usepackage[T1]{fontenc}

\usepackage{standalone}
\usepackage{fp}

\usepackage{tabularx}     
\usepackage{latexsym} 
\usepackage{longtable}
\usepackage{ltablex}  
\usepackage{booktabs}  
\usepackage{array}  
\usepackage{float}   
\usepackage{multirow}  
\usepackage{makecell}  

\usepackage[flushleft]{threeparttable}
\usepackage{dcolumn}
\newcolumntype{.}{D{.}{.}{-1}}
\newcolumntype{d}[1]{D{.}{\cdot}{#1}}

\usepackage{psfrag}
\usepackage{epstopdf}
\usepackage{graphics}     
\usepackage{graphicx}
\usepackage{pgfplots}
\usepackage{textcomp}  
\usepackage{pdfpages}  
\usepackage{tikz,fullpage}
\usepackage{latexsym}
\usetikzlibrary{arrows,petri,topaths,backgrounds,snakes,patterns,positioning}
\usetikzlibrary{pgfplots.groupplots}
\definecolor{mygray}{rgb}{0.9,0.9,0.9} 

\usepackage[dvipsnames]{xcolor}

\usepackage{fix-cm}   
\usepackage{geometry}
\usepackage{fullpage}  
\usepackage[titles]{tocloft}  
\usepackage[nottoc]{tocbibind}
\usepackage[raggedright]{titlesec}  
\usepackage[titletoc,title]{appendix}
\usepackage{rotfloat}  
\usepackage[justification=raggedright, labelfont=bf]{caption}  
\usepackage{enumitem}
\setlist[enumerate]{itemsep=-1pt}
\newlist{enumerate_noindent_newline}{enumerate}{1}
\setlist[enumerate_noindent_newline]{label=(\roman*),itemjoin=\\,itemsep=-1pt}
\newlist{enumerate_noindent_nonewline}{enumerate*}{1}
\setlist[enumerate_noindent_nonewline]{label=(\roman*),itemjoin=\,}
\usepackage{relsize}	   
\usepackage{anysize}
\usepackage{fancyhdr} 
\fancypagestyle{myplain}
{
  \fancyhf{}

  \fancyfoot[C]{\thepage}
}
\titleformat{\section}
  {\normalfont\Large\bfseries}{\thesection.}{1em}{}
  \titleformat{\subsection}
  {\normalfont\large\bfseries}{\thesubsection.}{1em}{} 
 

\usepackage[hang, flushmargin, bottom]{footmisc}

\usepackage{pdflscape}

\usepackage{authblk}
\usepackage[plainpages=false,breaklinks=true, citecolor=MidnightBlue, colorlinks=true, linkcolor=MidnightBlue, urlcolor=MidnightBlue,breaklinks]{hyperref}
\usepackage{breakcites}
\usepackage[round]{natbib}

\usepackage{bibentry}
\nobibliography*

\usepackage{amsmath}     
\usepackage{amssymb}
\usepackage{bm}
\usepackage{bbm}
\usepackage{mathrsfs}
\usepackage{mathtools}
\usepackage{rotating}     
\usepackage{amsthm}
\usepackage{etoolbox}
\usepackage{thmtools}
\newtheorem{assumption}{Assumption}
\newtheorem{theorem}{Theorem}
\newtheorem{lemma}{Lemma}
\newtheorem{proposition}{Proposition}
\newtheorem{corollary}{Corollary}

\newtheorem{definition}{Definition}

\newtheorem{example}{Example}

\renewcommand{\thmcontinues}[1]{continued}

\usepackage{cleveref}
\crefname{assumption}{Assumption}{Assumptions}
\Crefname{assumption}{Assumption}{Assumptions}
\crefname{example}{Example}{Examples}
\Crefname{example}{Example}{Examples}
\crefname{condition}{Condition}{Conditions}
\Crefname{condition}{Condition}{Conditions}
\crefname{theorem}{Theorem}{Theorems}
\Crefname{theorem}{Theorem}{Theorems}
\crefname{lemma}{Lemma}{Lemmas}
\Crefname{lemma}{Lemma}{Lemmas}
\crefname{proposition}{Proposition}{Propositions}
\Crefname{proposition}{Proposition}{Propositions}
\crefname{corollary}{Corollary}{Corollaries}
\Crefname{corollary}{Corollary}{Corollaries}
\crefname{definition}{Definition}{Definitions}
\Crefname{definition}{Definition}{Definitions}
\crefname{remark}{Remark}{Remarks}
\Crefname{remark}{Remark}{Remarks}
\crefname{table}{table}{tables}
\Crefname{table}{Table}{Tables}

\usepackage{scalerel}

\DeclareMathOperator*{\Var}{\mathbb{V}}
\newcommand{\E}{\mathbb{E}}

\newcommand{\prob}[1]{\mathbb{P}\left(#1\right)}
\newcommand{\probnb}{\mathbb{P}}

\newcommand{\Probn}{\mathbb{P}_n}
\newcommand{\Probnb}{\bar{\mathbb{P}}_n}
\newcommand{\Eb}[1]{\E\left[#1\right]}
\newcommand{\Ebc}[2]{\Eb{\left. #1\,\right\vert\, #2}}

\newcommand{\Vb}[1]{\Var\left[#1\right]}

\newcommand{\os}[2]{\overset{#1}{#2}}

\makeatletter
\newcommand{\mypm}{\mathbin{\mathpalette\@mypm\relax}}
\newcommand{\@mypm}[2]{\ooalign{%
  \raisebox{.1\height}{$#1+$}\cr
  \smash{\raisebox{-.6\height}{$#1-$}}\cr}}
\makeatother
\makeatletter
\newcommand{\mymp}{\mathbin{\mathpalette\@mymp\relax}}
\newcommand{\@mymp}[2]{\ooalign{%
  \raisebox{.5\height}{$#1-$}\cr
  \smash{\raisebox{-.2\height}{$#1+$}}\cr}}
\makeatother
\makeatletter
\newcommand{\pushright}[1]{\ifmeasuring@#1\else\omit\hfill$\displaystyle#1$\fi\ignorespaces}
\newcommand{\pushleft}[1]{\ifmeasuring@#1\else\omit$\displaystyle#1$\hfill\fi\ignorespaces}
\makeatother
\newcommand{\tld}[1]{\tilde{#1}}
\newcommand{\fr}[2]{\frac{#1}{#2}}

\newcommand{\ba}[1]{\bar{#1}}
\newcommand{\chk}[1]{\check{#1}}

\newcommand{\norm}[2]{\left\lVert{#2}\right\rVert_{#1}}

\newcommand{\supnorm}[1]{\left\lVert{#1}\right\rVert_{\infty}}
\newcommand{\deriv}{\mathrm{d}}
\newcommand{\dderiv}{\,\mathrm{d}}
\newcommand{\Deriv}{\mathrm{D}}
\newcommand{\bigO}[1]{O\left(#1\right)}

\newcommand{\bigOPs}[2]{O_{{#1}}\left(#2\right)}
\newcommand{\smallO}[1]{o\left(#1\right)}

\newcommand{\smallOPs}[2]{o_{#1}\left(#2\right)}
\newcommand{\indep}{\rotatebox[origin=c]{90}{$\models$}}
\renewcommand{\indep}{\!\perp\!\!\!\perp}

\newcommand{\rsquig}{\rightsquigarrow}
\newcommand{\rsquigs}[1]{\overset{#1}{\rightsquigarrow}}

\newcommand{\ceq}{\coloneqq}  
\newcommand{\eqc}{\eqqcolon} 
\newcommand{\indic}[1]{\mathbbm{1}_{#1}}

\newcommand{\Int}{\textup{Int}\,}
\newcommand{\set}[1]{\left\{#1\right\}}
\newcommand{\sumn}[1]{\sum_{#1\in[n]}}
\newcommand{\nsumn}[1]{\fr{1}{n}\sum_{#1\in[n]}}

\newcommand{\normaldist}{\mathcal{N}}
\newcommand{\real}{\mathbb{R}}

\newcommand{\dummy}{\{0,1\}}

\newcommand{\lin}{\mathrm{lin}\,}
\newcommand{\tanset}{\mathcal{T}}
\newcommand{\cond}{{\vert\,}}
\newcommand{\mhit}[1]{\mathit{#1}}
\newcommand{\derivtnull}{\fr{\deriv}{\deriv t}\big\vert_{t=0}}

\newcommand{\Vsupp}{\mathcal{V}}
\newcommand{\Ximagesp}{\mathfrak{X}}
\newcommand{\Vimagesp}{\mathfrak{V}}
\newcommand{\Zimagesp}{\mathfrak{Z}}
\newcommand{\Vsigma}{\mathscr{F}_{\Vimagesp}}
\newcommand{\Xsigma}{\mathscr{F}_{\Ximagesp}}
\newcommand{\Zsigma}{\mathscr{F}_{\Zimagesp}}
\newcommand{\VXsigma}{\mathscr{F}_{\Vimagesp\times\Ximagesp}}

\newcommand{\VZsigma}{\mathscr{F}_{\Vimagesp\times\Zimagesp}}
\newcommand{\Vsigmasub}[1]{\mathscr{F}_{{\Vimagesp}_{#1}}}
\newcommand{\mesV}{\nu_V}
\newcommand{\mesX}{\nu_X}

\newcommand{\mesVXsub}[1]{\nu_{V_{#1}X}}
\newcommand{\bsetV}{B_{\mathsf{v}}}

\newcommand{\bsetX}{B_{\mathsf{x}}}
\newcommand{\bsetZ}{B_{\mathsf{z}}}

\newcommand{\Qbar}{\bar{Q}}
\newcommand{\qbar}{\bar{q}}
\newcommand{\Qinv}{Q^{-1}}
\newcommand{\linopQX}{Q_\Xsupp}
\newcommand{\linopQXadj}{Q_{\Xsupp}^{*}}
\newcommand{\linopQXinv}{\linopQX^{-1}}
\newcommand{\linopQXinvDom}{L_2(P_{VX})\cap L_2(P_V\otimes\Qbar)}
\newcommand{\linopQXinvSpa}{\mathrm{L}_2}
\newcommand{\linopQXinvMes}{P_{\mathrm{L}}}

\newcommand{\setQident}{\mathcal{Q}_{\psi}}
\newcommand{\setQ}{\mathcal{Q}}
\newcommand{\setQdisc}{\mathcal{Q}_J}
\newcommand{\setQdiscinv}{\mathcal{Q}^{\mathrm{I}}_J}
\newcommand{\setQgen}{\mathcal{Q}_{\delta}}

\newcommand{\setQbar}{\mathcal{P}_{\Ximagesp}}
\newcommand{\linopL}{L_Q}

\newcommand{\muX}{\mu_{\Xsupp}}
\newcommand{\muXt}{\mu_{\Xsupp,t}}
\newcommand{\muXhat}{{\hat{\mu}_{\Xsupp}}}
\newcommand{\muXchk}{{\check{\mu}_{\Xsupp}}}
\newcommand{\muXs}[1]{\mu_{#1}}
\newcommand{\muXsub}[1]{\mu_{\Xsupp,#1}}
\newcommand{\muXmodel}{\mathcal{M}(\Theta)}
\newcommand{\gammaV}{\gamma_{\Vsupp}}

\newcommand{\gammaVt}{\gamma_{\Vsupp,t}}
\newcommand{\gammaVhat}{\hat{\gamma}_{\Vsupp}}
\newcommand{\gammaVtld}{\tld{\gamma}_{\Vsupp}}
\newcommand{\gammaVchk}{\check{\gamma}_{\Vsupp}}
\newcommand{\gammaVsub}[1]{\gamma_{\Vsupp,#1}}
\newcommand{\rieszX}{r}
\newcommand{\rieszXs}[1]{\rieszX_{#1}}
\newcommand{\rieszXhat}{\hat{r}}
\newcommand{\rieszXchk}{\check{r}}

\newcommand{\rieszXmodelgamma}{\mathcal{R}_\gamma(\Theta)}

\newcommand{\piX}{\pi_{\Xsupp}}
\newcommand{\piXhat}{\hat{\pi}_{\Xsupp}}

\newcommand{\neigh}{\mathrm{Nb}}
\newcommand{\neighTheta}{\mathrm{Nb}(\theta_0)}
\newcommand{\neighGamma}{\mathrm{Nb}(\gammaV(c))}

\newcommand{\expderiv}{e}
\newcommand{\expderivhat}{\hat{\expderiv}}
\newcommand{\expderivchk}{\check{\expderiv}}
\newcommand{\nweight}{w}
\newcommand{\secnu}{\vartheta}  
\newcommand{\secnuimagesp}{\mathcal{T}}
\newcommand{\secnuhat}{\hat{\secnu}} 
\newcommand{\secnuchk}{\check{\secnu}} 
\newcommand{\infnupar}{\theta}
\newcommand{\infnuparhat}{\hat\theta}
\newcommand{\infnuparchk}{\check\theta}

\newcommand{\innerprodsub}[3]{\langle#1,#2\rangle_{#3}}

\newcommand{\Xsupp}{\mathcal{X}}

\begin{document}




\thispagestyle{empty}

\newgeometry{
  top=25mm,
  bottom=20mm,
  inner=25mm,
  outer=20mm,
}


\title{{\Huge \textbf{Private Rate-Double-Robust Inference}}}
\date{ }

\author[1]{M\'at\'e Kormos}
\author[2]{Aad van der Vaart}

\affil[1]{\footnotesize{Department of Mathematics, Computer Science and Statistics, Ghent University, Krijgslaan 299, \newline Ghent, 9000, Belgium}}
\affil[2]{\footnotesize{Delft Institute of Applied Mathematics, Delft University of Technology, Mekelweg 4, \newline Delft, 2628 CD,  The Netherlands}}

\begin{titlepage}
\maketitle
\thispagestyle{empty}

\vspace{2cm}


\noindent We reconcile privacy protection and rate-double-robust inference. The privacy of individuals is protected by a local privacy mechanism: injecting noise into their sensitive data, revealing only the noisy data for inference. Hence, privacy protection hinders inference. In contrast, the inference of a target parameter is rate-double-robust when the large-sample bias of an estimator of the parameter is characterised by a trade-off between the estimation errors of two other, nuisance, parameters. Hence, rate-double-robustness facilitates inference. Our starting point of reconciliation is a class of rate-double-robust target parameters indexed linearly by an infinite-dimensional and nonlinearly by a low-dimensional regression. Among others, this includes causal parameters. To infer these targets privately, we show how suitable privacy mechanisms transfer the semiparametric properties of the sensitive-data model to the private setting. Rate-double-robustness is transferred, enabling locally-private, unbiased and semiparametrically efficient inference of our target parameters. Finally, we transform general nonparametric nuisance estimators into private ones, which inherit convergence properties of their nonprivate counterparts. For parametric nuisance models, we develop a private method-of-moments estimator and its large-sample inference theory.


\end{titlepage}






\newpage
\setcounter{page}{0}


\pagenumbering{roman}
\begingroup
\clearpage
\let\clearpage\relax
\vspace*{-1.8cm}
\tableofcontents
\endgroup





\newpage

\setcounter{page}{0}
\pagenumbering{arabic}



\section{Introduction}
\label{priv:sec:introduction}

Sensitive data of units in a sample are desirable to protect. This may be accomplished by a privacy mechanism, which disguises the sensitive data by deliberately injecting noise into them. Next, only the noisy --- and not the sensitive --- data are revealed, preserving the privacy of sampled units, but hindering inference.

This paper is concerned with the inference of a parameter $\chi(P_{VX})$ of the distribution $P_{VX}$ of the data $(V,X)$ when $X$ is privacy-protected. Hence, we wish to infer $\chi(P_{VX})$ from the data $(V,Z)$, where $Z$ is the noisy version of $X$ disguised by a given privacy mechanism.

We guarantee privacy by local mechanisms: the noise is injected to the sensitive data of each unit. Then we consider inference under a fixed mechanism, employed uniformly for each unit. While some of our results hold for general local mechanisms, specialising the mechanisms yields more interesting results. Our specialised mechanisms leave the sensitive data intact with probability $\alpha$, and output pure noise with probability $1-\alpha$, guaranteeing total-variation privacy \citep{barber_privacy_2014}. It is this $\alpha$-identity which enables inference. Under this specialisation, $X$ can take values in any measurable space, such as metric spaces. Thus, these mechanisms are much more flexible than those using additive noise \citep{hutchison_calibrating_2006}.

We focus on the inference of parameters $\chi(P_{VX})$ with a rate-double-robustness (or mixed-bias) property. A parameter has this property if the large-sample bias of an estimator thereof is characterised by the product of the estimation errors of two other parameters, which are then called nuisance parameters. The product is attractive as the errors can compensate each other. This is favourable for infinite-dimensional nuisance parameters with large estimation errors. If the product vanishes, $\chi(P_{VX})$ can be inferred unbiasedly in the large-sample limit. Examples of rate-double-robust parameters include average treatment effects.

Our contribution is threefold. First, we propose a novel class of rate-double-robust parameters in the \emph{nonprivate} setting, motivated by \cite{rotnitzky_characterization_2021} and \cite{chernozhukov_automatic_2022}. The class comprises parameters which depend linearly on an infinite-dimensional regression and nonlinearly on a low-dimensional regression. 
While \cite{rotnitzky_characterization_2021} consider dependence parameters more general than regressions, they only allow for linear dependencies.\footnote{Their sufficient conditions for a product-form bias \cite[Proposition 3]{rotnitzky_characterization_2021} stipulate a parameter structure where both factors in the product are ratios of two regressions, with the same denominator in both. In general, the variationally dependent denominators do not naturally translate into double-robustness.  An exception is when the denominator and the nominator are chosen so that the resulting ratio in each factor is itself a regression function. But then these parameters are strictly included in our class.} While \citet[Example 6]{chernozhukov_automatic_2022} show that the average treatment effect on the treated --- which falls into our class --- is rate-double-robust, they do not generalise this result to nonlinear dependencies on low-dimensional regressions. Generalisation is straightforward as low-dimensional parameters are estimable at a fast rate.


Second, turning to the private setting, we provide conditions for the privacy mechanism to infer $\chi(P_{VX})$ from the observed noisy data $(V,Z)\sim P_{VZ}$, and show how our specialised mechanisms satisfy them.
Namely, we connect the semiparametric properties of the statistical models for $P_{VX}$ and $P_{VZ}$. Under our specialised mechanisms, if a parameter is rate-double-robust in the nonprivate setting, then so it remains  in the private setting. This leads to privacy-protected large-sample unbiased inference of $\chi(P_{VX})$ from $(V,Z)$, for fast-enough nuisance estimators. The limiting variance increases with the noise level of the privacy mechanism, but it is semiparametrically efficient in the private model induced by a nonparametric model for $P_{VX}$ and by the specialised mechanism.

Third, we study private estimation of the nuisance parameters. Expressing them as expected-loss minimisers in the \emph{private} setting paves the way for estimation through empirical risk minimisation. Alternatively, given a nonprivate ``source'' estimator in a general class of nonparametric estimators, we transform it to an estimator which uses only the noisy data $(V,Z)$. We show how the transformed estimator inherit the guarantees of its nonprivate source. For example, the convergence rates of privatised kernel and orthogonal series estimators remain the same as those of their nonprivate counterparts inflated by the noise level of the privacy mechanism. For parametric nuisance models, we develop a private method-of-moments estimator for $\real^K$-valued parameters identified from moment conditions \citep{hansen_large_1982,newey_chapter_1994}. We also derive its limiting distribution --- an apparently new result in private inference.

In summary, to the best of our knowledge, our work is the first achieving locally private, efficient and unbiased rate-double-robust inference for parameters as general as the ones in our proposed class, and with data taking values in generic spaces.


In \Cref{priv:sec:literature}, we situate our work in the literature. In \Cref{priv:sec:dr}, we introduce our rate-double-robust class without privacy. \Cref{priv:sec:priv} adds privacy. \Cref{priv:sec:estim_dr_priv} discusses private estimation of the double-robust parameters. \Cref{priv:sec:estim_nuisance_priv} focuses on the private estimation of nuisance parameters. \Cref{priv:sec:conclusion} concludes.



\section{Literature}
\label{priv:sec:literature}

Privacy-preserving inference \citep{warner_randomized_1965,evfimievski_limiting_2003,hutchison_calibrating_2006} can offer central or local privacy guarantees; see \cite{desfontaines_sok_2022} for a  survey.\footnote{Some other privacy notions are  \emph{element-level privacy} \citep{asi_element_2019} or homomorphic encryption \citep{gentry_fully_2009,yang_comprehensive_2019}.} Central mechanisms inject noise into sample aggregates, hence are less stringent than the local ones adopted by us, which noise individual data.

In the central paradigm, parametric (regression) models are studied by \cite{smith_efficient_2008}, \cite{sheffet_differentially_2017}, \cite{alabi_differentially_2020}, \cite{jiang_analysis_2024}, and the nonparametric median by \cite{drechsler_non-parametric_2021}. In the local paradigm, convergence rates \citep{loh_high-dimensional_2012,acharya_hadamard_2019,berrett_strongly_2021} and Fisher-information bounds \citep{barnes_fisher_2020} are derived.

Causal parameters are important instances in our class. Their private inference is addressed by \cite{kusner_private_2016}, \cite{zhu_causal_2022}, \cite{ohnishi_locally_2023}, and \cite{agarwal_causal_2024}. Compared to them, we support more general private covariate adjustment or parameters, under less stringent assumptions about $X$.


Private efficiency theory, also optimising over the privacy mechanism, is pioneered by \cite{steinberger_efficiency_2023} for parametric models. We only consider efficient inference for a \emph{given} mechanism, but we adopt nonparametric models like \cite{duchi_minimax_2018} and \cite{duchi_right_2024} do. They provide minimax rates up to constants, whereas our results are asymptotically exact. \cite{butucea_interactive_2023} considers minimax rates in nonparametric models, specialised to the expected density; we infer parameters in a broad class.

Private M-estimators have been well studied for parametric models \citep{chaudhuri_differentially_2011,kifer_private_2012,bassily_differentially_2014,yamamoto_differentially_2017,lei_differentially_2011,slavkovic_perturbed_2021,mangold_high-dimensional_2023}. Yet, only \cite{asi_element_2019} and \cite{asi_near_2020} appear to derive limiting distributions. Compared to the former, we offer more stringent, local privacy; unlike the latter, our method-of-moments estimator is asymptotically unbiased.




\section{Rate-Double-Robust Inference}
\label{priv:sec:dr}

In this section, we construct our target parameters $\chi(P_{VX})$ and derive their inferential properties \emph{without} privacy. This serves as a basis for private inference in \Cref{priv:sec:priv}. Thus, for now, we consider the nonprivate setting when the sensitive data $(V,X)$ are observable.

Here, $(V,X)$ is a random element defined on the probability space $(\Omega, \mathscr{F}_\Omega, \probnb)$, with distribution $P_{VX}$ belonging to $\mathcal{P}_{VX}\subseteq \mathcal{P}_{\Vimagesp\Ximagesp}$, where $\mathcal{P}_{\Vimagesp\Ximagesp}$ is the set of all possible distributions on the measurable space $(\Vimagesp\times\Ximagesp,\VXsigma)$, where $\Vimagesp=\Vimagesp_\mathfrak{1}\times\Vimagesp_\mathfrak{2}\times\ldots$. Our primary interest is in the nonparametric model
\begin{align}
\mathcal{P}_{VX}=\mathcal{P}_{\Vimagesp\Ximagesp}.\label{priv:eq:nonparamodel_set} \tag{M}
\end{align}

Each functional $\chi:\mathcal{P}_{VX}\to\real$ yields a parameter $\chi(P_{VX})$. \Cref{priv:sec:dr:subsec:class} introduces our class of $\chi$ which yield rate-double-robust parameters; \Cref{priv:sec:dr:subsec:inferential_properties} studies their inferential properties.

\emph{Notation.}\quad For $h:\Vimagesp\times\Ximagesp\to\real$, we let $P_{VX}h\ceq P_{VX}h(V,X)\ceq\int_{\Vimagesp\times\Ximagesp}h(v,x)\dderiv P_{VX}(v,x)$. Fix $p\in[1,\infty)$. With $\norm{L_p(P_{VX})}{h}\ceq\left(P_{VX}|h|^p\right)^{\fr{1}{p}}$, write $L_p(P_{VX})$ for all $h:\Vimagesp\times\Ximagesp\to\real$ with $\norm{L_p(P_{VX})}{h}^p<\infty$, and $L_p^0(P_{VX})$ for all $h\in L_p(P_{VX})$ with $P_{VX}h=0$. Let $\supnorm{h}\ceq\sup_{(v,x)\in\Vimagesp\times\Ximagesp}|h(v,x)|$, and $\rho((h,a),(h',a'))\ceq \norm{L_2(P_{VX})}{h-h'}+|a-a'|$ be a metric on $L_2(P_{VX})\times\real\ni(h,a)$. Let $\delta_{(v,x)}$ be the Dirac measure at $(v,x)\in\Vimagesp\times\Ximagesp$. We call a $(\pi_{1}\circ V,\pi_{2}\circ V,\ldots)$ a collection of coordinates of $V$, if all the $\pi_j\circ v\in \Vimagesp_\mathfrak{j'}$ for all $v\in\Vimagesp$ for some $\Vimagesp_\mathfrak{j'}\in \{\Vimagesp_\mathfrak{1},\Vimagesp_\mathfrak{2},\ldots\}$; for example, if $\Vimagesp=\Vimagesp_\mathfrak{1}\times\Vimagesp_\mathfrak{2}\times\Vimagesp_\mathfrak{3}$ with corresponding $V=(V_1,V_2,V_3)$, then $(V_2,V_1)$ is a collection of coordinates of $V$. 


\subsection{Rate-Double-Robust Parameter Class}
\label{priv:sec:dr:subsec:class}

Let $V_1$ and $V_2$ be two arbitrary collections of coordinates of $V$ with values in $\Vimagesp_{\mhit{1}}$ and $\Vimagesp_{\mhit{2}}$, respectively, where, importantly, $\Vimagesp_{\mhit{2}}$ is finite. For given $m,g:\Vimagesp\times\Ximagesp\to\real$, define the regressions
\begin{align}
\label{priv:eq:regressions_def}\tag{R}
\begin{aligned}
\muX(v_{\mhit{1}},x)&\ceq \Ebc{m(V,X)}{V_{\mhit{1}}=v_{\mhit{1}},X=x},\quad (v_{\mhit{1}},x)\in \Vimagesp_{\mhit{1}}\times\Ximagesp, \\
\gammaV(v_{\mhit{2}})&\ceq \Ebc{g(V,X)}{V_{\mhit{2}}=v_{\mhit{2}}},\quad v_{\mhit{2}}\in\Vimagesp_{\mhit{2}},
\end{aligned}
\end{align}
assuming $\muX\in L_2(P_{V_{\mhit{1}}X}),\gammaV\in L_2(P_{V_{\mhit{2}}})$. For a given $f:\Vimagesp\times\Ximagesp\times L_2(P_{V_{\mhit{1}}X})\times L_2(P_{V_{\mhit{2}}})\to\real$, our targets are $\chi(P_{VX})=\E f(V,X, \muX, \gammaV)$. As $\Vimagesp_{\mhit{2}}$ is finite, we can define without loss of generality our target parameter as
\begin{align}
\chi(P_{VX})&\ceq \E f(V,X, \muX, \gammaV(c)) \label{priv:eq:dr_para}\tag{T}
\end{align}
for a fixed $c\in\Vimagesp_{\mhit{2}}$, and $f:\Vimagesp\times\Ximagesp\times L_2(P_{V_{\mhit{1}}X})\times \Gamma \to\real$, for $\Gamma\supseteq g(\Ximagesp,\Vimagesp)$. We constrain $f$, requiring that
\begin{align}
L_2(P_{V_{\mhit{1}}X})\ni\mu \mapsto  \E  f(V,X,\mu,\gamma) &\text{ be $\norm{L_2(P_{VX})}{\cdot}$-continuous for all }\gamma\in\Gamma, \label{priv:eq:f_continu}\tag{C.C} \\
L_2(P_{V_{\mhit{1}}X})\ni\mu\mapsto f(V,X,\mu,\gamma)& \text{ be linear } P_{VX}\text{-a.s. for all } \gamma\in\Gamma,\label{priv:eq:f_lin}\tag{C.L}  \\
\Gamma \ni\gamma\mapsto f(V,X,\mu,\gamma)& \text{ be twice continuously differentiable }P_{VX}\text{-a.s. for all }\nonumber \\	
&\text{ } \mu\in L_2(P_{V_{\mhit{1}}X})\text{ with } \text{derivatives $\partial_\gamma f$, $\partial_\gamma^2 f$.} \label{priv:eq:f_diff}\tag{C.D}  
\end{align}
Conditions \eqref{priv:eq:f_continu}, \eqref{priv:eq:f_lin}, and \eqref{priv:eq:f_diff} restrict the structure of the dependencies on the regressions to enforce rate-double-robustness. 

Further, we impose the integrability conditions 
\begin{align}
\label{priv:eq:integrability}\tag{C.I}
\begin{aligned}
\E f(V,X, \muX, \gammaV(c))^2 <\infty, \quad \E \indic{V_{\mhit{2}}=c}g(V,X)^2<\infty, \\
\Ebc{m(V,X)^2}{V_{\mhit{1}},X}<\infty\quad \text{ $P_{V_{\mhit{1}}X}$-a.s..}
\end{aligned}
\end{align}


Different choices of $(V_{\mhit{1}},V_{\mhit{2}},m,g,f)$ satisfying the above conditions give rise to a class of parameters of the form \eqref{priv:eq:dr_para}. In \Cref{priv:app:sec:dr:subsec:examples}, we present examples such as average treatment effects (\cite{rotnitzky_characterization_2021}; \cite{chernozhukov_automatic_2022}), ``geometric parameters,'' and parameters from economics. We also demonstrate how dependence on \emph{multiple} regressions can be accommodated. In \Cref{priv:sec:dr:subsec:inferential_properties}, we show how parameters in this class lend themselves to rate-double-robust inference. 

\subsection{Inferential Properties}
\label{priv:sec:dr:subsec:inferential_properties}

Consider the one-step estimator \citep{bolthausen_semiparametric_2002,van_der_vaart_higher_2014,laber_semiparametric_2024} of  $\chi(P_{VX})$. Starting from an arbitrary ``plug-in'' estimator $\chi(\hat P_{VX})$ constructed from a random sample $\mathcal{S}\ceq((V_i,X_i))_{i\in[n]}$ from $P_{VX}$, the one-step estimator 
\begin{align}
\hat\chi_n\ceq \chi(\hat P_{VX})+\Probn \hat{\tld\chi} =  \chi(\hat P_{VX})+\nsumn{i}\hat{\tld\chi}(V_i,X_i) \label{priv:eq:bias_corrected_chihat}
\end{align}
corrects for the plug-in bias via a directional-derivative expansion of the functional $\chi$ \citep{van_der_vaart_higher_2014}. This derivate is representable with the so-called efficient influence function $\tld\chi$ of $\chi(P_{VX})$ for the model $\mathcal{P}_{VX}$. Hence, informally,
\begin{align}
\chi(\hat P_{VX})-\chi(P_{VX})\approx -P_{VX}\hat{\tld\chi}, \label{priv:eq:vonmises_approx}
\end{align}
expanding by the estimate $\hat{\tld\chi}$ of $\tld\chi$ in the direction $P_{VX}-\hat P_{VX}$. This motivates \eqref{priv:eq:bias_corrected_chihat} with the unknown $P_{VX}\hat{\tld\chi}$ estimated with $\Probn\hat{\tld\chi}$. 

To construct $\hat{\tld\chi}$ in \eqref{priv:eq:bias_corrected_chihat}, we need $\tld\chi$. Let $\rieszX\in L_2(P_{V_{\mhit{1}}X})$ be the function satisfying
\begin{align}
\E f(V,X, \mu, \gammaV(c))=\E \rieszX(V_{\mhit{1}},X)\mu(V_{\mhit{1}},X)\quad \text{ for all } \mu\in L_2(P_{V_{\mhit{1}}X}). \label{priv:eq:riesz_rep}
\end{align}
The existence and uniqueness of $\rieszX$ follows from the Riesz representation theorem by \eqref{priv:eq:f_continu} and \eqref{priv:eq:f_lin}, whence $\rieszX$ is called the Riesz representer. The representer $\rieszX$, whose dependence on $(P_{VX},f,c)$ is silent in our notation, is obtained by manipulating the left-hand side of \eqref{priv:eq:riesz_rep} until the right-hand side is reached (see \Cref{priv:app:sec:dr:subsec:examples}). With $\rieszX$, $\tld\chi$ is derived in \Cref{priv:eqprop:eif}. The dependence on the regressions manifests itself in \eqref{priv:eq:eif_chi}; under no dependence, the (generalised) derivates $\rieszX,\partial_\gamma f$ are zero.

\begin{proposition}[Efficient Influence Function of $\chi(P_{VX})$]
\label{priv:eqprop:eif}
In the nonparametric model \eqref{priv:eq:nonparamodel_set} for $P_{VX}$, the efficient influence function $\tld\chi:
\Vimagesp\times\Ximagesp\to\real$ of $\chi(P_{VX})$ in \eqref{priv:eq:dr_para} is, at $P_{VX}$,
\begin{align}
\label{priv:eq:eif_chi}
\begin{aligned}
\tld\chi(v,x)\ceq&\, \rieszX(v_{\mhit{1}},x)(m(v,x)-\muX(v_{\mhit{1}},x))\\
&+\fr{\indic{v_{\mhit{2}}=c}}{p_{V_{\mhit{2}}}(c)}(g(v,x)-\gammaV(c))\E \partial_\gamma f(V,X, \muX, \gammaV(c))  \\
&+ f(v,x, \muX, \gammaV(c))-\chi(P_{VX}), \quad (v,x)\in\Ximagesp\times\Vimagesp, 
\end{aligned}
\end{align}
where we denote by $v_{\mhit{1}},v_{\mhit{2}}$ the coordinates of $v$ that correspond to $V_{\mhit{1}},V_{\mhit{2}}$ of $V$. When $V_\mhit{2}=\varnothing$, it is understood that $\fr{\indic{v_{\mhit{2}}=c}}{p_{V_{\mhit{2}}}(c)}(g(v,x)-\gammaV(c))=g(v,x)-\E g(V,X)$.
\end{proposition}
\begin{proof}
All proofs are in the appendix.
\end{proof}

Now we verify that $\chi(P_{VX})$ in \eqref{priv:eq:dr_para} is rate-double-robust. \Cref{priv:thm:dr} below controls the approximation error in \eqref{priv:eq:vonmises_approx}. Because $\Vimagesp_{\mhit{2}}$ is finite, the first, product term dominates in \eqref{priv:eq:chihat_biasdecomp} for reasonable ``estimators''  $p_{V_{\mhit{2}}}'(c),\gammaV'(c)$. This proves the one-step estimator \eqref{priv:eq:bias_corrected_chihat} of $\chi(P_{VX})$ rate-double-robust with bias characterised by the product of estimation errors of the nuisance parameters $\muX,\rieszX$.

\begin{theorem}[Rate-Double-Robustness]
\label{priv:thm:dr}
Let $\muX',r'\in L_2(P_{V_{\mhit{1}}X})$ and $p_{V_{\mhit{2}}}'(c),\gammaV'(c),\chi',\expderiv''\in\real$ all be arbitrary. Set 
\begin{align}
\label{priv:eq:eif_chi_prime}
\begin{aligned}
\tld\chi'(v,x)\ceq&\, \rieszX'(v_{\mhit{1}},x)(m(v,x)-\muX'(v_{\mhit{1}},x))+\fr{\indic{v_{\mhit{2}}=c}}{p_{V_{\mhit{2}}}'(c)}(g(v,x)-\gammaV'(c))\expderiv''  \\
&+ f(v,x, \muX', \gammaV'(c))-\chi'. 
\end{aligned}
\end{align}
Then
\begin{align}
\label{priv:eq:chihat_biasdecomp}
\begin{aligned}
\chi'-\chi(P_{VX})+P_{VX}\tld\chi' =&\, -P_{VX}(r-r')(\muX-\muX')   \\
&+(\gammaV(c)-\gammaV'(c))\left(\fr{p_{V_{\mhit{2}}}(c)}{p_{V_{\mhit{2}}}'(c)}\expderiv''-\expderiv'\right) \\
&-(\gammaV(c)-\gammaV'(c))^2\fr{P_{VX}\partial_\gamma^2 f(V,X,\muX',\widetilde{\gammaV(c)})}{2}  
\end{aligned}
\end{align}
for some $\widetilde{\gammaV(c)}$ between $\gammaV(c)$ and $\gammaV'(c)$, and $\expderiv'\ceq P_{VX}\partial_\gamma f(V,X,\muX',\gammaV'(c))$.
\end{theorem}

The function $\tld\chi$ is a key object in semiparametric efficiency theory. By definition, it is an element of the $L_2(P_{VX})$-completion $\overline{\lin}\tanset_{VX}$ of the linear span of the \emph{tangent set of the model $\mathcal{P}_{VX}$ at $P_{VX}$}, $\tanset_{VX}$: the set of all
directions in which $P_{VX}$ can be perturbed so that the perturbed $P_{VX}$ remains in $\mathcal{P}_{VX}$. For the nonparametric model $\mathcal{P}_{\Vimagesp\Ximagesp}$ in \eqref{priv:eq:nonparamodel_set}, $\overline{\lin}\tanset_{VX}=\tanset_{VX}=L_2^0(P_{VX})$. Among all elements of $\overline{\lin}\tanset_{VX}$, it is $\tld\chi$ whose squared norm $P_{VX}\tld\chi^2$ is the limiting variance of asymptotically efficient estimators of $\chi(P_{VX})$ (\Citet[Definition 2.8, Part III]{bolthausen_semiparametric_2002}). In the nonprivate setting, we show that \eqref{priv:eq:bias_corrected_chihat} is asymptotically efficient (\Cref{priv:app:sec:nonprivate_estimation}). In the following, we prove an analogue in the private setting.





\section{Privacy}
\label{priv:sec:priv}

To preserve the privacy of each unit in the sample $\mathcal{S}=((V_i,X_i))_{i\in[n]}$, a noisy version $Z_i$ of $X_i$ is generated, and only $\bar{\mathcal{S}}\ceq((V_i,Z_i))_{i\in[n]}$ is revealed to infer $\chi(P_{VX})$. In \Cref{priv:sec:priv:subsec:priv_mech}, we describe how $Z_i$ is generated to enable inference from $\bar{\mathcal{S}}$, which is discussed in \Cref{priv:sec:priv:subsec:inferential_properties}.

\subsection{Privacy Mechanism}
\label{priv:sec:priv:subsec:priv_mech}

The $Z_i$ are generated from $X_i$ via a privacy mechanism, which can be thought of as a random map. Given a measurable space $(\Zimagesp,\Zsigma)$, the $Z_i$ are generated as random draws $Z_i\cond ((V_j,X_j))_{j\in[n]}\sim Q(\cdot\cond X_i)$ for all $i\in[n]$ for a $Q\in\setQ(\Ximagesp\to\Zimagesp)$, where $\setQ(\Ximagesp\to\Zimagesp)$ is the set of all Markov kernels $Q:\Zsigma\times\Ximagesp\to[0,1]$, so that $B\mapsto Q(B\cond x)$ is a probability measure for all $x\in\Ximagesp$, and $x\mapsto Q(B\cond x)$ is measurable for every $B\in\Zsigma$. The kernel $Q$ is called a \emph{local noninteractive privacy mechanism}: local, because it noises the data of each unit $i$, and noninteractive, because it does not use the data of units $j\neq i$ to generate $Z_i$ \citep{steinberger_efficiency_2023}. Therefore, $((V_i,Z_i))_{i\in[n]}$ is a random sample from the distribution of $(V,Z)$, the mixture
\begin{align}
P_{VZ}(\bsetV, \bsetZ)= \int_{\bsetV}\int_{\Ximagesp}Q(\bsetZ\cond x) \dderiv P_{VX}(v,x),\quad \bsetV\in\Vsigma, \bsetZ\in\Zsigma. \label{priv:eq:distribution_vz}
\end{align}
The distribution $P_{VZ}$, determined by $Q$ and $P_{VX}$, thus belongs to the model
\begin{align}
\mathcal{P}_{VZ}(Q,\mathcal{P})\ceq\left\{P\in\mathcal{P}_{\Vimagesp\Zimagesp}: P(\bsetV, \bsetZ)= \int_{\bsetV}\int_{\Ximagesp}Q(\bsetZ\cond x) \dderiv \tld P(v,x) \right.  \nonumber \\
\text{ holds for all } \left. \bsetV\in\Vsigma, \bsetZ\in\Zsigma, \text{ as } \tld P \text{ runs through } \mathcal{P}\subset \mathcal{P}_{\Vimagesp\Ximagesp} \vphantom{\int_{\bsetV}\int_{\Ximagesp}} \right\}, \label{priv:eq:vzq_model}
\end{align}
the set of all possible distributions of $(V,Z)$ induced by the mechanism $Q$ as the distribution of $(V,X)$ varies across $\mathcal{P}$; here, $\mathcal{P}_{\Vimagesp\Zimagesp}$ is the set of all probability distributions on $(\Vimagesp\times\Zimagesp,\VZsigma)$. Fixing $\mathcal{P}=\mathcal{P}_{VX}$ and $\tld P=P_{VX}$ in \eqref{priv:eq:vzq_model} yields $P=P_{VZ}$.

How to choose the mechanism? First, the output space $\Zimagesp$ has to be specified. The choice of $\Zimagesp$ is an unexplored topic in privacy literature, beyond our current scope. 
 Hence, for some of our results to follow, $\Zimagesp$ can be any given measurable space; however, $\Zimagesp=\Ximagesp$ --- the usual choice in the literature --- shall yield more insightful results.

Second, given $\Zimagesp$, a mechanism $Q\in\setQ(\Ximagesp\to\Zimagesp)$ has to be chosen. Regard \eqref{priv:eq:distribution_vz} as the flow of information from $P_{VX}$ to $P_{VZ}$. If $Q(\cdot\cond x)$ does not depend on $x$, then $Z$ carries no information about $X$, constituting maximal privacy but precluding inference. In the other extreme, if $\Zimagesp=\Ximagesp$, and $Q(\cdot\cond x)=\delta_{x}$ concentrates on $X$, then $P_{VZ}=P_{VX}$, leading to the opposite effect. Thus, a sufficient and necessary condition for the identification of \emph{every} parameter $\chi(P_{VX})$ from $P_{VZ}$ is the existence of a map $\linopL:\mathcal{P}_{VZ}(Q, \mathcal{P}_{VX})\to\mathcal{P}_{VX}$ such that
\begin{align}
 P_{VX}=\linopL(P_{VZ}). \label{priv:eq:linopL}
 \end{align}
The map $\linopL$ inverts \eqref{priv:eq:distribution_vz} to recover $P_{VX}$ from every $P_{VZ}$ generated by a given $Q$. With $\linopL$, every parameter\footnote{The existence of an $\linopL$ satisfying \eqref{priv:eq:linopL} is not necessary (but clearly sufficient) for the identification of some parameters --- e.g.\ when $\chi$ is the functional of only the marginal $P_V$.} of the sensitive-data distribution is identifiable from the noisy-data distribution as
\begin{align}
\psi(P_{VZ})\ceq \chi \circ \linopL(P_{VZ}) = \chi(P_{VX}). \label{priv:eq:identification_psi}
\end{align}
The dependence of $\psi$ on $Q$ remains implicit in our notation. To infer $\chi(P_{VX})$, we wish to choose $Q$ such that $L_Q$ exists. Consider first a discrete $X$.

\begin{example}[Finitely discretely distributed $X$]
\label{priv:ex:discrete_covar}
Suppose that $\Ximagesp=\set{x_1,\ldots,x_{|\Ximagesp|}}$ and $\Zimagesp=\set{z_1,\ldots,z_{|\Zimagesp|}}$ are finite sets, and that $P_{VX}$ has a  $\mesV\times\mesX$-density $p_{VX}$ for the counting measure $\mesX$.
Then $P_{VZ}$ in \eqref{priv:eq:distribution_vz} admits a $\mesV\times\mesX$-density
\begin{align}
p_{VZ}(v,z)=\sum_{x\in\Ximagesp} Q(\set{z}\cond x)p_{VX}(v,x),\quad (v,z)\in\Vimagesp\times\Zimagesp. \label{priv:eq:vz_disrete_dens}
\end{align}
Representing $Q$ as the $|\Zimagesp|$-by-$|\Ximagesp|$ matrix
\begin{align}
Q = \begin{bmatrix}
(Q(\set{z_j}\mid x_1))_{j\in[|\Zimagesp|]} & (Q(\set{z_j}\mid x_2))_{j\in[|\Zimagesp|]}  & \cdots & (Q(\set{z_j}\mid x_{|\Ximagesp|}))_{j\in[|\Zimagesp|]} 
\end{bmatrix}, \label{priv:eq:q_as_matrix}
\end{align}
the display \eqref{priv:eq:vz_disrete_dens} is equivalent to
\begin{align*}
[0,1]^{|\Zimagesp|\times 1}\ni \ba p_{VZ}(v) \ceq 
(p_{VZ}(v,z_j))_{j\in[|\Zimagesp|]}
= Q (p_{VX}(v,x_j))_{j\in[{|\Ximagesp|}]}
\eqc Q \ba p_{VX}(v),\quad v\in\Vimagesp.
\end{align*}
For $z\in\real^{|\Zimagesp|\times 1}, x\in\real^{|\Ximagesp|\times 1}$ and  $Q\in\setQ(\set{x_1,\ldots,x_{|\Ximagesp|}}\to\set{z_1,\ldots,z_{|\Zimagesp|}})$, consider the system of linear equations
$z=Qx$ in $x$. Only if $|\Zimagesp|\geq|\Ximagesp|$, can this system have a unique solution. Let us impose $|\Zimagesp|\ceq|\Ximagesp|\eqc J$, and set
\begin{align}
\setQdisc\ceq \setQ(\set{x_1,\ldots,x_J}\to\set{z_1,\ldots,z_J}). \label{priv:eq:discreteQdef}
\end{align}
Then the system has a unique solution for all $z\in\real^{J\times 1}$ if and only if $Q\in\setQdisc$ viewed as a matrix is invertible with inverse $\Qinv$, in which case the solution is $\Qinv z$  (e.g.\ \Citet[Theorems 1.3 and 1.4]{piziak_matrix_2007}). Conclude that if $Q$ is invertible, then $\ba p_{VX}(v)=\Qinv \ba p_{VZ}(v)$ for all $v\in\Vimagesp$. Hence, if 
\begin{align}
Q\in \setQdiscinv\ceq \set{Q\in \setQdisc: \Qinv \text{ exists}}  \label{priv:eq:discreteQinv}
\end{align}
then $\linopL$ in \eqref{priv:eq:linopL} exists, is unique, and is completely determined by the matrix $\Qinv$. An example (\Citet[Section 2.2]{steinberger_efficiency_2023}) of $Q\in \setQdiscinv$ is
\begin{align}
Q=c_{J,\alpha}
\begin{bmatrix}
e^\alpha & 1 & \cdots &1 \\
1 & e^\alpha & \cdots & 1 \\
\vdots & \vdots & \ddots & \vdots \\
1 & 1& \cdots & e^\alpha
\end{bmatrix}, \Qinv = c_{J,\alpha}^{(-1)}
\begin{bmatrix}
e^\alpha+J-2 & -1 & \cdots &-1 \\
-1 & e^\alpha+J-2 & \cdots & -1 \\
\vdots & \vdots & \ddots & \vdots \\
-1 & -1& \cdots & e^\alpha+J-2
\end{bmatrix}\label{priv:eq:discrete_covar_Qexample}
\end{align}
with $c_{J,\alpha}\ceq\fr{1}{e^\alpha+J-1}$ and $c_{J,\alpha}^{(-1)}\ceq \fr{e^\alpha+J-1}{e^{2\alpha}+(J-2)e^\alpha-J+1}$ for any $\alpha>0$. 
\end{example}

\Cref{priv:ex:discrete_covar} shows in \eqref{priv:eq:discrete_covar_Qexample} the role of the parameter $\alpha$ determining the privacy level: an $\alpha\approx 0$ equalises the entries of $Q$, with no information flowing from $X$ to $Z$. Formally, $Q$ in \eqref{priv:eq:discrete_covar_Qexample} satisfies 
$(e^\alpha-1)$-total-variational privacy for any $0<\alpha\leq\log(2)$:

\begin{definition}[{{Local $\alpha$-Total-Variation Privacy ($\alpha$-LTVP) \Citet[Definition 4]{barber_privacy_2014}}}]
\label{priv:def:local_tv_priv}
For $0\leq\alpha\leq1$, a mechanism $Q\in\setQ(\Ximagesp\to\Zimagesp)$ is locally $\alpha$-total-variationally private if $\sup_{B\in\Zsigma}|Q(B\cond x)-Q(B\cond x')|\leq\alpha$ for all $x,x'\in\Ximagesp$.
\end{definition}

For generic $X$ as well, total-variational privacy proves to be a suitable paradigm to ensure the existence of $\linopL$. Consider 
 \begin{align}
 \setQgen\ceq \left\{ Q\in \setQ(\Ximagesp\to\Ximagesp): Q(B \cond x)= \alpha \delta_x(B)+(1-\alpha)\Qbar(B), \alpha\in(0,1), \Qbar\in \setQbar \right\}, \label{priv:eq:qtv}
 \end{align}
where $\setQbar$ is the set of all probability measures on $(\Ximagesp,\Xsigma)$. The $Z$ drawn from a mechanism in \eqref{priv:eq:qtv} for a unit with $X=x$ equals $x$ itself with probability $\alpha$, and is pure noise drawn from $\Qbar$ with probability $1-\alpha$. Hence, the smaller $\alpha$, the stricter the privacy. Trivially, any $Q$ in \eqref{priv:eq:qtv} is $\alpha$-LTVP, and, by \Cref{priv:lem:qtc_implications} in \Cref{priv:app:priv:subsec:aux}, it ensures the existence of
\begin{align*}
(\linopL P_{VZ})(\bsetV,\bsetX)\ceq \fr{1}{\alpha}P_{VZ}(\bsetV,\bsetX)-\fr{1-\alpha}{\alpha}P_V(\bsetV)\Qbar(\bsetX) \\
=\fr{1}{\alpha}P_{VZ}(\bsetV,\bsetX)-\fr{1-\alpha}{\alpha}P_{VZ}(\bsetV, \Ximagesp)\Qbar(\bsetX)=P_{VX}(\bsetV,\bsetX),\quad \bsetV\in\Vsigma,\bsetX\in\Xsigma,
\end{align*}
a linear map, whereby the identification \eqref{priv:eq:identification_psi} of $\chi(P_{VX})$ from $P_{VZ}$ readily follows.

It may seem restrictive that for discrete $X$ we allow for any invertible mechanism, but for generic $X$ we confine ourselves to \eqref{priv:eq:qtv}. However, for generic $X$, it appears difficult to obtain $L_Q$ without a Dirac measure; see \Cref{priv:app:priv:subsec:invertible_mechanisms} for further discussion. In summary, we collect in
\begin{align}
\setQident \ceq \setQdiscinv\cup\setQgen
\end{align}
the set of mechanisms implying \eqref{priv:eq:identification_psi}, understanding that $Q\in\setQdiscinv$ only if $\Ximagesp$ and $P_{VX}$ are as in \Cref{priv:ex:discrete_covar}.

\subsection{Private Inferential Properties}
\label{priv:sec:priv:subsec:inferential_properties}

Now we derive the private analogue of the inferential properties in \Cref{priv:sec:dr:subsec:inferential_properties}. This entails the efficient influence function of $\psi(P_{VZ})$ in \eqref{priv:eq:identification_psi} at $P_{VZ}$ in the model $\mathcal{P}_{VZ}(Q, \mathcal{P}_{VX})$ of \eqref{priv:eq:vzq_model}, \emph{given} a mechanism $Q\in\setQ(\Ximagesp\to\Zimagesp)$.\footnote{See \Cref{priv:lem:semiparametric_vz} in \Cref{priv:app:priv:subsec:private_inferential_prop} for further semiparametric properties.} In line with private-inference practice, $Q$ is treated as common knowledge available for inference. Some of our results hold for any model $\mathcal{P}_{VX}$, but the main interest is in the nonparametric model $\mathcal{P}_{\Vimagesp\Ximagesp}$ of \eqref{priv:eq:nonparamodel_set}.

For $Q\in\setQ(\Ximagesp\to\Zimagesp)$, define the linear operator $\linopQX: L_2(P_{VZ})\to L_2(P_{VX})$ as
\begin{align}
(\linopQX k)(v,x)&\ceq \int_\Zimagesp k(v,z)Q(\deriv z\cond x)=\Ebc{k(V,Z)}{V=v,X=x},  \label{priv:eq:linopQX} 
\end{align}
whose properties are derived in \Cref{priv:lem:linopqx_properties} in \Cref{priv:app:priv:subsec:aux}. In particular, it has adjoint $\linopQXadj:L_2(P_{VX})\to L_2(P_{VZ}), h\mapsto\Ebc{h(V,X)}{V=\cdot,Z=\cdot}$, and is invertible: when $Q\in\setQdisc$, then $\linopQXinv$ exists and is unique if and only if $Q\in\setQdiscinv$ with inverse $(Q^\intercal)^{-1}$ viewed as a transposed matrix; more generally, when $Q\in\setQgen$, then the range-restricted $\linopQX:L_2(P_{VZ})\to \linopQXinvDom$ and its inverse $\linopQXinv:\linopQXinvDom\to L_2(P_{VZ})$ are
\begin{align*}
(\linopQX k)(v,x)&=\alpha k(v,x)+(1-\alpha)\int_\Ximagesp k(v,z)\Qbar(\deriv z), \quad (v,x)\in\Vimagesp\times\Ximagesp, \\
(\linopQXinv h)(v,z)&=\fr{1}{\alpha}h(v,z)-\fr{1-\alpha}{\alpha}\int_\Ximagesp h(v,x)\Qbar (\deriv x), \quad (v,z)\in\Vimagesp\times\Ximagesp.
\end{align*}

\begin{theorem}[Efficient Influence Function of $\psi(P_{VZ})$]
\label{priv:prop:eif_priv}
Suppose that $\chi(P_{VX})$ has efficient influence function $\varphi\in \overline{\lin}\tanset_{VX}$ at $P_{VX}$ in some model $\mathcal{P}_{VX}\subset \mathcal{P}_{\Vimagesp\Ximagesp}$ with tangent set $\tanset_{VX}$, and that $Q\in\setQident$. If $\varphi \in \linopQX\linopQXadj \tanset_{VX}$, then the efficient influence function of $\psi(P_{VZ})$ at $P_{VZ}$ in the model $\mathcal{P}_{VZ}(Q, \mathcal{P}_{VX})$ of \eqref{priv:eq:vzq_model} is $\linopQXinv\varphi$. If $\varphi=\tld\chi$ in \eqref{priv:eq:eif_chi} in the nonparametric model $\mathcal{P}_{VX}=\mathcal{P}_{\Vimagesp\Ximagesp}$ satisfies $\tld\chi\in\linopQXinvDom$, then 
\begin{align}
\tld\psi\ceq \linopQXinv \tld\chi\label{priv:eq:eif_chi_vz}
\end{align}
is the efficient influence function of $\psi(P_{VZ})$ at $P_{VZ}$ in the model $\mathcal{P}_{VZ}(Q, \mathcal{P}_{\Vimagesp\Ximagesp})$.
\end{theorem}

In \Cref{priv:prop:eif_priv}, the conditions $\varphi \in \linopQX\linopQXadj \tanset_{VX}$ and $\tld\chi\in\linopQXinvDom$ are important. For instance, if $\Qbar$ of $Q\in\setQgen$ has large mass at extreme locations of $\tld\chi$, the latter may fail. If they hold, then an asymptotically efficient estimator of $\psi(P_{VZ})$ based on a random sample from $P_{VZ}\in \mathcal{P}_{VZ}(Q,$ $\mathcal{P}_{\Vimagesp\Ximagesp})$ with a given $Q\in\setQident$ has limiting variance 
\begin{align}
P_{VZ}\tld\psi^2=P_{VX}[\linopQX(\tld\psi^2)]=P_{VX}[\linopQX[(\linopQXinv \tld\chi)(\linopQXinv\tld\chi)]] \label{priv:eq:eff_var_vz}
\end{align}
by the properties of $\linopQX$. When there is no privacy, so $\linopQX$ and $\linopQXinv$ are the identity, \eqref{priv:eq:eff_var_vz} equals the nonprivate efficiency bound $P_{VX}\tld\chi^2$ in \Cref{priv:sec:dr:subsec:inferential_properties}, as expected.  Specifically, if $Q\in\setQgen$, then we have the bounds 
\begin{align}
\begin{aligned}\label{priv:eq:eff_var_vz_bound}
P_{VX}\tld\chi^2+\fr{1-\alpha}{\alpha}\left((P_V\otimes\Qbar)\tld\chi^2- P_V\left(\int \tld\chi(V,x)\Qbar(\deriv x)\right)^2\right) \leq P_{VZ}\tld\psi^2 \\ \leq \fr{2-\alpha}{\alpha} P_{VX}\tld\chi^2+\fr{2(2-\alpha)(1-\alpha)}{\alpha^2}(P_V\otimes\Qbar)\tld\chi^2
\end{aligned}
\end{align}
by \Cref{priv:lem:linopqx_properties}; hence, the private efficiency bound $P_{VZ}\tld\psi^2$ is never smaller than the nonprivate bound $P_{VX}\tld\chi^2$, and the stricter the privacy, the larger this gap in general.




\section{Private Estimation}
\label{priv:sec:estim_dr_priv}

In this section, we construct a private analogue of the one-step estimator \eqref{priv:eq:bias_corrected_chihat} of \eqref{priv:eq:dr_para}, and show that it reaches the efficient limit \eqref{priv:eq:eff_var_vz} in the nonparametric model \eqref{priv:eq:nonparamodel_set} for $P_{VX}$, given a mechanism $Q\in\setQident$.
We assume that three, mutually independent, random samples $\bar{\mathcal{S}}=((V_i,Z_i))_{i\in[n]}$, $\bar{\mathcal{S}}'=((V_i',Z_i'))_{i\in[n]},$ $\bar{\mathcal{S}}''=((V_i'',Z_i''))_{i\in[n]}$ from $P_{VZ}$ are available for inference.\footnote{If  $\mathcal{S}=\mathcal((V_i,X_i))_{i\in[n]}$, $\mathcal{S}'= ((V_i',X_i'))_{i\in[n]}$, $\mathcal{S}''=((V_i'',X_i''))_{i\in[n]}$ are three, mutually independent, random samples from $P_{VX}$, then, given a $Q\in\setQident$, the samples $\bar{\mathcal{S}}',\bar{\mathcal{S}}''$ are obtained by drawing $Z_i'\cond (\mathcal{S},\bar{\mathcal{S}},\mathcal{S}',\mathcal{S}'')\sim Q(\cdot\cond X_i')$ and $Z_i''\cond (\mathcal{S},\bar{\mathcal{S}},\mathcal{S}',\bar{\mathcal{S}}',\mathcal{S}'')\sim Q(\cdot\cond X_i'')$ for all $i\in[n]$.}

Analogously to \eqref{priv:eq:bias_corrected_chihat}, we begin with an initial estimator $\psi(\hat P_{VZ})$ and correct it as
\begin{align}
\hat\psi_n &\ceq \psi(\hat P_{VZ})+\Probnb \hat{\tld\psi}= \psi(\hat P_{VZ})+\nsumn{i} \hat{\tld\psi}(V_i,Z_i),
\end{align}
where $\Probnb$ is the empirical measure of $\bar{\mathcal{S}}$, and 
\begin{align}
\hat{\tld\psi}\ceq&\, \linopQXinv \check{\tld\chi}, \label{priv:eq:psihat} \\
\chk{\tld\chi}(v,x)\ceq&\, \rieszXchk(v_{\mhit{1}},x)(m(v,x)-\muXchk(v_{\mhit{1}},x))+\fr{\indic{v_{\mhit{2}}=c}}{\chk p_{V_{\mhit{2}}}(c)}(g(v,x)-\gammaVchk(c))\expderivchk \nonumber \\
&+ f(v,x, \muXchk, \gammaVchk(c))-\psi(\hat P_{VZ}) 
\eqc \chk{\tld\chi}_0(v,x)- \psi(\hat P_{VZ}). \label{priv:eq:eif_chi_hat_vz}
\end{align}
are estimates of the private and nonprivate influence functions $\tld\psi$ and $\tld\chi$ in \eqref{priv:eq:eif_chi_vz} and \eqref{priv:eq:eif_chi}, respectively. The inverse $\linopQXinv$ of \eqref{priv:eq:linopQX} is known as $Q$ is known. As $\linopQXinv h=h$ for a constant function $h$, we have $\hat\psi_n=\nsumn{i}  (\linopQXinv\chk{\tld\chi}_0)(V_i,Z_i)$, so $\psi(\hat P_{VZ})\in\real$ can be arbitrary.

In contrast, all the estimates 
\begin{align}
\label{priv:eq:nuisance_vz}
\begin{aligned}
\chk\eta&\ceq (\rieszXchk,\muXchk,\gammaVchk(c),\chk p_{V_{\mhit{2}}}(c), \expderivchk)\in L_2(P_{V_{\mhit{1}}X})\times L_2(P_{V_{\mhit{1}}X}) \times\Gamma\times \real\times\real\,\text{ of }  \\ 
\eta&\ceq (\rieszX,\muX,\gammaV(c), p_{V_{\mhit{2}}}(c), \expderiv) 
\end{aligned}
\end{align}
are based on the noisy samples $\bar{\mathcal{S}}'$ and $\bar{\mathcal{S}}''$ as clarified in \Cref{priv:tab:priv_dr_samples}. Specifically, by \Cref{priv:lem:linopqx_properties}, 
\begin{align}
P_{VZ}\linopQXinv h=P_{VX}h\,\text{ for all } h\in\linopQXinvDom. \label{priv:eq:changemeasure}
\end{align}
This motivates the estimators
\begin{align}
\expderivchk \ceq  \nsumn{i} \partial_\gamma \bar{f}(V_i'',Z_i'',\muXchk,\gammaVchk(c)), \label{priv:eq:expderivhat_vz} \\
\gammaVchk(c)\ceq  \fr{1}{\chk p_{V_{\mhit{2}}}(c)}\nsumn{i}\ba g_c(V_i',Z_i'),\quad \chk p_{V_{\mhit{2}}}(c)\ceq N_c/n,\, N_c\ceq \sumn{i}\indic{V_{\mhit{2}i}'=c}, \label{priv:eq:gammaVchk}\\
\bar{f}(v,z,\mu,\gamma)\ceq (\linopQXinv(v,x)\mapsto f(v,x,\mu,\gamma))(v,z), \, (v,z,\mu,\gamma)\in \Vimagesp\times\Zimagesp\times L_2(P_{V_{\mhit{1}}X})\times\Gamma, \label{priv:eq:fbar} \\
\ba g_c(v,z)\ceq (\linopQXinv g_c)(v,z),\quad  g_c(v,x)\ceq  \indic{v_{\mhit{2}}=c}g(v,x),\quad (v,x,z)\in\Vimagesp\times\Ximagesp\times\Zimagesp.
\end{align}
Indeed, $P_{VZ}\partial_\gamma \bar{f}=P_{VX}\partial_\gamma f$ and $\E\gammaVchk(c)=\gammaV(c)$, by $\partial_\gamma \bar{f}=\linopQXinv \partial_\gamma f$ and the definition \eqref{priv:eq:regressions_def} of $\gammaV(c)$ as $\fr{1}{p_{V_{\mhit{2}}}(c)}\E g_c(V,X)$. It follows that $\gammaV(c)$ and $p_{V_{\mhit{2}}}(c)$ are estimable by \eqref{priv:eq:gammaVchk} at rate $\bigOPs{P_{VZ}}{n^{-1/2}}$ because of their low dimensionality.

The estimation of $(\muX,\rieszX)$ is addressed in \Cref{priv:sec:estim_nuisance_priv}. If the estimators \eqref{priv:eq:nuisance_vz} are consistent and $f$ is continuous in an appropriate norm accounting for the noise measure $\Qbar$ in \eqref{priv:eq:qtv} as specified in \Cref{priv:ass:consistent_nuisance_priv} in \Cref{priv:app:sec:assumptions}, then the second, empirical process term in the decomposition
\begin{align}
\sqrt{n}(\hat\psi_n - \psi(P_{VZ})) &= \sqrt{n}\Probnb \tld\psi+\sqrt{n}(\Probnb-P_{VZ})(\hat{\tld{\psi}}-\tld\psi)+\sqrt{n}\bar{R}_n, \label{priv:eq:chi_eff_decomp_vz} \\
\ba R_n&\ceq \psi(\hat P_{VZ})-\psi(P_{VZ})+P_{VZ} \hat{\tld\psi}, \label{priv:eq:chi_eff_decomp_bias_vz}
\end{align}
is $\smallOPs{P_{VZ}}{1}$. The first term $\sqrt{n}\Probnb \tld\psi \rsquigs{P_{VZ}}$ $\normaldist(0,P_{VZ}\tld\psi^2)$ by the standard central limit theorem as $\tld\psi\in L_2^0(P_{VZ})$ is the efficient influence function \eqref{priv:eq:eif_chi_vz}. The change-of-measure property \eqref{priv:eq:changemeasure} combined with \Cref{priv:prop:eif_priv} and \eqref{priv:eq:eif_chi_hat_vz} implies that the third term
\begin{align}
 \label{priv:eq:bias_r_vz}
 \begin{aligned}
\bar R_n =&\ \psi(\hat P_{VZ})-\chi(P_{VX})+P_{VX}\chk{\tld\chi} \\
=&\ -P_{VX}(\rieszX-\rieszXchk)(\muX-\muXchk) +(\gammaV(c)-\gammaVchk(c))\left(\fr{p_{V_{\mhit{2}}}(c)}{\chk p_{V_{\mhit{2}}}(c)}\expderivchk-\bar{\expderiv}'\right) \\
&-(\gammaV(c)-\gammaVchk(c))^2\fr{P_{VX}\partial_\gamma^2 f(V,X,\muXchk,\gammaVchk(c))}{2},  
\end{aligned}
\end{align}
for some $\gammaVtld(c)$ between $\gammaV(c)$ and $\gammaVchk(c)$, and
\begin{align}
\bar{\expderiv}'\ceq P_{VX}\partial_\gamma f(V,X,\muXchk,\gammaVchk(c)) = P_{VZ}\partial_\gamma \bar{f}(V,Z,\muXchk,\gammaVchk(c)). \label{priv:eq:derivhatprime_vz}
\end{align}
Whence, $\sqrt{n}\ba R_n=\smallOPs{P_{VZ}}{1}$ under vanishing product $(\muXchk-\muX)(\rieszXchk-\rieszX)$ of estimation errors and regularity conditions. This yields our main result whereby $\hat\psi_n$ is rate-double-robust and asymptotically efficient in the nonparametric model \eqref{priv:eq:nonparamodel_set} for $P_{VX}$.

\begin{assumption}[Rates of Private Estimators]
\label{priv:ass:rates_nuisance_priv}
The estimators \eqref{priv:eq:nuisance_vz} satisfy $P_{VX}((\muXchk-\muX)(\rieszXchk-\rieszX))=\smallOPs{P_{VZ}}{n^{-1/2}}$, $\gammaVchk(c)-\gammaV(c)=\bigOPs{P_{VZ}}{n^{-1/2}}$, $\chk p_{V_{\mhit{2}}}(c)-p_{V_{\mhit{2}}}(c)=\bigOPs{P_{VZ}}{n^{-1/2}}$, and $P_{VX} \partial_\gamma^2 f(V,X,\muXchk,\gammaVtld(c))=\bigOPs{P_{VZ}}{1}$ for $\gammaVtld(c)$ between $\gammaVchk(c)$ and $\gammaV(c)$.
\end{assumption}
 
\begin{corollary}[Asymptotic Efficiency of $\hat\psi_n$]
\label{priv:cor:psihat_efficiency}
Suppose that $P_{VZ}\in\mathcal{P}_{VZ}(Q, \mathcal{P}_{\Vimagesp\Ximagesp})$ for a fixed mechanism $Q\in\setQident$. If \Cref{priv:ass:consistent_nuisance_priv,priv:ass:rates_nuisance_priv} hold, then
$\sqrt{n}(\hat\psi_n-\psi(P_{VZ}))\rsquigs{P_{VZ}}\normaldist(0, P_{VZ}\tld\psi^2)$ as $n\to\infty$.
\end{corollary}

\begin{table}
\caption{Use of Samples for Estimation}\label{priv:tab:priv_dr_samples}
\centering
\begin{threeparttable}[c]
\centering
\begin{tabular}{lccc}
Estimators & \multicolumn{3}{c}{Samples}  \\ 
& 				$\bar{\mathcal{S}}$ & $\bar{\mathcal{S}}'$ & $\bar{\mathcal{S}}''$ \\ \toprule
$\hat\psi_n$ 	&		\checkmark	&	\checkmark		& 	\checkmark	\\
$\expderivchk$ & 				&		\checkmark		&	\checkmark	\\
$\rieszXchk,\muXchk,\gammaVchk(c), \chk p_{V_{\mhit{2}}}(c)$ & & \checkmark & \\ \bottomrule
\end{tabular}
\begin{tablenotes}
\item \footnotesize{A sample is used in the construction of an estimator if and only if $\checkmark$ is present in their corresponding cell.}
\end{tablenotes}
\end{threeparttable}
\end{table}



\section{Private Estimation of Nuisance Parameters}
\label{priv:sec:estim_nuisance_priv}

\Cref{priv:cor:psihat_efficiency} essentially shows that if the product of the estimation errors of the regression $\muX$ and of the Riesz representer $\rieszX$ is small enough, then $\hat\psi_n$ is asymptotically efficient. In this section, we consider the estimation of $(\muX,\rieszX)$ from the random sample $\bar{\mathcal{S}}'$ of \Cref{priv:sec:estim_dr_priv} from $P_{VZ}\in\mathcal{P}_{VZ}(Q, \mathcal{P}_{VX})$ given a fixed mechanism $Q\in\setQident$.

By the Cauchy--Schwarz inequality, the product is bounded by individual errors as $P_{VX}(\muXchk-\muX)(\rieszXchk-\rieszX)\leq \sqrt{P_{VX}(\muXchk-\muX)^2}\sqrt{P_{VX}(\rieszXchk-\rieszX)^2}$. If $\muX$ belongs to a finite-dimensional model smoothly indexed by $\theta\in\real^K$, then the product vanishes fast enough for \Cref{priv:cor:psihat_efficiency} to apply under regularity conditions, even if $\rieszXchk$ is an infinite-dimensional nonparametric thus slower estimator, and \emph{vice versa}. To accommodate this tradeoff, we present results for finite- and infinite-dimensional models.

For finite-dimensional smooth models, a common estimator is the method-of-moments (\cite{hansen_large_1982,newey_chapter_1994}). In \Cref{priv:prop:rate_mm_priv}, we devise a locally-private method-of-moments estimator enabled by property \eqref{priv:eq:changemeasure}, and derive its limiting distribution under the standard regularity conditions of  \Cref{priv:ass:mm_priv} in \Cref{priv:app:priv_nuisance:subsec:finite_dim}. As in  \eqref{priv:eq:eff_var_vz_bound}, the worst-case dependence of the limiting variance $\ba\Sigma$ on $\alpha$ is $1/\alpha^2$ driven by $\ba \Phi$.

\begin{proposition}[Private Method-of-Moments]
\label{priv:prop:rate_mm_priv}
Let $\theta_0\ceq\arg\min_{\theta\in\Theta} P_{VX}\Xi_\theta$ for a fixed $\Xi_{\theta}:\Vimagesp\times\Ximagesp\to\real$, $\theta\in\Theta\subset\real^K$, for a fixed $K$. Let $\phi_{\tld\theta}(v,x)\ceq\Deriv_\theta \Xi_{\tld\theta}(v,x)^\intercal$ and $\dot\phi_{\tld\theta}(v,x)\ceq \Deriv_\theta \phi_{\tld\theta}(v,x)$ be the derivates as maps to $\real^{K\times 1}$ and to $\real^{K\times K}$, respectively, for $(v,x,\tld\theta)\in\Vimagesp\times\Ximagesp\times\Theta$. Let $\ba A_n\in\real^{K\times K}$ be an arbitrary sequence of (possibly random and then $\sigma(\bar{\mathcal{S}}')$-measurable) matrices with $\ba A_n\os{P_{VZ}}{\to}\ba A_0$ as $n\to\infty$ for a symmetric positive definite $\ba A_0$, and $\chk\theta$ be the solution to 
$\theta\mapsto \ba\Lambda_n(\theta)\ceq \big(\Probnb'\ba\phi_{\theta}^\intercal\big)\ba A_n(\Probnb'\ba\phi_{\tld\theta}\big) \equiv 0$ 
up to $\ba\Lambda_n(\chk\theta)=\smallOPs{P_{VZ}}{n^{-1/2}}$, where $\ba\phi_{\tld\theta}\ceq\Deriv_\theta \ba\Xi_{\tld\theta}$ with $\ba \Xi_\theta\ceq \linopQXinv \Xi_\theta$ for the inverse $\linopQXinv$ of \eqref{priv:eq:linopQX}. Let \Cref{priv:ass:mm_priv} in \Cref{priv:app:priv_nuisance:subsec:finite_dim} hold. If $P_{VZ}\in\mathcal{P}_{VZ}(Q, \mathcal{P}_{VX})$, for a fixed $Q\in\setQident$ and $P_{VX}\in\mathcal{P}_{VX}$ satisfying the given assumptions, then $\sqrt{n}(\chk\theta-\theta_0)\overset{P_{VZ}}{\rsquig} \mathcal{N}(0,\ba\Sigma)$ as $n\to\infty$, where $\ba \Sigma \ceq ({\dot\Phi^\intercal}\ba A_0{\dot\Phi})^{-1}{\dot\Phi^\intercal}\ba A_0 \ba \Phi \ba A_0{\dot\Phi}({\dot\Phi^\intercal}\ba A_0{\dot\Phi})^{-1}$, $\dot \Phi \ceq P_{VX} \dot \phi_{\theta_0}$, $\ba \Phi \ceq P_{VZ}\ba\phi_{\theta_0}\ba\phi_{\theta_0}^\intercal$
with $\dot \Phi =P_{VX}\dot\phi_{\theta_0}$ and $\phi_{\theta_0}=\Deriv_\theta\Xi_{\theta_0}^\intercal$.
Further, let $\xi_{\theta}:\Vimagesp\times\Ximagesp\to\real$, $\theta\in\Theta$, be (possibly random and then $\sigma(\bar{\mathcal{S}}')$-measurable) functions satisfying \Cref{priv:ass:mm_priv}. 
Then $\norm{L_2(P_{VX})}{\xi_{\chk\theta}-\xi_{\theta_0}}=\bigOPs{P_{VZ}}{n^{-1/2}}$.
\end{proposition}

If $\muX$ or $\rieszX$ follows a smooth model satisfying the conditions of \Cref{priv:prop:rate_mm_priv}, then they are estimable at parametric rate $\bigOPs{P_{VZ}}{n^{-1/2}}$. Indeed, the correctly parametrised versions of the (well-known) relations (\Cref{priv:lem:riesz_maxim} in \Cref{priv:app:priv_nuisance})
\begin{align}
\muX = \arg\min_{\mu\in L_2(P_{V_\mhit{1}}X)}P_{VX}\bigg[\Delta_{\mu}^2(V,X)\bigg], \quad  \Delta_{\mu}^2(v,x)\ceq (m(v,x) - \mu(v_{\mhit{1}},x))^2,  \label{priv:eq:sq_crit}  \\
\rieszX_\gamma = \arg\min_{h\in L_2(P_{V_\mhit{1}}X)}P_{VX}\Upsilon_{\gamma,h}, \quad  \Upsilon_{\gamma, h}(v,x) \ceq h(v_\mhit{1},x)^2-2f(v,x,h,\gamma), \gamma\in\Gamma, \label{priv:eq:risz_crit}
\end{align}
for the Riesz representer $\rieszX_\gamma$ of $\mu\mapsto \E f(V,X,\mu,\gamma)$, lead to $\norm{L_2(P_{VX})}{\xi_{\chk\theta}-\xi_{\theta_0}}=\bigOPs{P_{VZ}}{n^{-1/2}}$ for $\xi_{\theta_0}\in\{\muX,\rieszX\}$; see \Cref{priv:cor:muXpara_rate_priv,priv:cor:rieszXpara_rate_priv} in  \Cref{priv:app:priv_nuisance}.

For infinite-dimensional models $\theta_0\in\{\muX,\rieszX\}$, we transform nonprivate estimators $\infnuparhat$ into private ones for estimation from the noisy data $\bar{\mathcal{S}}'$. We consider estimators and their private transformation of the form
\begin{align}
\infnuparhat(v_\mhit{1},x)\ceq\nsumn{i}\nweight_{n,i}(v_\mhit{1}, x, V_{i}',X_i',\secnuhat), \label{priv:eq:nonparaestim} \\
\label{priv:eq:nonparaestim_priv}
\begin{aligned}
\infnuparchk(v_\mhit{1},x)&\ceq\nsumn{i}\ba\nweight_{n,i}(v_\mhit{1},x,V_i',Z_i',\secnuchk), \\
 \ba\nweight_{n,i}(v_\mhit{1},x,v',z',\secnu)&\ceq (\linopQXinv(v',x')\mapsto \nweight_{n,i}(v_\mhit{1}, x, v',x',\secnu))(v',z'),
\end{aligned}
\end{align}
where $\nweight_{n,i}:\Vimagesp_{\mhit{1}}\times\Ximagesp\times\Vimagesp\times\Ximagesp\times\secnuimagesp\to\real$, $i\in[n]$, is a triangular array of known, nonrandom functions. Here, $\secnuhat$ estimates $\secnu\in\secnuimagesp$, which allows $\theta_0$ to depend on a ``secondary-nuisance'' parameter $\secnu$. We remain agnostic about $\secnuimagesp$, affording flexible formulations. The operator $\linopQXinv$ is the inverse of  \eqref{priv:eq:linopQX}, and $\secnuchk$ is an estimator of $\secnu$ computed from $\bar{\mathcal{S}}'$.

\Cref{priv:prop:priv_bound} bounds the error of $\infnuparchk$. In \eqref{priv:eq:nonparaestim_priv_bound}, the first term is the secondary-nuisance error (e.g.\ the error of the conditioning density in kernel regression estimates), which is trivially zero if there is no dependence on $\secnu$; the second is a variance term, which usually scales inversely with the effective sample size (e.g. $1/(nh^d)$ for $d$-dimensional kernel estimators with bandwidth $h$) modulo the noise level of the privacy mechanism. The third term in \eqref{priv:eq:nonparaestim_priv_bound} is the squared bias, which, thanks to \eqref{priv:eq:changemeasure}, is identical to the nonprivate bias (e.g.\ $h^{2\beta}$ for said kernel estimates of $\beta$-smooth parameters). These terms generally require a case-by-case analysis, but arguments in the nonprivate setting can carry over to the private setting under the mechanism $Q\in\setQgen$. The noise level generally affects the first two terms, but never the third. This can translate into ``usual,'' \emph{nonprivate} convergence rates of \emph{private} versions of well-known estimators such as kernel or orthogonal series, modulo the noise level of the mechanism $Q$; see \Cref{priv:app:priv_nuisance:subsec:infinite_dim}.


\begin{proposition}[Private Error Bounds]
\label{priv:prop:priv_bound}
Fix a $Q\in\setQident$, and let $\infnuparchk$ be defined according to \eqref{priv:eq:nonparaestim_priv}. Suppose that for a sequence of constants $(a_n)$,
\begin{align}
\label{priv:eq:secnubound}
\begin{aligned}
P_{VX}T_{n}^2&=\bigOPs{P_{VZ}}{a_n^2},  \\
T_{n}(v_\mhit{1}, x)&\ceq \nsumn{i}\left(\ba\nweight_{n,i}(v_\mhit{1},x,V_i',Z_i',\secnuchk)-\ba\nweight_{n,i}(v_\mhit{1},x,V_i',Z_i',\secnu)\right),\quad (v_\mhit{1},x)\in \Vimagesp_{\mhit{1}}\times\Ximagesp. 
\end{aligned}
\end{align}
Let $\ba \sigma_i^2(v_\mhit{1},x)\ceq \Vb{\ba\nweight_{n,i}(v_\mhit{1},x,V,Z,\secnu)}$, $(i,v_\mhit{1},x)\in [n]\times \Vimagesp_{\mhit{1}}\times\Ximagesp$. Then for any $\theta\in L_2(P_{V_{\mhit{1}}X})$,
\begin{align}
\label{priv:eq:nonparaestim_priv_bound}
\begin{aligned}
P_{VX}(\infnuparchk-\infnupar)^2 \leq &\, \bigOPs{P_{VZ}}{a_n^2}+\bigOPs{P_{VZ}}{\fr{1}{n^2}\sumn{i}\int \ba \sigma_i^2(v_\mhit{1},x)  \dderiv P_{V_{\mhit{1}}X}(v_{\mhit{1}},x)} \\
&+4\int\left\{\nsumn{i}P_{VX}[\nweight_{n,i}(v_\mhit{1},x,V,X,\secnu)]-\infnupar(v_\mhit{1},x)\right\}^2\dderiv P_{V_{\mhit{1}}X}(v_{\mhit{1}},x).
\end{aligned}
\end{align}
Morover, if $Q\in\setQgen$, then 
\begin{align}
\ba \sigma_i^2(v_\mhit{1},x)\leq \fr{4}{\alpha^2}\E \nweight_{n,i,}^2(v_\mhit{1},x, V, Z,\secnu) \label{priv:eq:sigmacondbound}
\end{align}
for all $(i,v_\mhit{1},x)\in [n]\times\Vimagesp_{\mhit{1}}\times\Ximagesp$.
\end{proposition}

\section{Conclusion}
\label{priv:sec:conclusion}

We introduced a class of rate-double-robust target parameters indexed linearly by an infinite-dimensional and nonlinearly by a low-dimensional regression. The inference of these targets was considered in a private setting, when the data --- with values in general spaces --- of each individual are protected by a local privacy mechanism. We derived the semiparametric properties of the private model, and constructed asymptotically efficient estimators under a fixed privacy mechanism. This efficiency was shown attainable under the same rate-double-robustness condition as in the nonprivate setting, involving two infinite-dimensional parameters. Therefore, we concluded with the private estimation of these two parameters, preserving their nonprivate rates modulo the noise level of the privacy mechanism.

\vspace*{0cm}
{\scriptsize
\bibliography{./thesis_print.bib}

@article{kohler_multivariate_2008,
  title = {Multivariate {{Orthogonal Series Estimates}} for {{Random Design Regression}}},
  author = {Kohler, Michael},
  year = {2008},
  month = oct,
  journal = {Journal of Statistical Planning and Inference},
  volume = {138},
  number = {10},
  pages = {3217--3237},
  publisher = {Elsevier BV},
  issn = {0378-3758},
  doi = {10.1016/j.jspi.2008.01.011},
  urlll = {https://linkinghub.elsevier.com/retrieve/pii/S0378375808000475},
  urldate = {2025-07-23},
  copyright = {https://www.elsevier.com/tdm/userlicense/1.0/},
  langid = {english}
}

@book{tsybakov_introduction_2009,
  title = {Introduction to {{Nonparametric Estimation}}},
  author = {Tsybakov, Alexandre B.},
  year = {2009},
  series = {Springer {{Series}} in {{Statistics}}},
  publisher = {Springer New York},
  address = {New York, NY},
  issn = {0172-7397},
  doi = {10.1007/b13794},
  urlll = {https://link.springer.com/10.1007/b13794},
  urldate = {2025-07-29},
  copyright = {https://www.springernature.com/gp/researchers/text-and-data-mining},
  isbn = {978-0-387-79051-0 978-0-387-79052-7},
  langid = {english}
}

@InProceedings{kusner_private_2016,
  title = 	 {Private Causal Inference},
  author = 	 {Kusner, Matt J. and Sun, Yu and Sridharan, Karthik and Weinberger, Kilian Q.},
  booktitle = 	 {Proceedings of the 19th International Conference on Artificial Intelligence and Statistics},
  pages = 	 {1308--1317},
  year = 	 {2016},
  editor = 	 {Gretton, Arthur and Robert, Christian C.},
  volume = 	 {51},
  series = 	 {Proceedings of Machine Learning Research},
  address = 	 {Cadiz, Spain},
  month = 	 {09--11 May},
  publisher =    {PMLR},
  urlll = 	 {https://proceedings.mlr.press/v51/kusner16.html}
}

@article{butucea_interactive_2023,
  title = {Interactive versus Noninteractive Locally Differentially Private Estimation: {{Two}} Elbows for the Quadratic Functional},
  shorttitle = {Interactive versus Noninteractive Locally Differentially Private Estimation},
  author = {Butucea, Cristina and Rohde, Angelika and Steinberger, Lukas},
  year = {2023},
  month = {apr},
  journal = {The Annals of Statistics},
  volume = {51},
  number = {2},
  issn = {0090-5364},
  doi = {10.1214/22-AOS2254},
  urlll = {https://projecteuclid.org/journals/annals-of-statistics/volume-51/issue-2/Interactive-versus-noninteractive-locally-differentially-private-estimation--Two-elbows/10.1214/22-AOS2254.full},
  urldate = {2026-05-12}
}

@article{ray_semiparametric_2020,
  title = {Semiparametric {{Bayesian}} Causal Inference},
  author = {Ray, Kolyan and Van Der Vaart, Aad},
  year = 2020,
  month = oct,
  journal = {The Annals of Statistics},
  volume = {48},
  number = {5},
  issn = {0090-5364},
  doi = {10.1214/19-AOS1919},
  urlll = {https://projecteuclid.org/journals/annals-of-statistics/volume-48/issue-5/Semiparametric-Bayesian-causal-inference/10.1214/19-AOS1919.full},
  urldate = {2025-08-21}
}

@unpublished{pollard_asymptotia_2005,
  title = {Asymptopia: {{An Exposition}} of {{Statistical Asymptotic Theory}}},
  author = {Pollard, David},
  year = 2005,
  urlll = {http://www.stat.yale.edu/~pollard/Courses/607.spring05/handouts/DQM.pdf}
}

@inproceedings{acharya_hadamard_2019,
  title = {{Hadamard} {Response}: {Estimating} {Distributions} {Privately}, {Efficiently}, and {With} {Little} {Communication}},
  booktitle = {Proceedings of the Twenty-Second International Conference on Artificial Intelligence and Statistics},
  author = {Acharya, Jayadev and Sun, Ziteng and Zhang, Huanyu},
  editor = {Chaudhuri, Kamalika and Sugiyama, Masashi},
  year = {2019},
  series = {Proceedings of Machine Learning Research},
  volume = {89},
  pages = {1120--1129},
  publisher = {PMLR},
  urlll = {https://proceedings.mlr.press/v89/acharya19a.html}
}

@misc{agarwal_causal_2024,
  title = {Causal {{Inference}} {{With}} {{Corrupted Data}}: {{Measurement Error}}, {{Missing Values}}, {{Discretization}}, and {{Differential Privacy}}},
  shorttitle = {Causal {{Inference}} with {{Corrupted Data}}},
  author = {Agarwal, Anish and Singh, Rahul},
  year = {2024},
  month = feb,
  number = {arXiv:2107.02780},
  eprint = {2107.02780},
  primaryclass = {cs, econ, math, stat},
  publisher = {arXiv},
  urlll = {http://arxiv.org/abs/2107.02780},
  urldate = {2024-08-26},
  archiveprefix = {arXiv}
}

@misc{alabi_differentially_2020,
  title = {Differentially {{Private Simple Linear Regression}}},
  author = {Alabi, Daniel and McMillan, Audra and Sarathy, Jayshree and Smith, Adam and Vadhan, Salil},
  year = {2020},
  month = jul,
  number = {arXiv:2007.05157},
  eprint = {2007.05157},
  primaryclass = {cs, stat},
  publisher = {arXiv},
  urlll = {http://arxiv.org/abs/2007.05157},
  urldate = {2023-05-23},
  archiveprefix = {arXiv}
}

@misc{asi_element_2019,
  title = {Element {{Level Differential Privacy}}: {{The Right Granularity}} of {{Privacy}}},
  shorttitle = {Element {{Level Differential Privacy}}},
  author = {Asi, Hilal and Duchi, John C. and Javidbakht, Omid},
  year = {2019},
  month = dec,
  number = {arXiv:1912.04042},
  eprint = {1912.04042},
  primaryclass = {cs, stat},
  publisher = {arXiv},
  urlll = {http://arxiv.org/abs/1912.04042},
  urldate = {2024-09-13},
  archiveprefix = {arXiv}
}

@misc{asi_near_2020,
  title = {Near {{Instance-Optimality}} in {{Differential Privacy}}},
  author = {Asi, Hilal and Duchi, John C.},
  year = {2020},
  month = may,
  number = {arXiv:2005.10630},
  eprint = {2005.10630},
  primaryclass = {cs, stat},
  publisher = {arXiv},
  urlll = {http://arxiv.org/abs/2005.10630},
  urldate = {2024-09-13},
  archiveprefix = {arXiv}
}

@misc{barber_privacy_2014,
  title = {Privacy and {{Statistical Risk}}: {{Formalisms}} and {{Minimax Bounds}}},
  shorttitle = {Privacy and {{Statistical Risk}}},
  author = {Barber, Rina Foygel and Duchi, John C.},
  year = {2014},
  month = dec,
  number = {arXiv:1412.4451},
  eprint = {1412.4451},
  primaryclass = {cs, math, stat},
  publisher = {arXiv},
  urlll = {http://arxiv.org/abs/1412.4451},
  urldate = {2024-08-23},
  archiveprefix = {arXiv}
}

@article{barnes_fisher_2020,
  title = {Fisher {{Information Under Local Differential Privacy}}},
  author = {Barnes, Leighton Pate and Chen, Wei-Ning and {\"O}zg{\"u}r, Ayfer},
  year = {2020},
  journal = {IEEE Journal on Selected Areas in Information Theory},
  volume = {1},
  number = {3},
  pages = {645--659},
  doi = {10.1109/JSAIT.2020.3039461}
}

@misc{bassily_differentially_2014,
  title = {Differentially {{Private Empirical Risk Minimization}}: {{Efficient Algorithms}} and {{Tight Error Bounds}}},
  shorttitle = {Differentially {{Private Empirical Risk Minimization}}},
  author = {Bassily, Raef and Smith, Adam and Thakurta, Abhradeep},
  year = {2014},
  month = oct,
  number = {arXiv:1405.7085},
  eprint = {1405.7085},
  primaryclass = {cs, stat},
  publisher = {arXiv},
  urlll = {http://arxiv.org/abs/1405.7085},
  urldate = {2023-05-24},
  archiveprefix = {arXiv}
}

@article{berrett_strongly_2021,
  title = {Strongly {{Universally Consistent Nonparametric Regression}} and {{Classification}} {{With}} {{Privatised Data}}},
  author = {Berrett, Thomas B. and Gy{\"o}rfi, L{\'a}szl{\'o} and Walk, Harro},
  year = {2021},
  month = jan,
  journal = {Electronic Journal of Statistics},
  volume = {15},
  number = {1},
  issn = {1935-7524},
  doi = {10.1214/21-EJS1845}
}

@incollection{bolthausen_semiparametric_2002,
  title = {Semiparametric {{Statistics}}},
  booktitle = {Lectures on {{Probability Theory}} and {{Statistics}}},
  author = {Bolthausen, Erwin and {van der Vaart}, Aad W. and Perkins, Edwin},
  editor = {Bernard, Pierre and Morel, Jean-Michel and Takens, Floris and Teissier, Bernard},
  year = {2002},
  series = {Lecture {{Notes}} in {{Mathematics}}},
  volume = {1781},
  publisher = {Springer Berlin Heidelberg},
  address = {Berlin, Heidelberg},
  doi = {10.1007/b93152},
  urlll = {http://link.springer.com/10.1007/b93152},
  urldate = {2024-03-15},
  isbn = {978-3-540-43736-9 978-3-540-47944-4}
}

@misc{chaudhuri_differentially_2011,
  title = {Differentially {{Private Empirical Risk Minimization}}},
  author = {Chaudhuri, Kamalika and Monteleoni, Claire and Sarwate, Anand D.},
  year = {2011},
  month = feb,
  number = {arXiv:0912.0071},
  eprint = {0912.0071},
  primaryclass = {cs},
  publisher = {arXiv},
  urlll = {http://arxiv.org/abs/0912.0071},
  urldate = {2023-05-24},
  archiveprefix = {arXiv}
}

@article{chernozhukov_automatic_2022,
  title = {Automatic {{Debiased Machine Learning}} of {{Causal}} and {{Structural Effects}}},
  author = {Chernozhukov, Victor and Newey, Whitney K. and Singh, Rahul},
  year = {2022},
  journal = {Econometrica},
  volume = {90},
  number = {3},
  pages = {967--1027},
  issn = {0012-9682},
  doi = {10.3982/ECTA18515},
  urlll = {https://www.econometricsociety.org/doi/10.3982/ECTA18515},
  urldate = {2024-08-13},
  langid = {english}
}

@misc{desfontaines_sok_2022,
  title = {{{SoK}}: {{Differential Privacies}}},
  shorttitle = {{{SoK}}},
  author = {Desfontaines, Damien and Pej{\'o}, Bal{\'a}zs},
  year = {2022},
  month = nov,
  number = {arXiv:1906.01337},
  eprint = {1906.01337},
  primaryclass = {cs},
  publisher = {arXiv},
  urlll = {http://arxiv.org/abs/1906.01337},
  urldate = {2024-05-23},
  archiveprefix = {arXiv}
}

@misc{drechsler_non-parametric_2021,
  title = {Non-{{Parametric Differentially Private Confidence Intervals}} for the {{Median}}},
  author = {Drechsler, Joerg and {Globus-Harris}, Ira and McMillan, Audra and Sarathy, Jayshree and Smith, Adam},
  year = {2021},
  month = jul,
  number = {arXiv:2106.10333},
  eprint = {2106.10333},
  primaryclass = {cs, stat},
  publisher = {arXiv},
  urlll = {http://arxiv.org/abs/2106.10333},
  urldate = {2023-05-17},
  archiveprefix = {arXiv}
}

@article{duchi_minimax_2018,
  title = {Minimax {{Optimal Procedures}} for {{Locally Private Estimation}}},
  author = {Duchi, John C. and Jordan, Michael I. and Wainwright, Martin J.},
  year = {2018},
  month = jan,
  journal = {Journal of the American Statistical Association},
  volume = {113},
  number = {521},
  pages = {182--201},
  issn = {0162-1459, 1537-274X},
  doi = {10.1080/01621459.2017.1389735},
  urlll = {https://www.tandfonline.com/doi/full/10.1080/01621459.2017.1389735},
  urldate = {2024-09-12},
  langid = {english}
}

@article{duchi_right_2024,
  title = {The {Right} {Complexity} {Measure} in {Locally} {Private} {Estimation}: {It} {Is} {Not} the {Fisher} {Information}},
  shorttitle = {The Right Complexity Measure in Locally Private Estimation},
  author = {Duchi, John C. and Ruan, Feng},
  year = {2024},
  month = feb,
  journal = {The Annals of Statistics},
  volume = {52},
  number = {1},
  issn = {0090-5364},
  doi = {10.1214/22-AOS2227},
  urlll = {https://projecteuclid.org/journals/annals-of-statistics/volume-52/issue-1/The-right-complexity-measure-in-locally-private-estimation--It/10.1214/22-AOS2227.full},
  urldate = {2024-08-23}
}

@inproceedings{evfimievski_limiting_2003,
  title = {Limiting {{Privacy Breaches}} in {{Privacy Preserving Data Mining}}},
  booktitle = {Proceedings of the Twenty-Second {{ACM SIGMOD-SIGACT-SIGART}} Symposium on {{Principles}} of Database Systems},
  author = {Evfimievski, Alexandre and Gehrke, Johannes and Srikant, Ramakrishnan},
  year = {2003},
  month = jun,
  pages = {211--222},
  publisher = {ACM},
  address = {San Diego California},
  doi = {10.1145/773153.773174},
  urlll = {https://dl.acm.org/doi/10.1145/773153.773174},
  urldate = {2024-09-11},
  isbn = {978-1-58113-670-8},
  langid = {english}
}

@phdthesis{gentry_fully_2009,
  title = {A {{Fully Homomorphic Encryption Scheme}}},
  author = {Gentry, Craig},
  year = {2009},
  urlll = {https://crypto.stanford.edu/craig},
  school = {Stanford University}
}

@article{hahn_role_1998,
  title = {On the {{Role}} of the {{Propensity Score}} in {{Efficient Semiparametric Estimation}} of {{Average Treatment Effects}}},
  author = {Hahn, Jinyong},
  year = {1998},
  month = mar,
  journal = {Econometrica},
  volume = {66},
  number = {2},
  eprint = {2998560},
  eprinttype = {jstor},
  pages = {315},
  issn = {00129682},
  doi = {10.2307/2998560},
  urlll = {https://www.jstor.org/stable/2998560?origin=crossref},
  urldate = {2021-09-17}
}

@article{hansen_large_1982,
  title = {Large {{Sample Properties}} of {{Generalized Method}} of {{Moments Estimators}}},
  author = {Hansen, Lars Peter},
  year = {1982},
  month = jul,
  journal = {Econometrica},
  volume = {50},
  number = {4},
  eprint = {1912775},
  eprinttype = {jstor},
  pages = {1029},
  issn = {00129682},
  doi = {10.2307/1912775},
  urlll = {https://www.jstor.org/stable/1912775?origin=crossref},
  urldate = {2024-09-06}
}

@incollection{hutchison_calibrating_2006,
  title = {Calibrating {{Noise}} to {{Sensitivity}} in {{Private Data Analysis}}},
  booktitle = {Theory of {{Cryptography}}},
  author = {Dwork, Cynthia and McSherry, Frank and Nissim, Kobbi and Smith, Adam},
  editor = {Hutchison, David and Kanade, Takeo and Kittler, Josef and Kleinberg, Jon M. and Mattern, Friedemann and Mitchell, John C. and Naor, Moni and Nierstrasz, Oscar and Pandu Rangan, C. and Steffen, Bernhard and Sudan, Madhu and Terzopoulos, Demetri and Tygar, Dough and Vardi, Moshe Y. and Weikum, Gerhard and Halevi, Shai and Rabin, Tal},
  year = {2006},
  volume = {3876},
  pages = {265--284},
  publisher = {Springer Berlin Heidelberg},
  address = {Berlin, Heidelberg},
  doi = {10.1007/11681878_14},
  urlll = {http://link.springer.com/10.1007/11681878_14},
  urldate = {2024-09-11},
  isbn = {978-3-540-32731-8 978-3-540-32732-5}
}

@inproceedings{jiang_analysis_2024,
  title = {Analysis of {{Differentially Private Synthetic Data}}: {{A Measurement Error Approach}}},
  booktitle = {{{AAAI Conference}} on {{Artificial Intelligence}}},
  author = {Jiang, Yangdi and Liu, Yi and Yan, Xiaodong and Charest, Anne-Sophie and Kong, Linglong and Jiang, Bei},
  year = {2024},
  urlll = {https://api.semanticscholar.org/CorpusID:268699082}
}

@inproceedings{kifer_private_2012,
  title = {Private {{Convex Empirical Risk Minimization}} and {{High-Dimensional Regression}}},
  booktitle = {Proceedings of the 25th {{Annual Conference}} on {{Learning Theory}}},
  author = {Kifer, Daniel and Smith, Adam and Thakurta, Abhradeep},
  editor = {Mannor, Shie and Srebro, Nathan and Williamson, Robert C.},
  year = {2012},
  month = jun,
  series = {Proceedings of {{Machine Learning Research}}},
  volume = {23},
  pages = {25.1--25.40},
  publisher = {PMLR},
  address = {Edinburgh, Scotland},
  urlll = {https://proceedings.mlr.press/v23/kifer12.html}
}

@book{kress_linear_2014,
  title = {Linear {{Integral Equations}}},
  author = {Kress, Rainer},
  year = {2014},
  series = {Applied Mathematical Sciences},
  edition = {3rd},
  number = {volume 82},
  publisher = {Springer},
  address = {New York},
  isbn = {978-1-4614-9592-5},
  lccn = {QA431 .K776 2014}
}

@inproceedings{lei_differentially_2011,
  title = {Differentially {{Private M-Estimators}}},
  booktitle = {{{NIPS}}},
  author = {Lei, Jing},
  year = {2011}
}

@article{loh_high-dimensional_2012,
  title = {High-{{Dimensional Regression}} {{With}} {{Noisy}} and {{Missing Data}}: {{Provable Guarantees}} {{With}} {{Nonconvexity}}},
  shorttitle = {High-Dimensional Regression with Noisy and Missing Data},
  author = {Loh, Po-Ling and Wainwright, Martin J.},
  year = {2012},
  month = jun,
  journal = {The Annals of Statistics},
  volume = {40},
  number = {3},
  eprint = {1109.3714},
  primaryclass = {cs, math, stat},
  issn = {0090-5364},
  doi = {10.1214/12-AOS1018},
  urlll = {http://arxiv.org/abs/1109.3714},
  urldate = {2023-05-23},
  archiveprefix = {arXiv}
}

@inproceedings{mangold_high-dimensional_2023,
  title = {{High}-{Dimensional} {Private} {Empirical} {Risk} {Minimization} by {Greedy} {Coordinate} {Descent}},
  booktitle = {Proceedings of The 26th International Conference on Artificial Intelligence and Statistics},
  author = {Mangold, Paul and Bellet, Aur\'elien and Salmon, Joseph and Tommasi, Marc},
  editor = {Ruiz, Francisco and Dy, Jennifer and van de Meent, Jan-Willem},
  year = {2023},
  series = {Proceedings of Machine Learning Research},
  volume = {206},
  pages = {4894--4916},
  publisher = {PMLR},
  urlll = {https://proceedings.mlr.press/v206/mangold23a.html}
}

@incollection{newey_chapter_1994,
  title = {Chapter 36: {{Large Sample Estimation}} and {{Hypothesis Testing}}},
  booktitle = {Handbook of {{Econometrics}}},
  author = {Newey, Whitney K. and McFadden, Daniel},
  year = {1994},
  volume = {4},
  pages = {2111--2245},
  publisher = {Elsevier},
  doi = {10.1016/S1573-4412(05)80005-4},
  urlll = {https://linkinghub.elsevier.com/retrieve/pii/S1573441205800054},
  urldate = {2024-09-06},
  isbn = {978-0-444-88766-5},
  langid = {english}
}

@phdthesis{neyman_jerzy_applications_1924,
  title = {On the {{Applications}} of the {{Theory}} of {{Probability}} to {{Agricultural Experiments}}},
  author = {Neyman, Jerzy},
  year = {1924},
  school = {University of Warsaw}
}

@misc{ohnishi_locally_2023,
  title = {Locally {{Private Causal Inference}} for {{Randomized Experiments}}},
  author = {Ohnishi, Yuki and Awan, Jordan},
  year = {2023},
  publisher = {arXiv},
  doi = {10.48550/ARXIV.2301.01616},
  urlll = {https://arxiv.org/abs/2301.01616},
  urldate = {2024-09-12},
  copyright = {arXiv.org perpetual, non-exclusive license}
}

@book{piziak_matrix_2007,
  title = {Matrix {{Theory}}},
  author = {Piziak, Robert and Odell, Patrick L.},
  year = {2007},
  month = feb,
  edition = {1st},
  publisher = {{Chapman and Hall/CRC}},
  doi = {10.1201/9781420009934},
  urlll = {https://www.taylorfrancis.com/books/9781420009934},
  urldate = {2024-08-25},
  isbn = {978-1-4200-0993-4},
  langid = {english}
}

@book{polanin_handbook_1998,
  title = {Handbook of {{Integral Equations}}},
  author = {Pol{\^a}nin, Andrej Dmitrievi{\v c} and Manzhirov, Aleksandr Vladimirovich},
  year = {1998},
  edition = {1st},
  publisher = {CRC press},
  address = {Boca Raton, London, New York, Washington},
  isbn = {978-0-8493-2876-3},
  langid = {english},
  lccn = {515.45}
}

@book{reed_methods_1972,
  title = {Methods of {{Modern Mathematical Physics}}},
  author = {Reed, Michael and Simon, Barry},
  year = {1972},
  publisher = {Academic Press},
  address = {New York},
  isbn = {978-0-12-585001-8},
  langid = {english},
  lccn = {530.15}
}

@article{rotnitzky_characterization_2021,
  title = {Characterization of {{Parameters}} {{With}} a {{Mixed Bias Property}}},
  author = {Rotnitzky, Andrea and Smucler, Ezequiel and Robins, James M.},
  year = {2021},
  month = mar,
  journal = {Biometrika},
  volume = {108},
  number = {1},
  pages = {231--238},
  issn = {0006-3444, 1464-3510},
  doi = {10.1093/biomet/asaa054},
  urlll = {https://academic.oup.com/biomet/article/108/1/231/5899828},
  urldate = {2024-08-10},
  copyright = {https://academic.oup.com/journals/pages/open\_access/funder\_policies/chorus/standard\_publication\_model},
  langid = {english}
}

@article{rubin_estimating_1974,
  title = {Estimating {{Causal Effects}} of {{Treatments}} in {{Randomized}} and {{Nonrandomized Studies}}},
  author = {Rubin, Donald B.},
  year = {1974},
  journal = {Journal of Educational Psychology},
  volume = {66},
  number = {5},
  pages = {688--701},
  issn = {0022-0663},
  doi = {10.1037/h0037350},
  urlll = {http://content.apa.org/journals/edu/66/5/688},
  urldate = {2020-10-05},
  langid = {english}
}

@inproceedings{sheffet_differentially_2017,
  title = {Differentially {{Private Ordinary Least Squares}}},
  booktitle = {Proceedings of the 34th {{International Conference}} on {{Machine Learning}}},
  author = {Sheffet, Or},
  editor = {Precup, Doina and Teh, Yee Whye},
  year = {2017},
  month = aug,
  series = {Proceedings of {{Machine Learning Research}}},
  volume = {70},
  pages = {3105--3114},
  publisher = {PMLR},
  urlll = {https://proceedings.mlr.press/v70/sheffet17a.html}
}

@misc{slavkovic_perturbed_2021,
  title = {Perturbed {M}-{Estimation}: {A} {Further} {Investigation} of {Robust} {Statistics} for {Differential} {Privacy}},
  author = {Slavkovic, Aleksandra B. and Molinari, Roberto},
  year = {2021},
  urlll = {https://api.semanticscholar.org/CorpusID:237194827}
}

@misc{smith_efficient_2008,
  title = {Efficient, {{Differentially Private Point Estimators}}},
  author = {Smith, Adam},
  year = {2008},
  month = sep,
  number = {arXiv:0809.4794},
  eprint = {0809.4794},
  primaryclass = {cs},
  publisher = {arXiv},
  urlll = {http://arxiv.org/abs/0809.4794},
  urldate = {2023-05-17},
  archiveprefix = {arXiv}
}

@misc{steinberger_efficiency_2023,
  title = {Efficiency in {{Local Differential Privacy}}},
  author = {Steinberger, Lukas},
  year = {2023},
  month = jan,
  number = {arXiv:2301.10600},
  eprint = {2301.10600},
  primaryclass = {math, stat},
  publisher = {arXiv},
  urlll = {http://arxiv.org/abs/2301.10600},
  urldate = {2023-03-27},
  archiveprefix = {arXiv}
}

@book{van_der_vaart_asymptotic_1998,
  title = {Asymptotic {{Statistics}}},
  author = {{van der Vaart}, Aad W.},
  year = {1998},
  series = {Cambridge {{Series}} on {{Statistical}} and {{Probabilistic Mathematics}}},
  publisher = {Cambridge University Press},
  isbn = {1-107-26372-7},
  langid = {english},
  lccn = {98015176}
}

@article{warner_randomized_1965,
  title = {Randomized {{Response}}: {{A Survey Technique}} for {{Eliminating Evasive Answer Bias}}},
  shorttitle = {Randomized {{Response}}},
  author = {Warner, Stanley L.},
  year = {1965},
  month = mar,
  journal = {Journal of the American Statistical Association},
  volume = {60},
  number = {309},
  pages = {63--69},
  issn = {0162-1459, 1537-274X},
  doi = {10.1080/01621459.1965.10480775},
  urlll = {http://www.tandfonline.com/doi/abs/10.1080/01621459.1965.10480775},
  urldate = {2024-08-26},
  langid = {english}
}

@incollection{yamamoto_differentially_2017,
  title = {Differentially {{Private Empirical Risk Minimization}} {{With}} {{Input Perturbation}}},
  booktitle = {Discovery {{Science}}},
  author = {Fukuchi, Kazuto and Tran, Quang Khai and Sakuma, Jun},
  editor = {Yamamoto, Akihiro and Kida, Takuya and Uno, Takeaki and Kuboyama, Tetsuji},
  year = {2017},
  volume = {10558},
  pages = {82--90},
  publisher = {Springer International Publishing},
  address = {Cham},
  doi = {10.1007/978-3-319-67786-6_6},
  urlll = {http://link.springer.com/10.1007/978-3-319-67786-6_6},
  urldate = {2024-09-12},
  isbn = {978-3-319-67785-9 978-3-319-67786-6},
  langid = {english}
}

@article{yang_comprehensive_2019,
  title = {A {{Comprehensive Survey}} on {{Secure Outsourced Computation}} and {{Its Applications}}},
  author = {Yang, Yang and Huang, Xindi and Liu, Ximeng and Cheng, Hongju and Weng, Jian and Luo, Xiangyang and Chang, Victor},
  year = {2019},
  month = oct,
  journal = {IEEE Access},
  volume = {7},
  pages = {159426--159465},
  issn = {2169-3536},
  doi = {10.1109/ACCESS.2019.2949782}
}

@misc{zhu_causal_2022,
  title = {Causal {{Inference}} {{With}} {{Treatment Measurement Error}}: {{A Nonparametric Instrumental Variable Approach}}},
  shorttitle = {Causal {{Inference}} with {{Treatment Measurement Error}}},
  author = {Zhu, Yuchen and Gultchin, Limor and Gretton, Arthur and Kusner, Matt and Silva, Ricardo},
  year = {2022},
  month = jun,
  number = {arXiv:2206.09186},
  eprint = {2206.09186},
  primaryclass = {cs, stat},
  publisher = {arXiv},
  urlll = {http://arxiv.org/abs/2206.09186},
  urldate = {2023-06-23},
  archiveprefix = {arXiv}
}

@article{van_der_vaart_higher_2014,
  title = {Higher {{Order Tangent Spaces}} and {{Influence Functions}}},
  author = {Van Der Vaart, Aad},
  date = {2014-11-01},
  year = {2014},
  journal = {Statistical Science},
  shortjournal = {Statist. Sci.},
  volume = {29},
  number = {4},
  issn = {0883-4237},
  doi = {10.1214/14-STS478},
  urlll = {https://projecteuclid.org/journals/statistical-science/volume-29/issue-4/Higher-Order-Tangent-Spaces-and-Influence-Functions/10.1214/14-STS478.full},
  urldate = {2026-03-31}
}

@inbook{laber_semiparametric_2024,
  title = {Semiparametric {{Doubly Robust Targeted Double Machine Learning}}: {{A Review}}},
  shorttitle = {Semiparametric {{Doubly Robust Targeted Double Machine Learning}}},
  booktitle = {Handbook of {{Statistical Methods}} for {{Precision Medicine}}},
  author = {Kennedy, Edward H.},
  date = {2024-10-07},
  year = {2024},
  edition = {1},
  pages = {207--236},
  publisher = {{Chapman and Hall/CRC}},
  location = {Boca Raton},
  doi = {10.1201/9781003216223-10},
  url = {https://www.taylorfrancis.com/books/9781003216223/chapters/10.1201/9781003216223-10},
  urldate = {2026-03-31},
  bookauthor = {Laber, Eric and Chakraborty, Bibhas and Moodie, Erica E. M. and Cai, Tianxi and Laan, Mark Van Der},
  isbn = {978-1-003-21622-3},
  langid = {english}
}
\bibliographystyle{plainnat}
}




\newpage

\begin{appendices}


\section{Assumptions}
\label{priv:app:sec:assumptions}

Results in \Cref{priv:sec:estim_dr_priv,priv:sec:estim_nuisance_priv} rely on the following assumptions, respectively.

\begin{assumption}[Consistent Private Estimators]
\label{priv:ass:consistent_nuisance_priv}
Either the privacy mechanism $Q\in\setQgen$ in \eqref{priv:eq:qtv}; or $Q\in\setQdiscinv$ and $(V,X)$ is distributed on a finite set with density $p_{VX}$ with respect to the counting measure. Let $\iota\ceq 1$ in the former and $\iota\ceq 0$ in the latter case. Define the norm $\norm{\linopQXinvSpa}{\cdot}\ceq\norm{L_2(P_{VX})}{\cdot}+\iota\norm{L_2(P_V\otimes \Qbar)}{\cdot}$ and the measure $\linopQXinvMes\ceq P_{VX}+\iota P_V \otimes \Qbar$, and let $\linopQXinvSpa\ceq \set{f: \Vimagesp\times\Ximagesp \to \real: \norm{\linopQXinvSpa}{f}<\infty}$. It holds that $\tld\chi,\chk{\tld\chi},\rieszX\in \linopQXinvSpa$ and
\begin{align}
\norm{\linopQXinvSpa}{f(\cdot,\mu,\gamma)-f(\cdot,\muX,\gammaV(c))}\to 0&\, \text{ as }\, \rho((\mu,\gamma),(\muX,\gammaV(c)))\to 0,\label{priv:eq:f_sq_continu_priv}  \\
\norm{\linopQXinvSpa}{\partial_\gamma f(\cdot,\mu,\gamma)-\partial_\gamma f(\cdot,\muX,\gammaV(c))}\to 0&\, \text{ as }\, \rho((\mu,\gamma),(\muX,\gammaV(c)))\to 0,\label{priv:eq:f_diffcontinu_priv}
\end{align}
and
\begin{align}
\norm{\linopQXinvSpa}{\rieszXchk-\rieszX}&=\smallOPs{P_{VZ}}{1}, \label{priv:eq:consistency_riesz_l2_vz} \\
\gammaVchk(c)-\gammaV(c)&=\smallOPs{P_{VZ}}{1}, \label{priv:eq:consistency_gammaVhat_vz} \\
\chk p_{V_2}(c)-p_{V_2}(c)&=\smallOPs{P_{VZ}}{1}. \label{priv:eq:consistency_phat_vz}
\end{align}
Further, it either holds that
\begin{align}
\supnorm{m-\muX}&=\bigO{1}, \label{priv:eq:consistency_bigobound_vz} \\
\supnorm{\muX-\muXchk}&=\smallOPs{P_{VZ}}{1},  \label{priv:eq:consistency_mu_supn_vz}
\end{align}
or that
\begin{align}
\supnorm{m-\muXchk}&=\bigOPs{P_{VZ}}{1}, \label{priv:eq:consistency_bigobound_hat_vz} \\
\norm{\linopQXinvSpa}{\muX-\muXchk}&=\smallOPs{P_{VZ}}{1},  \label{priv:eq:consistency_mu_l2_vz} \\
\linopQXinvMes\left(\set{(V_{\mhit{1}},X)\in \Vimagesp_{\mhit{1}}\times\Ximagesp: |r(V_{\mhit{1}},X)|>\ba R }\right)&=0  \label{priv:eq:consistency_riesz_bound_vz}
\end{align}
for some constant $\ba R<\infty$. We may replace $\norm{\linopQXinvSpa}{\cdot}$ with $\supnorm{.}$ and \eqref{priv:eq:consistency_riesz_bound_vz} with $\supnorm{\rieszX}<\infty$.
\end{assumption}

\begin{assumption}[Private Method-of-Moments]
\label{priv:ass:mm_priv}
The set $\Theta\subset\real^K$ is compact, and $\theta_0\in\Int\Theta$ is the unique minimiser of $\theta\mapsto P_{VX}\Xi_\theta$ for $\Xi_{\theta}:\Vimagesp\times\Ximagesp\to\real$, $\theta\in\Theta$. The derivative $\phi_{\tld\theta}(v,x)\ceq\Deriv_\theta \Xi_{\tld\theta}(v,x)^\intercal$ as a map to $\real^{K\times 1}$ exists at all $(v,x,\tld\theta)\in\Vimagesp\times\Ximagesp\times\neighTheta$, for a neighbourhood $\neighTheta$ of $\theta_0$, and satisfy $\norm{L_1(P_{VZ})}{\norm{2}{\phi_{\theta_0}}^2}<\infty$, where, for a fixed $(v,x)\in\Vimagesp\times\Ximagesp$, $\norm{2}{\phi_{\theta_0}}^2(v,x)$ is sum of the $K$ squared entries of $\phi_{\theta_0}(v,x)$. The derivative $\dot\phi_{\tld\theta}(v,x)\ceq \Deriv_\theta \phi_{\tld\theta}(v,x)$ as a map to $\real^{K\times K}$ exists at all $(v,x,\tld\theta)\in\Vimagesp\times\Ximagesp\times\neighTheta$, and $\theta\mapsto \dot\phi_{\theta}(v,x)$ is continuous at all $(v,x,\theta)\in\Vimagesp\times\Ximagesp\times\neighTheta$, with the expectation $P_{VX}\dot\phi_{\theta_0}$ existent and invertible as a matrix. Further, $\norm{L_1(P_{VZ})}{\sup_{\theta\in\neighTheta}\norm{1}{\dot \phi_\theta}}<\infty$, where, for a fixed $(v,x)\in\Vimagesp\times\Ximagesp$, $\norm{1}{\dot \phi_\theta}(v,x)$ is the sum of the absolute values of the $K^2$ entries of $\dot\phi_\theta(v,x)$. The $\xi_\theta$ are such that the derivative $\Deriv_\theta\xi_{\tld\theta}(v,x)$ exists at all $\tld\theta\in\neighTheta$ for all $(v,x)\in\Vimagesp\times\Ximagesp$, and $\norm{L_2(P_{VX})}{\sup_{\tld\theta\in\neighTheta}\norm{2}{\Deriv_\theta \xi_{\tld\theta}}}=\bigOPs{P_{VZ}}{1}$.
\end{assumption}



\section{Rate-Double-Robust Inference}
\label{priv:app:sec:dr}

\Cref{priv:app:sec:dr:subsec:examples} contains examples of parameters in our rate-double-robust class. \Cref{priv:app:sec:dr:subsec:proofs_main} proves the main results in \Cref{priv:sec:dr}. 

\subsection{Examples}
\label{priv:app:sec:dr:subsec:examples}

In this section, we present examples of parameters $\chi(P_{VX})=\E f(V,X, \muX, \gammaV(c))$ defined in \Cref{priv:sec:dr:subsec:class}. In \Cref{priv:ex:ate_dr,priv:ex:att_dr} we discuss causal estimands; in \Cref{priv:ex:deriv,priv:ex:ray_int,priv:ex:line_int},  ``geometric'' parameters. In  \Cref{priv:ex:known_riesz}, we study the special case of a known Riesz representer illustrated by \Cref{priv:ex:discounting} with a simple model from economics. Finally, in \Cref{priv:ex:extension}, we consider the extension of our class to feature dependence on multiple regressions.

In \Cref{priv:ex:ate_dr,priv:ex:att_dr}, we adopt the potential outcome framework of \cite{neyman_jerzy_applications_1924} and \cite{rubin_estimating_1974}. Specifically, we consider a binary treatment $D\in\dummy$ and an observed outcome $Y=DY^1+(1-D)Y^0$ for partially unobserved potential outcomes $Y^0,Y^1$ with values in $\real$. Let $X$ be covariates taking values in a measurable space $(\Ximagesp,\Xsigma)$ with law $P_X$ and satisfying unconfoundedness $Y^d\indep D\mid X$ for $d\in\dummy$. Let
\begin{align}
\label{priv:eq:causal_definitions}
\begin{aligned}
\muX(d,x)&\ceq\Ebc{Y}{D=d,X=x} \\
\piX(d\cond x)&\ceq \Ebc{\indic{D=d}}{X=x}
\end{aligned}
\end{align}
for $(d,x)\in\dummy\times\Ximagesp$. We assume that $\piX(d\cond X)\geq \epsilon$ for all $d\in\dummy$ $P_{VX}$-a.s., for some $\epsilon>0$. In \Cref{priv:ex:ate_dr,priv:ex:att_dr}, \Cref{priv:eqprop:eif} recovers the familiar efficient influence functions for average treatment effects, under the nonparametric model with unknown propensity score. Moreover, the bias formulae in \Cref{priv:ex:att_dr} show the average treatment effect on the treated rate-double-robust. This latter result is aligned with \citet[Example 6]{chernozhukov_automatic_2022}, and is an improvement on \citet[Example 12]{rotnitzky_characterization_2021}, who too, establish asymptotic normality, but not efficiency, as this parameter is not natively included in their class.

\begin{example}[Average Treatment Effect]
\label{priv:ex:ate_dr}
In the causal context \eqref{priv:eq:causal_definitions}, let $V\ceq (Y,D)$ and $V_{\mhit{1}}\ceq D$, $m(V,X)\ceq Y$; $V_{\mhit{2}}$, $g$, and $\gammaV$ are not used. Then $\E Y^d=\E \muX(d,X)$ for $d\in\dummy$. Hence, $f(V,X,\muX,\gammaV(c))\ceq \muX(d,X)$ gives $\chi(P_{VX})= \E Y^d$, while $f(V,X,\muX,\gammaV(c))\ceq \muX(1,X)-\muX(0,X)$ gives $\chi(P_{VX})= \E Y^1-\E Y^0$, with \eqref{priv:eq:f_continu}, \eqref{priv:eq:f_lin}, \eqref{priv:eq:f_diff}, \eqref{priv:eq:sq_continu} holding.

The Riesz representer for $\E Y^d$ is 
$\rieszX(d',x)=\fr{\indic{d'=d}}{\piX(d \cond x)}$ 
and for $\E Y^1-\E Y^0$, it is 
$\rieszX(d,x)=\fr{d}{\piX(1\cond x)}-\fr{1-d}{1-\piX(1\cond x)}.$
None of these representers $r$ depends on $\gammaV(c)$.

The efficient influence function \eqref{priv:eq:eif_chi}, in the nonparametric model (with unknown $\piX$), of $\E Y^d$, with $\partial_\gamma f = 0$, is
\begin{align*}
\tld\chi(\tld v,\tld x) =\fr{\indic{\tld d=d}}{\piX(d\mid \tld x)}(\tld y-\muX(d,\tld x))+\muX(d,\tld x) - \chi(P_{VX});
\end{align*}
of $\chi(P_{VX})=\E Y^1 - \E Y^0$, with $\partial_\gamma f = 0$, is
\begin{align*}
\tld\chi(v,x)=&\, \fr{d}{\piX(1\cond x)}(y-\muX(1,x)) -\fr{1-d}{1-\piX(1\cond x)}(y-\muX(0,x)) \\
&+\muX(1,X)-\muX(0,X) - \chi(P_{VX}), 
\end{align*}
as in \citet[Proof of Theorem 1]{hahn_role_1998}, since $\indic{\tld d=d}h(\tld d,\tld x)=\indic{\tld d=d}h(d,\tld x)$ for any $h$. Note that while the model for $(Y^0,Y^1,D,$ $X)$ is \emph{not} nonparametric, because it is constrained by $Y^d\indep D\cond X$, the model for $(Y,D,X)$ is nonparametric if the models for $Y^0\cond X$, $Y^1\cond X$, $D\cond X$ and $X$ are all nonparametric.

Consequently, the bias \eqref{priv:eq:chihat_biasdecomp} established by \Cref{priv:thm:dr} is
\begin{align*}
R_n \ceq&\ -P_{VX}(\rieszX-\rieszXhat)(\muX-\muXhat) \\
=&\, P_{DX}\left[\indic{D=d}\fr{\piXhat(d\cond X)-\piX(d\cond X)}{\piXhat(d\cond X)\piX(d\cond X)}\big(\muXhat(D,X)-\muX(D,X)\big)\right],
\end{align*}
for $\E Y^d$ where $\muXhat,\rieszXhat$ are some estimators of its Riesz representer and regression, respectively. For $\E Y^1-\E Y^0$, the same quantity is
\begin{align*}
R_n \ceq -P_{VX}(\rieszX-\rieszXhat)(\muX-\muXhat) \\
= P_{DX}\left[D\fr{\piXhat(1\cond X)-\piX(1\cond X)}{\piXhat(1\cond X)\piX(1\cond X)}\big(\muXhat(D,X)-\muX(D,X)\big)\right] \\
+P_{DX}\left[(1-D)\fr{\piXhat(1\cond X)-\piX(1\cond X)}{(1-\piXhat(1\cond X))(1-\piX(1\cond X))}\big(\muXhat(D,X)-\muX(D,X)\big)\right]
\end{align*}
for the corresponding regression and representer (estimators). By the tower property of expectation conditioning on $X$ and by the definition of $\piX$, we arrive to the usual bias formulae
\begin{align*}
R_n =&\, P_{X}\left[\fr{\piXhat(d\cond X)-\piX(d\cond X)}{\piXhat(d\cond X)}\big(\muXhat(d,X)-\muX(d,X)\big)\right] & \text{ for } \E Y^d, \\
R_n =&\, P_{X}\left[\fr{\piXhat(1\cond X)-\piX(1\cond X)}{\piXhat(1\cond X)}\big(\muXhat(1,X)-\muX(1,X)\big)\right] & \\
&+P_{X}\left[\fr{\piXhat(1\cond X)-\piX(1\cond X)}{1-\piXhat(1\cond X)}\big(\muXhat(0,X)-\muX(0,X)\big)\right] & \text{ for } \E Y^1-\E Y^0.
\end{align*}

\end{example}

\noindent If we did not allow for the dependence on the parameter $\gammaV$, our class of parameters \eqref{priv:eq:dr_para} and conditions \eqref{priv:eq:f_continu}, \eqref{priv:eq:f_lin} would be a strict subset of those of \cite{rotnitzky_characterization_2021}. Allowing for such a nonlinear but smooth dependence enables us to capture more parameters, such as the average treatment effect on the treated, which was indeed shown double robust by \citet[Example 6]{chernozhukov_automatic_2022}. 

\begin{example}[Average Treatment Effect on the Treated]
\label{priv:ex:att_dr}
Consider the setting of \Cref{priv:ex:ate_dr}, but now take $g(V,X)\ceq D$, and $\gammaV(c)\ceq p_1$ for $p_1\ceq \E D$;  $V_\mhit{2}$ is again unused and taken to be empty. Then $\Ebc{Y^0}{D=1}=\E D\muX(0,X)/p_1$. Hence $f(V,X,\muX,\gammaV)\ceq D\muX(0,X)/p_1$ gives $\chi(P_{VX})= \Ebc{Y^0}{D=1}$, while $f(V,X,\muX,\gammaV)\ceq D(\muX(1,X)-\muX(0,X))/p_1$ gives $\chi(P_{VX})= \Ebc{Y^1-Y^0}{D=1}$,
with \eqref{priv:eq:f_continu}, \eqref{priv:eq:f_lin}, \eqref{priv:eq:f_diff}, \eqref{priv:eq:sq_continu} holding.

The Riesz representer for $\Ebc{Y^0}{D=1}$ is 
$\rieszX(d,x)=\fr{1-d}{p_1}\fr{\piX(1\cond x)}{1-\piX(1\cond x)},$ 
and for $\Ebc{Y^1-Y^0}{D=1}$, it is 
$\rieszX(d,x)=\fr{d}{p_1}-\fr{1-d}{p_1}\fr{\piX(1\cond x)}{1-\piX(1\cond x)}.$
Both representers $\rieszX$ depend on $\gammaV(c)=p_1$.

The efficient influence function \eqref{priv:eq:eif_chi}, in the nonparametric model (with unknown $\piX$), of $\Ebc{Y^0}{D=1}$, with $\partial_\gamma f(V,X,\muX,\gammaV(c))=-D\muX(0,X)/p_1^2$, is
\begin{align*}
\tld\chi(v,x)=&\, \fr{1-d}{p_1}\fr{\piX(1\cond x)}{1-\piX(1\cond x)}(y-\muX(d,x))-\fr{d-p_1}{p_1}\chi(P_{VX})+D\muX(0,X)/p_1 \\
&- \chi(P_{VX}) \\
=& \fr{1-d}{p_1}\fr{\piX(1\cond x)}{1-\piX(1\cond x)}(y-\muX(0,x))+\fr{d}{p_1}\muX(0,X) -\fr{d}{p_1}\chi(P_{VX});
\end{align*}
of $\Ebc{Y^1-Y^0}{D=1}$, with $\partial_\gamma f(V,X,\muX,\gammaV(c))=-D(\muX(1,X)-\muX(0,X))/p_1^2$, is
\begin{align*}
\tld\chi(v,x)=&\,\left( \fr{d}{p_1}-\fr{1-d}{p_1}\fr{\piX(1\cond x)}{1-\piX(1\cond x)}\right) (y-\muX(d,x))-\fr{d-p_1}{p_1}\chi(P_{VX}) \\
&+ \fr{d}{p_1}(\muX(1,x)-\muX(0,x))-\chi(P_{VX}) \\
=&\,\fr{d}{p_1}(y-\muX(1,x))-\fr{1-d}{p_1}\fr{\piX(1\cond x)}{1-\piX(1\cond x)}(y-\muX(0,x)) \\
&+ \fr{d}{p_1}(\muX(1,x)-\muX(0,x)) + \fr{d}{p_1}\chi(P_{VX})
\end{align*}
as in \citet[Proof of Theorem 1]{hahn_role_1998}. The remark on the nonparametric nature of the $(Y,D,X)$-model in \Cref{priv:ex:ate_dr} applies here equally.

Consequently, the first term of the bias in \eqref{priv:eq:bias_r} established by \Cref{priv:thm:dr} for $\Ebc{Y^0}{D=1}$ is
\begin{align*}
-P_{VX}(\rieszX-\rieszXhat)(\muX-\muXhat) \\
= \fr{\hat p_1- p_1}{\hat p_1 p_1}P_X\left[\fr{(1-\piXhat(1\cond X))\piX(1\cond X)}{1-\piXhat(1\cond X)}(\muXhat(0,X)-\muX(0,X))\right] \\
+ \fr{p_1}{p_1\hat p_1}P_{X}\left[\fr{(\piX(1\cond X)-\piXhat(1\cond X))(\muXhat(0,X)-\muX(0,X))}{1-\piXhat(1\cond X)}\right],
\end{align*}
by the tower property of expectation conditioning on $X$; for $\Ebc{Y^1-Y^0}{D=1}$, it is
\begin{align*}
-P_{VX}(\rieszX-\rieszXhat)(\muX-\muXhat) = \fr{\hat p_1 - p_1}{\hat p_1 p_1}P_X\left[\piX(1\cond X)(\muXhat(1,X)-\muX(1,X))\right] \\
-\fr{\hat p_1- p_1}{\hat p_1 p_1}P_X\left[\fr{(1-\piXhat(1\cond X))\piX(1\cond X)}{1-\piXhat(1\cond X)}(\muXhat(0,X)-\muX(0,X))\right] \\
- \fr{p_1}{p_1\hat p_1}P_{X}\left[\fr{(\piX(1\cond X)-\piXhat(1\cond X))(\muXhat(0,X)-\muX(0,X))}{1-\piXhat(1\cond X)}\right].
\end{align*}
Suppose $\expderivhat-\expderiv'=\smallOPs{P_{VX}}{1}$ and that $P_{YDX} \partial_\gamma^2 f(Y,D,X,\muXhat,\tld p_1)=\bigOPs{P_{VX}}{1}$. Then for the bias to vanish as $\smallOPs{P_{VX}}{n^{-1/2}}$, it suffices that $\norm{L_1(P_{VX})}{\Delta}=\smallOPs{P_{VX}}{n^{-1/2}}$, where $\Delta(x)\ceq (\muXhat(0\cond x)-\muX(0\cond x))(\piX(1\cond x)-\piXhat(1\cond x))$, and $\muXhat$ be consistent because $\hat p_1-p_1=\bigOPs{P_{VX}}{n^{-1/2}}$, provided $1-\piXhat(1\cond x)$ is bounded away from zero.

\end{example}

Parameters with a more geometric interpretation are also included in our class.

\begin{example}[Average Approximate Derivative]
\label{priv:ex:deriv}
Let $\Ximagesp=\real$,  and $V\ceq Y$ for a random variable $Y$. Set $m(V,X)\ceq Y$, and assume that $\muX(v_\mhit{1},x)=\Ebc{Y}{V_\mhit{1}=v_\mhit{1},X=x}$ does not depend on $v_\mhit{1}$, so we use $\muX(x)$ to refer to its value. Assume that the marginal distribution $P_X$ of $X$ admits a Lebesque density $p_X$.

Fix a strictly positive constant $\epsilon\in\real$. Setting $f(v,x,\mu,\gamma)\ceq \fr{\mu(x+\epsilon)-\mu(x)}{\epsilon}$ and $J_{P_{VX},\epsilon}(\mu)\ceq \Eb{\fr{\mu(X+\epsilon)-\mu(X)}{\epsilon}}$, $\mu\in L_2(P_X)$, yields the parameter
\begin{align*}
\chi(P_{VX}) = J_{P_{VX},\epsilon}(\muX) = \Eb{\fr{\muX(X+\epsilon)-\muX(X)}{\epsilon}}=\int_{-\infty}^\infty \fr{\muX(x+\epsilon)-\muX(x)}{\epsilon}p_X(x)\deriv x.
\end{align*} 
For small $\epsilon$, the integrand represents an approximate derivative of $\muX$, although we do not actually require that $\muX$ be differentiable, unlike \citet[Example 2: Weighted Average Derivative]{chernozhukov_automatic_2022}.

Assume that $(P_{VX},\epsilon)$ satisfies
\begin{align}
b(P_{VX},\epsilon)\ceq \int_{-\infty}^\infty \left(\fr{p_X(x-\epsilon)}{p_X(x)}\right)^2 p_X(x)\deriv x<\infty. \label{priv:eq:b_bound}
\end{align}
Then $J_{P_{VX},\epsilon}$ has Riesz representer $\rieszX(x)=\fr{p_X(x-\epsilon)-p_X(x)}{\epsilon p_X(x)}$, whose dependence on $(P_{VX},\epsilon)$ is suppressed in notation. The last display ensures that $J_{P_{VX},\epsilon}$ is continuous so that the Riesz representation theorem applies. Indeed, by a change of variables, and the Cauchy--Schwarz inequality, for $\mu,\tld\mu\in L_2(P_X)$,
\begin{align*}
|J_{P_{VX},\epsilon}(\mu)-J_{P_{VX},\epsilon}(\tld\mu)|\leq \Eb{|\mu(X)-\tld\mu(X)|\fr{p_X(X-\epsilon)}{p_X(X)}}+\norm{L_2(P_X)}{\mu-\tld\mu}  \\
\leq \left(\sqrt{b(P_{VX},\epsilon)}+1\right)\norm{L_2(P_X)}{\mu-\tld\mu}.
\end{align*}
Hence, \eqref{priv:eq:f_continu}, \eqref{priv:eq:f_lin}, and \eqref{priv:eq:f_diff} hold. In addition, if $(P_{VX},\epsilon)$ satisfies $\supnorm{\fr{p_X(\cdot-\epsilon)}{p_X(\cdot)}}<\infty$ --- which is stronger than \eqref{priv:eq:b_bound} --- , then \eqref{priv:eq:sq_continu} holds by similar arguments.

\end{example}

\begin{example}[Ray-Type Integral]
\label{priv:ex:ray_int}
Let $\Ximagesp=\real^d$, and $V\ceq (Y,W)$, for a random variable $Y$, and a random vector $W$ with values $\real^d$ and Lebesgue density $p_W$.  Set $m(V,X)\ceq Y$, and assume that $\muX(v_\mhit{1},x)\equiv\muX(x)$ does not depend on $v_\mhit{1}$. Suppose that $P_X$ has a Lebesgue density $p_X$. 

Fix constants $0<t_0<t_1$. Setting $f(v,x,\mu,\gamma)\ceq \int_{t_{0}}^{t_{1}} \mu(tw)\norm{2}{w}\dderiv t$ and $J_{P_{VX},t_{0},t_{1}}(\mu)\ceq \E \int_{t_0}^{t_1} \mu(tW)\norm{2}{W}\dderiv t$, $\mu\in L_2(P_X)$, where $\norm{2}{w}=\sqrt{\sum_{j=1}^d w_j^2}$ is the Euclidean norm on $\real^d$, yields the parameter
\begin{align*}
\chi(P_{VX})\ceq J_{P_{VX},t_0,t_1}(\muX)= \E \int_{t_0}^{t_1} \muX(tW)\norm{2}{W}\dderiv t=\int_{\real^d}\int_{t_0}^{t_1} \muX(tw)\norm{2}{w}\dderiv t p_W(w)\dderiv w.
\end{align*}
The quantity $f(v,x,\mu,\gamma)$ is the integral of $\mu$ along the curve $\zeta_{w}: [t_0,t_1]\to\real^d$, $\zeta_{w}(t)\ceq t w$, the line segment between $t_0w$ and $t_1w$, and $J_{P_{VX},t_0,t_1}(\mu)$ is its average across draws from $p_W$. With $t_0$ close to zero and $t_1$ to infinity, $f(v,x,\mu,\gamma)$ approximates the ray emanating from the origin in the direction of $w$, but care must be taken to ensure integrability: if $(P_{VX},t_0,t_1)$ is such that
\begin{align*}
\real^d \ni x\mapsto \fr{\norm{2}{x}}{p_X(x)}\int_{t_0}^{t_1}p_W\left(\fr{x}{t}\right)\fr{1}{t^{d+1}}\dderiv t \in L_2(P_X),
\end{align*}
then the map in the display is the Riesz representer of $J_{P_{VX},t_0,t_1}$, with \eqref{priv:eq:f_continu} and \eqref{priv:eq:f_lin} holding. (Condition \eqref{priv:eq:f_diff} is trivially satisfied.) This follows from
\begin{align*}
J_{P_{VX},t_0,t_1}(\mu)=\int_{\real^d}\int_{t_0}^{t_1} \mu(tw)\norm{2}{w}\dderiv t\, p_W(w)\dderiv w = \int_{t_0}^{t_1}\int_{\real^d}\mu(u)\norm{2}{u/t}p_W(u/t)\fr{1}{t^d}\dderiv u \dderiv t \\
= \int_{\real^d}\mu(u)\fr{\norm{2}{u}}{p_X(u)}\int_{t_0}^{t_1}p_W(u/t)\fr{1}{t^{d+1}}\dderiv t  p_X(u) \dderiv u = \E \mu(X)\fr{\norm{2}{X}}{p_X(X)}\int_{t_0}^{t_1}p_W(X/t)\fr{1}{t^{d+1}}\dderiv t,
\end{align*}
where the second equality is by a change of variables $(wt, t)\mapsto (u, t)$ with a Jacobian whose determinant is $t^{-d}$.
\end{example}

\begin{example}[Line Integral]
\label{priv:ex:line_int}
Instead of taking an integral along a ray as in \Cref{priv:ex:ray_int}, we can integrate along the curve $\zeta_{w_{\mathrm{b}},w_\mathrm{e}}:[0,1]\to\real^d$, $\zeta_{w_{\mathrm{b}},w_{\mathrm{e}}}(t)\ceq (1-t) w_{\mathrm{b}}+tw_{\mathrm{e}}$, which is the straight line connecting $w_{\mathrm{b}}$ and $w_{\mathrm{e}}$ in $\real^d$.

 Let $\Ximagesp=\real^d$, and $V\ceq (Y, W)$, for a random variable $Y$, and a random vector $W=(W_{\mathrm{b}}, W_{\mathrm{e}})$ with values $\real^d\times\real^d$ and Lebesgue density $p_W$.  Set $m(V,X)\ceq Y$, and assume that $\muX(v_\mhit{1},x)\equiv\muX(x)$ does not depend on $v_\mhit{1}$. Suppose that $P_X$ has a Lebesgue density $p_X$. Setting $f(v,x,\mu,\gamma)\ceq \int_{0}^1 \mu((1-t) w_{\mathrm{b}}+tw_{\mathrm{e}})\norm{2}{w_\mathrm{e}-w_\mathrm{b}}\dderiv t$ and $J_{P_{VX}}(\mu)=\E \int_{0}^1 \mu((1-t) W_{\mathrm{b}}+tW_{\mathrm{e}})\norm{2}{W_\mathrm{e}-W_\mathrm{b}}\dderiv t$ yields the parameter
\begin{align*}
 \chi(P_{VX})= J_{P_{VX}}(\muX)= \E \int_{0}^1 \muX((1-t) W_{\mathrm{b}}+tW_{\mathrm{e}})\norm{2}{W_\mathrm{e}-W_\mathrm{b}}\dderiv t \\
 = \int_{\real^d\times\real^d} \int_{0}^1 \muX((1-t) w_{\mathrm{b}}+tw_{\mathrm{e}})\norm{2}{w_\mathrm{e}-w_\mathrm{b}}\dderiv t\,  p_W(w_\mathrm{b}, w_\mathrm{e}) \dderiv w_\mathrm{b} \dderiv w_\mathrm{e}.
\end{align*}
This is the average line integral of $\muX$ across draws of random endpoints $(W_{\mathrm{b}}, W_{\mathrm{e}})$ from $p_W$. If $P_{VX}$ is such that
\begin{align*}
\real^d\ni x\mapsto \fr{1}{p_X(x)}\int_{0}^1 \int_{\real^d}\norm{2}{\Delta}p_W(x-t\Delta, x+(1-t)\Delta)\dderiv \Delta \dderiv t \in L_2(P_{X}),
\end{align*}
then the map in the display is the Riesz representer of $J_{P_{VX}}$. This follows from
\begin{align*}
\int_{\real^d\times\real^d} \int_{0}^1 \mu((1-t) w_{\mathrm{b}}+tw_{\mathrm{e}})\norm{2}{w_\mathrm{e}-w_\mathrm{b}}\dderiv t  p_W(w_\mathrm{b}, w_\mathrm{e}) \dderiv w_\mathrm{b} \dderiv w_\mathrm{e} \\
=\int_{\real^d\times\real^d} \int_{0}^1 \mu(u)\norm{2}{\Delta}p_W(u-t\Delta, u+(1-t)\Delta)|\mathrm{det}J(u,\Delta,t)|\dderiv t  \dderiv \Delta \dderiv u,
\end{align*}
by a change of variables $$(w_\mathrm{b},w_\mathrm{e},t)\mapsto ((1-t) w_{\mathrm{b}}+tw_{\mathrm{e}}, w_\mathrm{e}-w_\mathrm{b},t)\eqc (u,\Delta, t)\in\real^d\times\real^d\times[0,1].$$ Above, $\mathrm{det}J(u,\Delta,t)$ is the determinant of the Jacobian matrix, in $\real^{(2d+1)\times(2d+1)}$, of $(u,\Delta,t)\mapsto (g_1, g_2,g_3)(u,\Delta,t)$ $\ceq (u-t\Delta, u+(1-t)\Delta,t)= (w_\mathrm{b},w_\mathrm{e},t)$, which is
\begin{align*}
J(u,\Delta,t)=\begin{bmatrix}
\Deriv_ u g_1 & \Deriv_\Delta g_1 & \Deriv_t g_1 \\
\Deriv_ u g_2 & \Deriv_\Delta g_2 & \Deriv_t g_2 \\
\Deriv_ u g_3 & \Deriv_\Delta g_3 & \Deriv_t g_3
\end{bmatrix}(u,\Delta,t) =
\begin{bmatrix}
\mathrm{I}_d & -t \mathrm{I}_d & -\Delta \\
\mathrm{I}_d & (1-t) \mathrm{I}_d & -\Delta \\
0_d^\intercal & 0_d^\intercal & 1 \\
\end{bmatrix},
\end{align*}
where $\mathrm{I}_d$ is the identity matrix in $\real^d$, and $0_d^\intercal \in\real^{1\times d}$ is the transposed zero vector. Hence, $|\mathrm{det}J(u,\Delta,t)|=|\mathrm{det}\begin{bmatrix}
\mathrm{I}_d & -t \mathrm{I}_d  \\
\mathrm{I}_d & (1-t) \mathrm{I}_d
\end{bmatrix}|=1$,
because of how the determinant of block matrices is computed. 

While not pursued here, these arguments could be extended to higher order B\'ezier curves with random control points $W_1,\ldots, W_k$, or other curves indeed.
\end{example}

A special subset of our class is a set of parameters with known Riesz representers. Their estimators admit favourable efficiency properties.

\begin{example}[Known Riesz Representer]
\label{priv:ex:known_riesz}
In the general context of \Cref{priv:sec:dr}, suppose that $f(v,x,\mu,\gamma)=\mu(v_\mhit{1},x)b(v_\mhit{1},x)$, independent of $\gamma$, for a \emph{known} function $b:\Vimagesp_\mhit{1}\times\Ximagesp\to\real$. If $b\in L_2(P_{V_{\mhit{1}}X})$, then, trivially, $b$ is the Riesz representer of $L_2(P_{V_{\mhit{1}}X})\ni\mu\mapsto \E \mu(V_\mhit{1},X)b(V_\mhit{1},X)$. 

In this case, \Cref{priv:ass:rates_nuisance_priv} is automatically satisfied, and asymptotic efficiency as per \Cref{priv:cor:psihat_efficiency} follows merely under the consistency requirements of \Cref{priv:ass:consistent_nuisance_priv} for an estimator of $\muX$.
\end{example}

Situations with a known Riesz representer may exist: 

\begin{example}[Time Discounting]
\label{priv:ex:discounting}
Consider a simple economics model with two time periods. Furthest in the future, in Period 2, a stock value $Y\geq 0$ is going to be revealed. In Period 1, a random element $X$ is realised, which potentially helps predict the stock value through $\muX(x)\ceq\Ebc{Y}{X=x}$. In Period 1, an agent owns a stock, and signs a binding contract --- after $X=x$ is realised --- that she would sell the stock at price $\muX(x)$ in Period 2. (This is a ``fair bet,'' as her \emph{a priori} expected profit given $X=x$, $\Ebc{\muX(x)-Y}{X=x}$, is zero.)  

Viewed from Period 1, the present value of the agent's Period-2 income is $\fr{1}{1+\kappa(x)}\muX(x)$, where $\kappa:\Ximagesp\to[0,\infty)$ is a discount rate, a \emph{nonrandom, known} function; the larger the $\kappa$, the more impatient she is. The element $X$ can be thought of as all relevant information available in Period 1.

The parameter
\begin{align*}
\chi(P_{VX})\ceq \E \fr{1}{1+\kappa(X)}\muX(X)
\end{align*}
is the expected (before $X$ is observed) present value of the agent's Period-2 \emph{income} viewed from Period 1. (The expected present value of her \emph{profit}, $\E \fr{1}{1+\kappa(X)}(\muX(X)-Y)$, is zero.) It is of the form \eqref{priv:eq:dr_para} for $V\ceq Y$ and $m(V,X)=Y$.

As $\kappa\geq 0$, $x\mapsto \fr{1}{1+\kappa(x)}\in L_2(P_X)$ is the Riesz-representer of $\mu\mapsto \E \fr{1}{1+\kappa(X)}\mu(X)$, and \eqref{priv:eq:f_continu} and \eqref{priv:eq:f_lin} both hold. This is an instance of a known Riesz representer in \Cref{priv:ex:known_riesz}. Consequently, a merely consistent estimator of $\muX$ suffices for asymptotically efficient estimation of $\chi(P_{VX})$. This is attractive as it affords ``complicated'' predictors $X$, at essentially no cost (although the asymptotic variance  $P_{VZ}\tld\psi^2$ may change with different choices of $X$).

\end{example}

We may extend our class to be indexed by multiple regressions. 

\begin{example}[Multiple Regressions]
\label{priv:ex:extension}
The parameter class \eqref{priv:eq:dr_para} in \Cref{priv:sec:dr} can easily be extended to feature linear dependence on multiple infinite-dimensional regressions
\begin{align}
\label{priv:eq:regression_generalisation}
\begin{aligned}
\muXsub{1}(v_{\mhit{1}},x)&\ceq \Ebc{m_1(V,X)}{V_{\mhit{1}}=v_{\mhit{1}},X=x}, \\
\vdots & \\
\muXsub{K}(v_{\mhit{1}},x)&\ceq \Ebc{m_K(V,X)}{V_{\mhit{1}}=v_{\mhit{1}},X=x},
\end{aligned}
\end{align}
where the $m_k: \Vimagesp\times \Ximagesp\to\real$. Indeed, $H_K\ceq (L_2(P_{V_\mhit{1}X}))^K$ is a Hilbert space with inner product $\innerprodsub{\mu}{\nu}{H_K}\ceq \sum_{k=1}^K\innerprodsub{\mu_k}{\nu_k}{L_2(P_{V_\mhit{1}X})}=\sum_{k=1}^K \int \mu_k\nu_k\dderiv P_{V_\mhit{1}X}$. Then we  define $f:\Vimagesp\times\Ximagesp\times H_K \times\Gamma\to\real$ requiring the linearity of $H_K \ni \mu \mapsto f(V,X,\mu,\gamma)$ $P_{VX}$-almost surely, for all $\gamma\in\Gamma$ (meaning that for all $\beta\in\real$, $f(V,X,(\beta\mu_1,\ldots,\beta\mu_K),\gamma)=\beta f(V,X,(\mu_1,\ldots,\mu_K),\gamma)$, $P_{VX}$-almost surely, for all $\gamma\in\Gamma$), and the continuity of $H_K\ni \mu \mapsto P_{VX}f(V,X,\mu,\gamma)$ for all $\gamma\in\Gamma$ for the distance $\norm{H_K}{\mu-\nu}\ceq \sqrt{\innerprodsub{\mu-\nu}{\mu-\nu}{H_K}}=\sqrt{\sum_{k=1}^K \int (\mu_k-\nu_k)^2\dderiv P_{V_\mhit{1}X}}$. Then the Riesz representation theorem applies equally: for all $(P_{VX},\gamma)$ there exists a unique element $\rieszXs{\gamma}=(\rieszXs{\gamma,1},\ldots,\rieszXs{\gamma,K})\in H_K$ such that 
\begin{align*}
P_{VX}f(V,X,\mu,\gamma)=\sum_{k=1}^K P_{V_\mhit{1}X}(\rieszXs{\gamma,k}\mu_k) \quad\text{ for all } \mu \in H_K.
\end{align*}
This yields results very similar to the single-infinite-dimensional-regression case.

Indeed, let $\muX\ceq (\muXsub{1}, \ldots, \muXsub{K})\in H_K$. Then efficient influence function of $\chi(P_{VX})\ceq P_{VX}f(V,X,\muX,\gammaV(c))$, in \Cref{priv:eqprop:eif} becomes
\begin{align*}
\tld\chi(v,x) =&\, \sum_{k=1}^K \rieszXs{k}(v_{\mhit{1}},x)(m_k(v,x)-\muXsub{k}(v_{\mhit{1}},x))\\
&+\fr{\indic{v_{\mhit{2}}=c}}{p_{V_{\mhit{2}}}(c)}(g(v,x)-\gammaV(c))\E \partial_\gamma f(V,X, \muX, \gammaV(c))  \\
&+ f(v,x, \muX, \gammaV(c))-\chi(P_{VX}), 
\end{align*}
where $\rieszX=(\rieszXs{1},\ldots,\rieszXs{K})\in H_K$ is the Riesz representer of $H_K\ni \mu \mapsto P_{VX}f(V,X,\mu,\gammaV(c))$.

Correspondingly, the rate-double-robustness property of \Cref{priv:thm:dr} becomes
\begin{align*}
\chi'-\chi(P_{VX})+P_{VX}\tld\chi' =&\, \sum_{k=1}^{K}P_{VX}(\rieszXs{k}'-\rieszXs{k})(\muXsub{k}-\muXsub{k}')   \\
&+(\gammaV(c)-\gammaV'(c))\left(\fr{p_{V_{\mhit{2}}}(c)}{p_{V_{\mhit{2}}}'(c)}\expderiv''-\expderiv'\right) \\
&-(\gammaV(c)-\gammaV'(c))^2\fr{P_{VX}\partial_\gamma^2 f(V,X,\muX',\widetilde{\gammaV(c)})}{2},
\end{align*}
with
\begin{align*}
\tld\chi'(v,x)\ceq&\, \sum_{k=1}^K\rieszXs{k}'(v_{\mhit{1}},x)(m_k(v,x)-\muXsub{k}'(v_{\mhit{1}},x))+\fr{\indic{v_{\mhit{2}}=c}}{p_{V_{\mhit{2}}}'(c)}(g(v,x)-\gammaV'(c))\expderiv''  \\
&+ f(v,x, \muX', \gammaV'(c))-\chi', 
\end{align*}
where the $\rieszXs{k}',\muXsub{k}'$ are both arbitrary elements of $L_2(P_{V_\mhit{1}X})$ with $\muX'\ceq(\muXsub{1}',\ldots,\muXsub{K}')$; $p_{V_{\mhit{2}}}'(c),\gammaV'(c),\chi',\expderiv''\in\real$ are arbitrary; $\expderiv'\ceq P_{VX}\partial_\gamma f(V,X,\muX',\gammaV'(c));$ and $\widetilde{\gammaV(c)}$ is some value between $\gammaV(c)$ and $\gammaV'(c)$. That is, the rate-double-robustness property holds, with pairwise rate-tradeoffs between the $\rieszXs{k}'$ and $\muXsub{k}'$.

Moreover, if $\Vimagesp_{1}=\Vimagesp_{\mhit{1},1}\times\ldots \times\Vimagesp_{\mhit{1}, L}$ and $\Ximagesp=\Ximagesp_1\times\ldots \Ximagesp_K$, then we could generalise \eqref{priv:eq:regression_generalisation} further by defining $L\times K$ regressions
\begin{align*}
\muXsub{l,k}(v_{\mhit{1},l},x_k)\ceq\Ebc{m_{lk}(V,X)}{V_{\mhit{1},l}=v_{\mhit{1},l}, X_{k}=x_k},\, (v_{\mhit{1},l},x_k) \in \Vimagesp_{\mhit{1},l}\times \Ximagesp_k,
\end{align*}
with $m_{lk}:\Vimagesp\times\Ximagesp\to\real$, for $(l,k)\in [L]\times [K]$, and by considering the Hilbert space $H\ceq \otimes_{l=1}^L\otimes_{k=1}^K L_2(P_{V_{\mhit{1},l}X_k})$ with the inner product $\innerprodsub{\mu}{\nu}{H}\ceq \sum_{l=1}^L\sum_{k=1}^K\innerprodsub{\mu_{l,k}}{\nu_{l,k}}{L_2(P_{V_{\mhit{1},l}X_k})}=\sum_{k=1}^K \int \mu_{l,k}\nu_{l,k}\dderiv P_{V_{\mhit{1},l}X_k}$. The efficient influence function and the rate-double-robustness would be preserved with the obvious modifications.

Finally, we remark that for $\Vimagesp_{\mhit{2}}=\Vimagesp_{\mhit{2},1}\times\ldots\times \Vimagesp_{\mhit{2},K}$, dependence on multiple low-dimensional regressions 
\begin{align*}
\gammaVsub{k}(v_{\mhit{2}})&\ceq \Ebc{g_k(V,X)}{V_{\mhit{2},k}=v_{\mhit{2},k}},\quad v_{\mhit{2},k}\in\Vimagesp_{\mhit{2},k},\,\, g_k:\Vimagesp \times\Ximagesp\to\real,\,\, k\in[K],
\end{align*}
can be accommodated as well. This entails a vector-valued derivative $\partial_\gamma f$ and corresponding vector $\expderiv$, with number of entries determined by the $|\Vimagesp_{\mhit{2},k}|$.
\end{example}

\subsection{Proofs}
\label{priv:app:sec:dr:subsec:proofs_main}

This section proves \Cref{priv:eqprop:eif,priv:thm:dr} in  \Cref{priv:sec:dr}.

We first note that \eqref{priv:eq:f_continu}, \eqref{priv:eq:f_lin} imply that for all $\gamma\in\Gamma$, there exists a unique function $\rieszX_{P_{VX},\gamma}:\Vimagesp_{\mhit{1}}\times\Ximagesp\to\real$, $\rieszX_{P_{VX},\gamma}\in L_2(P_{V_{\mhit{1}}X})$, such that
\begin{align*}
\E f(V,X, \mu, \gamma)=\E \rieszX_{P_{VX},\gamma}(V_{\mhit{1}},X)\mu(V_{\mhit{1}},X)\quad \text{ for all } \mu\in L_2(P_{V_{\mhit{1}}X}), 
\end{align*}
where the expectations are taken with respect to $P_{VX}$. This is the consequence of the Riesz representation theorem, and $\rieszX_{\gamma,P_{VX}}$ is called the Riesz representer of $\mu\mapsto\E f(V,X, \mu, \gamma)$.

\begin{proof}[Proof of \Cref{priv:eqprop:eif}]
First, note that by the definition of $\muX,\gammaV$,
 \begin{align*}
 \E \rieszX(V_{\mhit{1}},X)(m(V,X)-\muX(V_{\mhit{1}},X)) \\
 = \E\Ebc{\rieszX(V_{\mhit{1}},X)(m(V,X)-\muX(V_{\mhit{1}},X))}{V_{\mhit{1}},X} \\
 	=\E r(V_{\mhit{1}},X)\muX(V_{\mhit{1}},X)-\E r(V_{\mhit{1}},X)\muX(V_{\mhit{1}},X)=0, \\ 
\E \indic{V_{\mhit{2}}=c}(g(V,X)-\gammaV(c))= \Eb{\indic{V_{\mhit{2}}=c}g(V,X)}-p_{V_{\mhit{2}}}(c)\gammaV(c) \\
	=p_{V_{\mhit{2}}}(c)\Ebc{g(V,X)}{V_{\mhit{2}}=c}-p_{V_{\mhit{2}}}(c)\gammaV(c) = 0, 
 \end{align*} 
because $V_{\mhit{2}}$ is distributed on a finite set. Hence, by the definition of $\chi(P_{VX})$ in \eqref{priv:eq:dr_para}, $\E \tld\chi=0$. Because $r\in L_2(P_{V_{\mhit{1}}X})$, \eqref{priv:eq:integrability} implies that $\tld\chi$ is in the tangent set $L_2^0(P_{VX})$ of the nonparametric model. 

Consider now a regular submodel $t\mapsto P_{VX,t}$ for $t\in\real$ in the neighbourhood of zero 
 with $P_{VX,0}=P_{VX}$ and the property that $\fr{\deriv}{\deriv t}\big\vert_{t=0} P_{VX,t}(\deriv v,\deriv x)=s(v,x)P_{VX}(\deriv v,\deriv x)$ for the score function $s$ running through the tangent set $L_2^0(P_{VX})$; for example $\deriv P_{VX}=\exp(ts)(P_{VX}\exp(ts))^{-1}\deriv P_{VX}$ as in \cite{ray_semiparametric_2020}. See \Citet[Part III]{bolthausen_semiparametric_2002} for regular submodels. For \eqref{priv:eq:eif_chi}, it suffices to show $\derivtnull \chi(P_{VX,t})=\E \tld\chi s$. Assuming that differentiation and expectation commutes, we have 
 \begin{align*}
\derivtnull \chi(P_{VX,t})=&\,\derivtnull \int_{\Vimagesp\times\Ximagesp} f(v,x,\muXt,\gammaVt(c))\dderiv P_{VX,t}(v,x) \\
=&\, \int_{\Vimagesp\times\Ximagesp} \derivtnull f(v,x,\muXt,\gammaVt(c))\dderiv P_{VX}(v,x) \\
&+ \E f(V,X,\muX,\gammaV(c))s(V,X).
\end{align*}
Here,
\begin{align*}
\int_{\Vimagesp\times\Ximagesp} \derivtnull f(v,x,\muXt,\gammaVt(c))\dderiv P_{VX}(v,x) \\
 = \derivtnull \E f(V,X,\muXt,\gammaV(c)) +\derivtnull \E f(V,X,\muX,\gammaVt(c)).
\end{align*}
By \eqref{priv:eq:riesz_rep}, $\E f(V,X,\muXt,\gammaV(c))=\E \rieszX(V_{\mhit{1}},X)\muXt(V_{\mhit{1}},X)$. Taking derivatives, the properties of (conditional) expectation give
\begin{align}
\label{priv:eq:mux_t_derivative}
\begin{aligned}
\derivtnull \muXt(v_{\mhit{1}},x)&=\Ebc{(m(V,X)-\muX(v_{\mhit{1}},x))s(V,X)}{V_{\mhit{1}}=v_{\mhit{1}},X=x}, \\
\derivtnull \gammaVt(c)&=\Ebc{(g(V,X)-\gammaV(c))s(V,X)}{V_{\mhit{2}}=c}  \\
				   &= \E \fr{\indic{V_{\mhit{2}}=c}}{p_{V_{\mhit{2}}}(c)}(g(V,X)-\gammaV(c))s(V,X). 
\end{aligned}				 
\end{align}
This can be seen in a few steps. First, given a collection of coordinates $V_\mhit{j}$ of $V$, we can assume without loss of generality that $V=(V_\mhit{j},V_{-\mhit{j}})$ for another collection of coordinates $V_{-\mhit{j}}$ of $V$. Second, $P_{V_{-\mhit{j}}\cond V_{\mhit{j}}X}(B \cond v_{\mhit{j}},x)=\fr{\deriv P_{V_{-\mhit{j}}V_{\mhit{j}}X}(B,\cdot,\cdot)}{\deriv P_{V_{\mhit{j}}X}}(v_{\mhit{j}},x)$ can be verified to be the conditional distribution of $V_{-\mhit{j}}$ given $(V_{\mhit{j}},X)=(v_{\mhit{j}},x)$ as in the proof of \Cref{priv:lem:qtc_implications}, noting that  $P_{V_{-\mhit{j}}V_{\mhit{j}}X}(B,\cdot,\cdot)\ll  P_{V_{\mhit{j}}X}$ for all $B\in\Vsigmasub{\mhit{j}}$.  Third, if $\pi_t\ceq\fr{\deriv \lambda_t}{\deriv \nu_t}$ for valid submodels $t\mapsto(\lambda_t,\nu_t)$ of measures, then $\partial_t \pi_t=\fr{\deriv \partial_t \lambda_t}{\deriv \nu_t}-(\fr{\deriv \lambda_t}{\deriv \nu_t})(\fr{\deriv \partial_t \nu_t}{\deriv \nu_t})$ provided the densities on the right are well defined. Fourth, the given submodel $t\mapsto P_{VX,t}$ induces the marginal $\deriv P_{V_{\mhit{j}}X,t}(v_{\mhit{j}},x)=\int_{\Vimagesp_{\mhit{j}}} s(v_{-\mhit{j}},v_{\mhit{j}},x) P_{V_{-\mhit{j}}V_{\mhit{j}}X}(\deriv v_{-\mhit{j}}, \deriv v_{\mhit{j}}, \deriv x)$. 
Finally, apply these steps to obtain 
\begin{align*}
\fr{\deriv \partial_t P_{V_{-\mhit{j}},V_{\mhit{j}},X,t}(\deriv v_{-\mhit{j}},\cdot,\cdot)\big\vert_{t=0}}{\deriv P_{V_{\mhit{j}}X}}(v_{\mhit{j}},x)=s(v_{-\mhit{j}},v_{\mhit{j}}, x)P_{V_{-\mhit{j}}\cond V_{\mhit{j}}X}(\deriv v_{-\mhit{j}} \cond v_{\mhit{j}},x), \\
\fr{\deriv \partial_t P_{V_{\mhit{j}}X,t}\big\vert_{t=0}}{\deriv P_{V_{\mhit{j}X}}}(v_{\mhit{j}},x)=\int s(v_{-\mhit{j}},v_{\mhit{j}}, x)P_{V_{-\mhit{j}}\cond V_{\mhit{j}}X}(\deriv v_{-\mhit{j}} \cond v_{\mhit{j}},x).
\end{align*}

The display \eqref{priv:eq:mux_t_derivative} then implies
\begin{align*}
\derivtnull\E \rieszX(V_{\mhit{1}},X)\muXt(V_{\mhit{1}},X)\\
 = \Eb{r(V_{\mhit{1}},X)\Ebc{\left(m(V,X)-\muX(V_{\mhit{1}},X)\right)s(V,X)}{V_{\mhit{1}},X}} \\
= \E r(V_{\mhit{1}},X)\left(m(V,X)-\muX(V_{\mhit{1}},X)\right)s(V,X),\\
\derivtnull \E f(V,X,\muX,\gammaVt(c)) =\Eb{\partial_\gamma f(V,X,\muX,\gammaV(c))\derivtnull \gammaVt(c)} \\
=\Eb{\partial_\gamma f(V,X,\muX,\gammaV(c))\Eb{\fr{\indic{V_{\mhit{2}}=c}}{p_{V_{\mhit{2}}}(c)}(g(V,X)-\gammaV(c))s(V,X)}} \\
= \Eb{\Eb{\partial_\gamma f(V,X,\muX,\gammaV(c))} \fr{\indic{V_{\mhit{2}}=c}}{p_{V_{\mhit{2}}}(c)}(g(V,X)-\gammaV(c))s(V,X)}.
\end{align*}
Hence,
\begin{align*}
\derivtnull \chi(P_{VX,t}) =  \derivtnull \E f(V,X,\muXt,\gammaV(c)) +\derivtnull \E f(V,X,\muX,\gammaVt(c)) \\ 
+\Eb{f(V,X,\muX,\gammaV(c))s(V,X)} \\
= \E r(V_{\mhit{1}},X)\left(m(V,X)-\muX(V_{\mhit{1}},X)\right)s(V,X)  \\
 + \Eb{\Eb{\partial_\gamma f(V,X,\muX,\gammaV(c))} \fr{\indic{V_{\mhit{2}}=c}}{p_{V_{\mhit{2}}}(c)}(g(V,X)-\gammaV(c))s(V,X)}
 \\ +\Eb{f(V,X,\muX,\gammaV(c))s(V,X)} \\
 = \E \tld\chi(V,X)s(V,X),
\end{align*}
where the last equality follows from $\E s(V,X)=0$.
\end{proof}

\begin{proof}[Proof of \Cref{priv:thm:dr}]
First,
\begin{align}
P_{VX}\big[ -r(V_{\mhit{1}},X)h(V_{\mhit{1}},X)+f(V,X,h,\gammaV(c))\big]&=0, \label{priv:eq:zero_riesz} \\
P_{VX}\big[ (m(V,X)-\muX(V_{\mhit{1}},X))h(V_{\mhit{1}},X)\big]&=0 \label{priv:eq:zero_mu}
\end{align}
for all  $h\in L_2(P_{V_{\mhit{1}}X})$, where the first equality is by \eqref{priv:eq:riesz_rep} and the second is by the definition of $\muX$ and the tower property of expectation. Then, because $\tld\chi$ is an influence function satisfying $P_{VX}\tld\chi=0$, and $\tld\chi,\chi_0$ --- not depending on $(v,x)$ --- are constants with respect to $P_{VX}$-integration,
\begin{align*}
\chi'-\chi_0+P_{VX}\tld\chi' = P_{VX}\left[\chi'+\tld\chi'\right] - P_{VX}\left[\chi_0+\tld\chi\right] \\
= P_{VX}\left\{\rieszX'(V_{\mhit{1}},X)(m(V,X)-\muX'(V_{\mhit{1}},X))+\fr{\indic{V_{\mhit{2}}=c}}{p_{V_{\mhit{2}}}'(c)}(g(V,X)-\gammaV'(c))e'' \right. \\
+\left. \vphantom{\fr{\indic{V_{\mhit{2}}=c}}{p_{V_{\mhit{2}}}'(c)}} f(V,X, \muX', \gammaV'(c)) \right\} \\
-P_{VX}\left\{\rieszX(V_{\mhit{1}},X)(m(V,X)-\muX(V_{\mhit{1}},X))\vphantom{\fr{\indic{V_{\mhit{2}}=c}}{p_{V_{\mhit{2}}}(c)}} \right. \\
+\left. \fr{\indic{V_{\mhit{2}}=c}}{p_{V_{\mhit{2}}}(c)}(g(V,X)-\gammaV(c))P_{VX} \partial_\gamma f(V,X, \muX, \gammaV(c))+\vphantom{\fr{\indic{V_{\mhit{2}}=c}}{p_{V_{\mhit{2}}}(c)}} f(V,X, \muX, \gammaV(c)) \right\}  \\
-P_{VX}\left\{-r(V_{\mhit{1}},X)(\muX'(V_{\mhit{1}},X)-\muX(V_{\mhit{1}},X)) \right. \\
\left. +f(V,X,\muX',\gammaV(c))-f(V,X,\muX,\gammaV(c))\right\} \\
-P_{VX}\left\{\left[m(V,X)-\muX(V_{\mhit{1}},X)\right]\left[r'(V_{\mhit{1}},X)-r(V_{\mhit{1}},X)\right]\right\},
\end{align*}
where the last two expectations are zero: we apply \eqref{priv:eq:zero_riesz} and \eqref{priv:eq:zero_mu} choosing $h$ to be $\muX'-\muX$ and $r'-r$, respectively, and use the linearity of $f$ in \eqref{priv:eq:f_lin}. As $P_{VX}\fr{\indic{V_{\mhit{2}}=c}}{p_{V_{\mhit{2}}}(c)}(g(V,X)-\gammaV(c))=0$ by the tower property and the definition of $\gammaV$, 
\begin{align*}
\chi'-\chi_0+P_{VX}\tld\chi' =  -P_{VX}(r-r')(\muX-\muX') \\
+P_{VX}\left\{\fr{\indic{V_{\mhit{2}}=c}}{p_{V_{\mhit{2}}}'(c)}(g(V,X)-\gammaV'(c)) e'' \right. \\
\left. \phantom{\fr{\indic{V_{\mhit{2}}=c}}{p_{V_{\mhit{2}}}'(c)}} +f(V,X,\muX',\gammaV'(c))-f(V,X,\muX',\gammaV(c))\right\}.
\end{align*}
Next, a Taylor-approximation of $\gamma\mapsto f(V,X,\muX',\gamma)$ by \eqref{priv:eq:f_diff} with a mean-value representation of the the remainder term gives
\begin{align*}
f(V,X,\muX',\gammaV(c))=&\,f(V,X,\muX',\gammaV'(c))+\partial_\gamma f(V,X,\muX',\gammaV'(c))\big[\gammaV(c)-\gammaV'(c)\big] \\
&+\fr{1}{2} \partial_\gamma^2 f(V,X,\muX',\widetilde{\gammaV(c)})\big[\gammaV(c)-\gammaV'(c)\big]^2.
\end{align*}
As $P_{VX}\indic{V_{\mhit{2}}=c}(g(V,X)-\gammaV'(c))=p_{V_{\mhit{2}}}(c)(\gammaV(c)-\gammaV'(c))$,
\begin{align*}
P_{VX}\left\{\fr{\indic{V_{\mhit{2}}=c}}{p_{V_{\mhit{2}}}'(c)}(g(V,X)-\gammaV'(c)) e'' +f(V,X,\muX',\gammaV'(c))-f(V,X,\muX',\gammaV(c))\right\} \\
= \fr{p_{V_{\mhit{2}}}(c)}{p_{V_{\mhit{2}}}'(c)}(\gammaV(c)-\gammaV'(c))e''- P_{VX}\partial_\gamma f(V,X,\muX',\gammaV'(c))\big[\gammaV(c)-\gammaV'(c)\big] \\
-\fr{1}{2}P_{VX} \partial_\gamma^2 f(V,X,\muX',\widetilde{\gammaV(c)})\big[\gammaV(c)-\gammaV'(c)\big]^2,
\end{align*}
which yields the assertion by the definition of $e'$.
\end{proof}



\section{Privacy}\label{priv:app:priv} 

This section contains complementary results to and proofs of the claims in \Cref{priv:sec:priv} and \Cref{priv:sec:estim_dr_priv}. \Cref{priv:app:priv:subsec:private_inferential_prop} complements 
\Cref{priv:sec:priv} by deriving the tangent set of the private model, and proves \Cref{priv:prop:eif_priv}. \Cref{priv:app:priv:subsec:private_estimation} proves \Cref{priv:sec:estim_dr_priv}'s \Cref{priv:cor:psihat_efficiency}. Our proofs rest on the auxiliary results \Cref{priv:lem:linopqx_properties,priv:lem:qtc_implications,priv:lem:norms} in \Cref{priv:app:priv:subsec:aux}, proven there.

Finally, in \Cref{priv:app:priv:subsec:relation_to_diff_priv}, the relation between total-variation and differential privacy is quantified, while \Cref{priv:app:priv:subsec:invertible_mechanisms} discusses alternative choices for invertible privacy mechanisms in the sense of \eqref{priv:eq:linopL}.

\subsection{Private Inferential Properties}\label{priv:app:priv:subsec:private_inferential_prop}

This section complements \Cref{priv:sec:priv} by deriving the tangent set of the private model, and proves \Cref{priv:prop:eif_priv}.


\Cref{priv:lem:semiparametric_vz} shows that the tangent set of the model $\mathcal{P}_{VZ}(Q,\mathcal{P}_{VX})$ is
$$\tanset_{VZ}(Q,\mathcal{P}_{VX})=\linopQXadj\tanset_{VX}=\set{(v,z)\mapsto\Ebc{s(V,X)}{V=v,Z=z}: s\in\tanset_{VX}},$$
 which is typical of mixture models such as $P_{VZ}$ in \eqref{priv:eq:distribution_vz} (e.g.\ \Citet[Chapter 25.5]{van_der_vaart_asymptotic_1998}), and is also in agreement with \Citet[Lemma 3.1]{steinberger_efficiency_2023}.
\Cref{priv:lem:semiparametric_vz} also shows that $\mathcal{P}_{VZ}(Q,\mathcal{P}_{\Vimagesp\Ximagesp})$ remains nonparametric if $\linopQX$ is invertible, and when $X$ is discrete with $|\Ximagesp|=|\Zimagesp|$, this holds with ``if and only if.''

\begin{lemma}[Private Semiparametric Properties]
\label{priv:lem:semiparametric_vz}
Let $\setQdisc,\setQdiscinv,\setQgen$ be the sets of mechanism defined in \eqref{priv:eq:discreteQdef}, \eqref{priv:eq:discreteQinv}, \eqref{priv:eq:qtv}, respectively. Then the tangent set $\tanset_{VZ}(Q, \mathcal{P}_{VX})$ at $P_{VZ}$ in the model \eqref{priv:eq:vzq_model} is as follows.
\begin{enumerate}[label=(\roman*)]
\item \label{priv:lem:semiparametric_vz_transform} Let $Q\in\setQ(\Ximagesp\to\Zimagesp)$ be arbitrary. Then $\tanset_{VZ}(Q,\mathcal{P}_{VX})=\set{\linopQXadj s: s\in\tanset_{VX}}$ for the tangent set $\tanset_{VX}\subset  L_2^0(P_{VX})$ at $P_{VX}$ in any model $\mathcal{P}_{VX}$, and $\linopQXadj$ in \Cref{priv:lem:linopqx_properties}.
\item \label{priv:lem:semiparametric_vz_nonpara_discr} Suppose that  $\tanset_{VX}= L_2^0(P_{VX})$ and $Q\in\setQdisc$. Then $\tanset_{VZ}(Q, \mathcal{P}_{\Vimagesp\Ximagesp})=L_2^0(P_{VZ})$ if and only if $Q\in\setQdiscinv$.
\item \label{priv:lem:semiparametric_vz_nonpara_gen} Suppose that $\tanset_{VX}= L_2^0(P_{VX})$  and $Q\in\setQgen$. Then the closure of $\tanset_{VZ}(Q, \mathcal{P}_{\Vimagesp\Ximagesp})$ in $L_2(P_{VZ})$ is $L_2^0(P_{VZ})$.
\end{enumerate}
\end{lemma}

\begin{proof}[Proof of \Cref{priv:lem:semiparametric_vz}]

Assertion \ref{priv:lem:semiparametric_vz_transform}. Consider the submodel $t\mapsto \deriv P_{VX,t}=e^{ts}(P_{VX}e^{ts})^{-1}\deriv P_{VX}$ for $s\in\tanset_{VX}\subset L_2^0(P_{VX})$ (e.g.\ \cite{ray_semiparametric_2020}). Clearly, $P_{VX,t}\ll P_{VX}$ for all $t$. This submodel is differentiable in the quadratic mean with score $s$ (\Citet[Part III, Definition 1.6]{bolthausen_semiparametric_2002}): under $P_{VX,t}\ll P_{VX}$,
\begin{align}
\int \left[\fr{1}{t}\left(\sqrt{\fr{\deriv P_{VX,t}}{{\deriv P_{VX}}}}-1\right)-\fr{1}{2}s\right]^2\dderiv P_{VX}\to 0\,\text{ as } t\to0. \label{priv:eq:dqm}
\end{align}
The submodel $P_{VX,t}$ induces the submodel $P_{VZ,t}\ceq \int_{(\cdot)\times\Ximagesp} Q(\cdot\cond x)\dderiv P_{VX,t}(v,x)$ for $P_{VZ}=\int_{(\cdot)\times\Ximagesp} Q(\cdot\cond x)\dderiv P_{VX}(v,x)$ because $Q$ is known, following the construction \eqref{priv:eq:distribution_vz}. Thus, $P_{VZ,t}\ll P_{VZ}$ for all $t$. For the assertion, it is then necessary and sufficient that $P_{VZ,t}$ be differentiable in quadratic mean with score $\linopQXadj s$:
\begin{align}
\int \left[\fr{1}{t}\left(\sqrt{\fr{\deriv P_{VZ,t}}{{\deriv P_{VZ}}}}-1\right)-\fr{1}{2}\linopQXadj s\right]^2\dderiv P_{VZ}\to 0\,\text{ as } t\to0. \label{priv:eq:dqm_vz}
\end{align}
To show this, we follow \Citet[Chapter 6.6]{pollard_asymptotia_2005}. Note that $\fr{\deriv P_{VZ,t}}{{\deriv P_{VZ}}}=\linopQXadj \fr{\deriv P_{VX,t}}{{\deriv P_{VX}}}$, because one can verify (as in the proof of \Cref{priv:lem:qtc_implications}) that the conditional distribution on the right in $\linopQXadj$ is $P_{X\cond VZ}(B\cond v,z)=\fr{\deriv \int_{(\cdot)\times B}Q(\cdot\cond x)\deriv P_{VX}(\tld v, x)}{\deriv P_{VZ}}(v,z)$. Further,
\begin{align}
\fr{\deriv P_{VZ,t}}{{\deriv P_{VZ}}}=\linopQXadj \fr{\deriv P_{VX,t}}{{\deriv P_{VX}}} = \linopQXadj \left(\sqrt{\fr{\deriv P_{VX,t}}{{\deriv P_{VX}}}} -\linopQXadj \sqrt{\fr{\deriv P_{VX,t}}{{\deriv P_{VX}}}} + \linopQXadj \sqrt{\fr{\deriv P_{VX,t}}{{\deriv P_{VX}}}}  \right)^2 \nonumber \\
=\linopQXadj \left(\sqrt{\fr{\deriv P_{VX,t}}{{\deriv P_{VX}}}} -\linopQXadj \sqrt{\fr{\deriv P_{VX,t}}{{\deriv P_{VX}}}}\right)^2+ \left(\linopQXadj \sqrt{\fr{\deriv P_{VX,t}}{{\deriv P_{VX}}}} \right)^2 \eqc \sigma_t^2 + \left(\linopQXadj \sqrt{\fr{\deriv P_{VX,t}}{{\deriv P_{VX}}}} \right)^2, \label{priv:eq:sigmat}
\end{align}
where the last equality follows from $\linopQXadj$ being the conditional expectation given $(V,Z)$, and we introduced $\sigma_t^2\in L_2(P_{VZ})$. It follows that 
\begin{align}
L_2(P_{VZ})\ni\delta_t\ceq\sqrt{\fr{\deriv P_{VZ,t}}{{\deriv P_{VZ}}}}- \linopQXadj \sqrt{\fr{\deriv P_{VX,t}}{{\deriv P_{VX}}}}\geq 0. \label{priv:eq:positivedelta}
\end{align}

Write
\begin{align}
r_t\ceq \sqrt{\fr{\deriv P_{VX,t}}{{\deriv P_{VX}}}} - 1 - \fr{1}{2}ts,\quad
\ba r_t \ceq \sqrt{\fr{\deriv P_{VZ,t}}{{\deriv P_{VZ}}}} - 1 - \fr{1}{2}t \linopQXadj s \label{priv:eq:residuals_dqm}
\end{align}
for the residuals $r_t\in L_2(P_{VX})$ and $\ba r_t \in L_2(P_{VZ})$. By \eqref{priv:eq:dqm}, $P_{VX}r_t^2=\smallO{t^2}$ as $t\to0$. To show \eqref{priv:eq:dqm_vz}, we show $P_{VZ}\ba r_t^2=\smallO{t^2}$ as $t\to0$. Using that $(x+y)^2\leq 2(x^2+y^2)$ for all $x,y\in\real$, expand
\begin{align*}
P_{VX}\ba r_t^2 \leq  2 P_{VZ}(\linopQXadj r_t)^2 + 2 P_{VZ}(\ba r_t - \linopQXadj  r_t)^2.
\end{align*}
Here, the first term $P_{VZ}(\linopQXadj r_t)^2\leq$ $P_{VZ}\linopQXadj r_t^2=P_{VX}r_t^2=\smallO{t^2}$, where the first inequality is Jensen's, conditional on $(V,Z)$; the second term
\begin{align*}
P_{VZ}(\ba r_t - \linopQXadj  r_t)^2 = P_{VZ}\left(\sqrt{\fr{\deriv P_{VZ,t}}{{\deriv P_{VZ}}}} - \linopQXadj \sqrt{\fr{\deriv P_{VX,t}}{{\deriv P_{VX}}}}  \right)^2 = P_{VZ}\delta_t^2. 
\end{align*}
By definitions \eqref{priv:eq:sigmat} and \eqref{priv:eq:positivedelta}, $\left(\delta_t+\linopQXadj \sqrt{\fr{\deriv P_{VX,t}}{{\deriv P_{VX}}}}\right)^2=\fr{\deriv P_{VZ,t}}{{\deriv P_{VZ}}}=\sigma_t^2+\left(\linopQXadj \sqrt{\fr{\deriv P_{VX,t}}{{\deriv P_{VX}}}}\right)^2$, yielding the algebraic identity 
\begin{align}
\delta_t^2=\sigma_t^2-2\delta_t \linopQXadj \sqrt{\fr{\deriv P_{VX,t}}{{\deriv P_{VX}}}}.\label{priv:eq:delta_identity}
\end{align}
But $\delta_t\geq 0$ by \eqref{priv:eq:positivedelta}, and so is $\linopQXadj \sqrt{\fr{\deriv P_{VX,t}}{{\deriv P_{VX}}}}\geq 0$ as a density is nonnegative; hence $\delta_t^2\leq \sigma_t^2$. Using \eqref{priv:eq:residuals_dqm}, 
\begin{align}
\begin{aligned}\label{priv:eq:sigmabound}
\sigma_t^2= \linopQXadj \left(\fr{1}{2}t(s-\linopQXadj s)+r_t-\linopQXadj r_t\right)^2 \leq \fr{1}{2} t^2 \linopQXadj (s-\linopQXadj s)^2 + 2 \linopQXadj (r_t-\linopQXadj r_t)^2 \\
\leq \fr{t^2}{2} \linopQXadj s^2 + 2 \linopQXadj r_t, 
\end{aligned}
\end{align}
because the projection $\linopQXadj$ decreases the conditional variance. Here, by the tower property of expectations, $P_{VZ}\linopQXadj s^2=P_{VX}s^2=\bigO{1}$ as $s\in L_2(P_{VX})$ since it is a score, and $P_{VZ}\linopQXadj r_t^2=P_{VX}r_t^2=\smallO{t^2}$ by \eqref{priv:eq:dqm}. Hence, $P_{VZ}\sigma_t^2\to0$ as $t\to 0$.

For a fixed, strictly positive $\beta\in\real$, define the set of $(V,Z)$, 
$$A_{t,\beta}\ceq\left\{\linopQXadj \sqrt{\fr{\deriv P_{VX,t}}{{\deriv P_{VX}}}} \geq \fr{1}{2}, \sigma_t^2\leq \beta\right\}.$$
On $A_{t,\beta}$,  squaring \eqref{priv:eq:delta_identity} gives $$\delta_t^2=\fr{\sigma_t^4-\delta_t^4-2\delta_t^3\linopQXadj \sqrt{\fr{\deriv P_{VX,t}}{{\deriv P_{VX}}}}}{4\left(\linopQXadj \sqrt{\fr{\deriv P_{VX,t}}{{\deriv P_{VX}}}}\right)^2}\leq \sigma_t^4\leq \beta \sigma_t^2$$ since $\delta_t\geq 0$; while on its complement $A_{t,\beta}^\mathrm{c}$, like $P_{VZ}$-everywhere, $\delta_t^2\leq\sigma_t^2$ as seen above. From \eqref{priv:eq:sigmabound} and \eqref{priv:eq:dqm},
\begin{align*}
P_{VZ}\delta_t^2 = P_{VZ}\indic{A_{t,\beta}}\delta_t^2+P_{VZ}\indic{A_{t\beta}^{\mathrm{c}}}\delta_t^2 \leq \beta P_{VZ} \sigma_t^2 + P_{VZ}\indic{A_{t,\beta}^{\mathrm{c}}}\sigma_t^2 \\
\leq \fr{\beta t^2}{2} P_{VZ}\linopQXadj s^2 + \smallO{t^2}+\fr{t^2}{2}P_{VZ}\indic{A_{t,\beta}^{\mathrm{c}}}\linopQXadj s^2 + \smallO{t^2}.
\end{align*}
Since $P_{VZ}(A_{t,\beta}^{\mathrm{c}})\to 0$ as $t\to 0$ by the previous paragraph, a small enough choice of $\beta$ shows the right side $\smallO{t^2}$ 
Conclude that $P_{VZ}\ba r_t^2\leq P_{VZ}\delta_t^2=\smallO{t^2}$, whereby \eqref{priv:eq:dqm_vz} holds.

Assertions \ref{priv:lem:semiparametric_vz_nonpara_discr} and \ref{priv:lem:semiparametric_vz_nonpara_gen}. We build on \Citet[Chapter 25]{van_der_vaart_asymptotic_1998}. By \ref{priv:lem:semiparametric_vz_transform}, $\tanset_{VZ}=\set{\linopQXadj s: s\in L_2^0(P_{VX})}$, which is the range $R(\linopQXadj)$ of $\linopQXadj$. It follows from the defining relation between $\linopQX$ and $\linopQXadj$ that $R(\linopQXadj)^\perp=N((\linopQXadj)^{*})$, where $(\linopQXadj)^*=\linopQX$ in Hilbert spaces $L_2(P_{VX})$ and $L_2(P_{VZ})$, and 
$$R(\linopQXadj)^\perp\ceq\set{k\in L_2(P_{VZ}): P_{VZ}k\kappa =0 \text{ holds for all } \kappa\in R(\linopQXadj)}$$
is the orthocomplement of $R(\linopQXadj)$, and $N(\linopQX)\ceq\set{k\in L_2(P_{VZ}): \linopQX k = 0}$ is the kernel of $\linopQX$. By properties of Hilbert spaces, it follows that $\overline{R(\linopQXadj)}=N(\linopQX)^\perp$, where $\overline{R(\linopQXadj)}$ is the closure of $R(\linopQXadj)$ in $L_2(P_{VZ})$. Studying the kernel $N(\linopQX)$, the relation
\begin{align*}
0=(\linopQX k)(v,x)\, \text{ for all } (v,x)\in\Vimagesp\times\Ximagesp,
\end{align*}
for \ref{priv:lem:semiparametric_vz_nonpara_discr} is equivalent to
\begin{align*}
0_{J} = Q^\intercal \begin{bmatrix}
k(v,z_1) \\
k(v,z_2) \\
\vdots \\
k(v,z_J)
\end{bmatrix}\, \text{ for all } v\in\Vimagesp \quad \\ \iff
(Q^\intercal)^{-1}0_J=0_J= \begin{bmatrix}
k(v,z_1) \\
k(v,z_2) \\
\vdots \\
k(v,z_J)
\end{bmatrix}\, \text{ for all } v\in\Vimagesp,
\end{align*}
by invertibility of $Q$. Hence, $k=0$, and thus $\overline{R(\linopQXadj)}=N(\linopQX)^\perp=L_2(P_{VZ})$ so that the model remains nonparametric. For \ref{priv:lem:semiparametric_vz_nonpara_gen}, $\linopQXinv0=0$ with $\linopQXinv$ in \Cref{priv:lem:linopqx_properties} \ref{priv:lem:linopqx_properties_inv_qtc}, yields the same conclusion.
\end{proof}

\begin{proof}[Proof of \Cref{priv:prop:eif_priv}]
We follow \Citet[Chapter 25.5]{van_der_vaart_asymptotic_1998}. First, we address the nonparametric model $\mathcal{P}_{VX}=\mathcal{P}_{\Vimagesp\Ximagesp}$ with tangent set $\tanset_{VX}=L_2^0(P_{VX})$ and efficient influence function $\tld\chi$ in \eqref{priv:eq:eif_chi}. Consider the submodel $t\mapsto P_{VX,t}$ of \Cref{priv:lem:semiparametric_vz} with scores $s\in\tanset_{VX}=L_2^0(P_{VX})$. Because $Q\in\setQident$, these submodels induce submodel $t\mapsto P_{VZ,t}$ and the tangent set $\tanset_{VZ}(Q,\bar{P}_{VX})=\set{\linopQXadj s: s\in L_2^0(P_{VX})}$ by \Cref{priv:lem:semiparametric_vz}. By construction, since $Q\in\setQident$, $\psi(P_{VZ,t})=\chi(P_{VX,t})$ for all $t\in\real$.

The efficient influence function $\tld\psi$ of $\psi(P_{VZ})$ exists if and only if $$\derivtnull \psi(P_{VZ,t})=P_{VZ}[\tld\psi (\linopQXadj s)]$$ for all regular submodels $P_{VZ,t}$ with score $\linopQXadj s\in \set{\linopQXadj s: s\in L_2^0(P_{VX})}$. But $\derivtnull \psi(P_{VZ,t})=\derivtnull \chi(P_{VX,t})$ by the previous paragraph. Since $\tld\chi$ is the efficient influence function of $\chi(P_{VX})$ by assumption, we must also have that $$\derivtnull \chi(P_{VX,t})=P_{VX}[s\tld\chi].$$ Thus, for all  $s\in L_2^0(P_{VX})$,
\begin{align*}
P_{VZ}[\tld\psi (\linopQXadj s)]=\derivtnull \psi(P_{VZ,t})=\derivtnull \chi(P_{VX,t})=P_{VX}[s\tld\chi].
\end{align*}
By the definition of the adjoint $(\linopQXadj)^{*}=\linopQX$, the inner product $P_{VZ}[\tld\psi (\linopQXadj s)]$ is equal to $P_{VZ}[\tld\psi (\linopQXadj s)]=P_{VX}[(((\linopQXadj)^{*})\tld\psi)s]=P_{VX}[(\linopQX \tld\psi)s]$, which by the last display is equal to $P_{VX}[s\tld\chi]$ for all $s\in L_2^0(P_{VX})$. Hence $P_{VX}[(\linopQX \tld\psi)s]=P_{VX}[s\tld\chi]$ for all $s\in L_2^0(P_{VX})$, or, by rearrangement, $P_{VX}[(\linopQX \tld\psi-\tld\chi)s]=0$ for all $s\in L_2^0(P_{VX})$. Equivalently, $\linopQX \tld\psi-\tld\chi$ must be in the orthocomplement of $L_2^0(P_{VX})\subset L_2(P_{VX})$, which, as we show below, is 
$$L_2^0(P_{VX})^\perp=\set{f\in L_2(P_{VX}):f-P_{VX}f=0\, P_{VX}\text{-a.s.}}.$$
Now, $\tld\chi$ is the efficient influence function for $\chi(P_{VX})$, so $P_{VX}\tld\chi=0$. For $\tld\psi$ to be an influence function for $\psi(P_{VZ})$, we must have $P_{VZ}\tld\psi=0$, but by \Cref{priv:lem:linopqx_properties}\ref{priv:lem:linopqx_properties_change}, $P_{VZ}\tld\psi=P_{VX}\linopQX\tld\psi$. Hence, $\tld\chi,\linopQX\tld\psi \in L_2^0(P_{VX})$ and $\linopQX\tld\psi - \tld\chi \in L_2^0(P_{VX})$. But $\linopQX\tld\psi - \tld\chi \in L_2^0(P_{VX})^\perp$ too as we showed above. Since $L_2^0(P_{VX})\cap L_2^0(P_{VX})^\perp=\set{f:f=0\, P_{VX}\text{-a.s.}}$, we must have $\linopQX \tld\psi=\tld\chi$ $P_{VX}$-a.s.. Because $Q\in\setQident$, $\linopQXinv$ exists, giving $\tld\psi=\linopQXinv \tld\chi$.

To see that the orthocomplement $L_2^0(P_{VX})^\perp$ of $L_2^0(P_{VX})$ in $L_2(P_{VX})$ is 
$$\set{f\in L_2(P_{VX}):f-P_{VX}f=0\, P_{VX}\text{-a.s.}},$$ take some $f\in L_2^0(P_{VX})^{\perp}$. Because $f\in L_2^0(P_{VX})^{\perp}$ and $f-P_{VX}f \in L_2^0(P_{VX})$, we must have $P_{VX}[f(f-P_{VX}f)]=0$. Because $P_{VX}[f(f-P_{VX}f)]=P_{VX}[(f-P_{VX}f)(f-P_{VX}f)]$, we must have $f-P_{VX}f=0$ $P_{VX}$-a.s..\ Because $f\in L_2^0(P_{VX})^{\perp}$ was chosen arbitrarily, the assertion follows.

\textcolor{black}{Second, we address an arbitrary model $\mathcal{P}_{VX}\subset\mathcal{P}_{\Vimagesp\Ximagesp}$ with $\tanset_{VX}\subset L_2^0(P_{VX})$ and efficient influence function $\varphi\in\linopQX\linopQXadj \tanset_{VX}$. Let $\tld\Psi$ be the efficient influence function of $\psi(P_{VZ})$ in the model $\mathcal{P}_{VZ}(Q, \mathcal{P}_{VX})$ of \eqref{priv:eq:vzq_model}. We have $P_{VX}[(\linopQX \tld\Psi-\varphi)s]=0$ for all $s\in \tanset_{VX}$ by the arguments above, that is, 
$$\linopQX \tld\Psi-\varphi\in\tanset_{VX}^\perp\ceq\set{f\in L_2(P_{VX}): P_{VX}fs=0\text{ for all } s\in\tanset_{VX}}.$$
If $\linopQX \tld\Psi-\varphi \in\tanset_{VX}$ also, then $\linopQX \tld\Psi-\varphi=0$ a.s.\ must be, from where the assertion follows by the invertibility of $\linopQX$. From the previous display, $\tld\Psi=\linopQXinv(f_0+\varphi)$ for some $f_0\in\tanset_{VX}^\perp.$ Since $\tld\Psi\in\tanset_{VZ}(Q, \mathcal{P}_{VX})$, which set is $\set{\linopQXadj s: s\in \tanset_{VX}}$ by \Cref{priv:lem:semiparametric_vz}, we also have that $(\linopQXadj)^{-1}\tld\Psi\in\tanset_{VX}$.   
But then $P_{VX}f_0(\linopQXadj)^{-1}\tld\Psi=0$ because $f_0\in\tanset_{VX}^\perp$. Whence, $P_{VX}f_0(\linopQXadj)^{-1}\linopQXinv(f_0+\varphi)=0$, or, equivalently,
$$P_{VX}f_0(\linopQXadj)^{-1}\linopQXinv f_0=-P_{VX}f_0(\linopQXadj)^{-1}\linopQXinv\varphi.$$
By the definition of $\linopQX$ and $\linopQXadj$ and by the invertibility of $\linopQX$, one can verify with the tower property of expectations that $(\linopQXadj)^{-1}\linopQXinv=(\linopQX\linopQXadj)^{-1}$ is a positive definite operator, so that the left side of the previous display is zero if and only if $f_0=0$ a.s..\ But the right side is zero because $f_0\in\tanset_{VX}^\perp$, and $(\linopQX\linopQXadj)^{-1}\varphi\in\tanset_{VX}$ by the  assumption $\varphi\in\linopQX\linopQXadj \tanset_{VX}$. Hence, $f_0=0$ a.s., so $\tld\Psi=\linopQXinv\varphi$ a.s..
}
\end{proof}

\subsection{Private Estimation}\label{priv:app:priv:subsec:private_estimation}

This section proves the main result \Cref{priv:cor:psihat_efficiency} of \Cref{priv:sec:estim_dr_priv} via \Cref{priv:lem:chi_eff_empprocess_priv}.


\Cref{priv:lem:norms} in \Cref{priv:app:priv:subsec:aux} establishes some technical results concerning norms under $P_{VX}$ and $P_{VZ}$ and the continuity of the operators $\linopQX,\linopQXinv$. In the light of these technical results, under \Cref{priv:ass:consistent_nuisance_priv}, \Cref{priv:lem:chi_eff_empprocess_priv} shows that the empirical process term in \eqref{priv:eq:chi_eff_decomp_vz} vanishes as $\smallOPs{P_{VZ}}{1}$. It rests on the same arguments as its nonprivate counterpart,  \Cref{priv:lem:chi_eff_empprocess}, but it relies on the continuity of the operator $\linopQXinv$. It is solely because of this that \Cref{priv:ass:consistent_nuisance_priv} is more involved than its nonprivate counterpart, \Cref{priv:ass:consistent_nuisance}, and that it requires that the whole $(V,X)$ be distributed on a finite set, as opposed to only $X$ be finitely distributed. For the proof, note that
\begin{align*}
\partial_\gamma \bar{f}(v,z,\mu,\ba\gamma)\ceq \fr{\partial \bar{f}}{\partial \gamma}(v,z,\mu,\ba\gamma)= \fr{\partial }{\partial \gamma}(\linopQXinv(v,x)\mapsto f(v,x,\mu,\ba\gamma))(v,z)  \\
=(\linopQXinv (v,x)\mapsto \partial_\gamma f(v,x,\mu,\ba\gamma))(v,z),\quad (v,z,\mu,\ba\gamma)\in \Vimagesp\times\Zimagesp\times L_2(P_{V_{\mhit{1}}X})\times\Gamma.
\end{align*}

\begin{lemma}[Vanishing Empirical Process Term --- Private Estimators]
\label{priv:lem:chi_eff_empprocess_priv}
Assume that $\eta$ is estimated by \eqref{priv:eq:nuisance_vz} and \Cref{priv:ass:consistent_nuisance_priv} holds.
Then $\expderivchk-\expderiv=\smallOPs{P_{VZ}}{1}$ and $(\Probnb-P_{VZ})(\hat{\tld{\psi}}-\tld\psi)=\smallOPs{P_{VZ}}{n^{-1/2}}$. 
\end{lemma}

\begin{proof}[Proof of \Cref{priv:lem:chi_eff_empprocess_priv}]
As in the proof of \Cref{priv:lem:chi_eff_empprocess}, we apply that if a random function $\hat q \in L_2(P_{VZ})$ is independent of the random sample generating the process $\Probnb$, then
\begin{align}
\label{priv:eq:meansq_empproc_vz}
\begin{aligned}
\int (\hat{q}(v,z)-q(v,z))^2 \dderiv P_{VZ}(v,z)=\smallOPs{P_{VZ}}{1} 
 \text{ implies } \\
 \sqrt{n}(\Probnb-P_{VZ})(\hat q -q)=\smallOPs{P_{VZ}}{1}.
\end{aligned}
\end{align}
In particular, we shall combine this with \Cref{priv:lem:norms}, establishing the boundedness of $\linopQXinv$ for $\norm{\linopQXinvSpa}{\cdot}$, to show the convergence of
\begin{align}
\norm{L_2(P_{VZ})}{\linopQXinv T}\lesssim \norm{\linopQXinvSpa}{T} \label{priv:eq:meansq_empproc_vz_qinv}
\end{align}
to zero in $P_{VZ}$-probability for some $T \in \linopQXinvSpa$.

By the linearity of $\linopQXinv$, $\hat{\tld{\psi}}-\tld\psi=\linopQXinv(\chk{\tld\chi}-\tld\chi)$. By the definitions \eqref{priv:eq:eif_chi} and \eqref{priv:eq:eif_chi_hat_vz},
\begin{align}
\chk{\tld\chi}(v,x)-\tld\chi(v,x) =&\, \ba T_1(v,x)+\ba T_2(v,x)+\ba T_3(v,x)+\ba T_4, \label{priv:eq:emproc_decomp_vz} \\
\ba T_1(v,x)\ceq&\, \rieszXchk(v_{\mhit{1}},x)(m(v,x)-\muXchk(v_{\mhit{1}},x)) \nonumber \\
&-\rieszX(v_{\mhit{1}},x)(m(v,x)-\muX(v_{\mhit{1}},x)), \nonumber  \\
\ba T_2(v,x)\ceq&\, \fr{\indic{v_{\mhit{2}}=c}}{\chk p_{V_{\mhit{2}}}(c)}(g(v,x)-\gammaVchk(c))\expderivchk -\fr{\indic{v_{\mhit{2}}=c}}{p_{V_{\mhit{2}}}(c)}(g(v,x)-\gammaV(c))\expderiv, \nonumber \\
\ba T_3(v,x)\ceq&\, f(v,x, \muXchk, \gammaVchk(c)) -f(v,x, \muX, \gammaV(c)), \nonumber  \\
\ba T_4\ceq &\, -\psi(\hat P_{VZ})+\psi(P_{VZ}). \nonumber
\end{align}
As $\ba T_4$ is constant, not depending on $(v,x)$, $\linopQXinv \ba T_4=\ba T_4$ and $(\Probnb-P_{VZ}) \linopQXinv \ba T_4=0$. It remains to show $(\Probnb-P_{VZ})\linopQXinv \ba T_j=\smallOPs{P_{VZ}}{n^{-1/2}}$ for $j=1,2,3$ by the linearity of the process $\Probnb-P_{VZ}$.

\emph{Term }$\ba T_1.\,\,$ In the light of \eqref{priv:eq:meansq_empproc_vz_qinv}, $(\Probnb-P_{VZ})\linopQXinv \ba T_1=\smallOPs{P_{VZ}}{n^{-1/2}}$ can be established along the same steps as that of $T_1$ in the proof of \Cref{priv:lem:chi_eff_empprocess}. In particular, $\norm{L_2(P_{VZ})}{\linopQXinv \ba T_1}\lesssim \norm{\linopQXinvSpa}{\ba T_1}$. Namely, suppressing the arguments, write
\begin{align*}
\ba T_1=\rieszXchk(m-\muXchk)-\rieszX(m-\muX) &= (\rieszXchk-\rieszX+\rieszX)(m-\muXchk)-\rieszX(m-\muX) \\
&= (\rieszXchk-\rieszX)(m-\muXchk) + \rieszX(\muX-\muXchk).
\end{align*}
By \Cref{priv:ass:consistent_nuisance_priv}, $\supnorm{m-\muXchk}=\bigOPs{P_{VZ}}{1}$; either directly by 
\eqref{priv:eq:consistency_bigobound_hat_vz}, or by \eqref{priv:eq:consistency_bigobound_vz} and \eqref{priv:eq:consistency_mu_supn_vz}, noting that $\supnorm{m-\muXchk}\leq \supnorm{m-\muX}+\supnorm{\muX-\muXchk}=\bigO{1}+\smallOPs{P_{VZ}}{1}=\bigOPs{P_{VZ}}{1}$.
Then the convergence \eqref{priv:eq:consistency_riesz_l2_vz} of $\rieszX$ implies that $(\Probnb-P_{VZ}) \linopQXinv ((\rieszXchk-\rieszX)(m-\muXchk))=\smallOPs{P_{VZ}}{n^{-1/2}}$ by \eqref{priv:eq:meansq_empproc_vz_qinv} as
\begin{align*}
\norm{\linopQXinvSpa}{(\rieszXchk-\rieszX)(m-\muXchk)}\leq\supnorm{m-\muXchk}\norm{\linopQXinvSpa}{\rieszXchk-\rieszX}=\bigOPs{P_{VZ}}{1}\smallOPs{P_{VZ}}{1}=\smallOPs{P_{VZ}}{1}
\end{align*}
since $\supnorm{|q|^2}=\supnorm{q}^2$. 

By \Cref{priv:ass:consistent_nuisance_priv}, either \eqref{priv:eq:consistency_mu_supn_vz}, or \eqref{priv:eq:consistency_mu_l2_vz} and \eqref{priv:eq:consistency_riesz_bound_vz}. In the former case,
\begin{align*}
\norm{\linopQXinvSpa}{\rieszX(\muX-\muXchk)}\leq\supnorm{\muX-\muXchk}\norm{\linopQXinvSpa}{\rieszX}=\smallOPs{P_{VZ}}{1},
\end{align*}
because $\rieszX\in \linopQXinvSpa$. In the latter case, 
\begin{align*}
\norm{\linopQXinvSpa}{\rieszX(\muX-\muXchk)}\leq \ba R \norm{\linopQXinvSpa}{\muX-\muXchk}=\smallOPs{P_{VZ}}{1},
\end{align*}
since \eqref{priv:eq:consistency_riesz_bound_vz} bounds $\rieszX$ and $\muXchk$ is convergent by \eqref{priv:eq:consistency_mu_l2_vz}. Thus, $(\Probnb-P_{VZ})\linopQXinv$ $(\rieszX (\muX-\muXchk))=\smallOPs{P_{VZ}}{n^{-1/2}}$ by \eqref{priv:eq:meansq_empproc_vz}. Conclude that  $(\Probnb-P_{VZ})\linopQXinv \ba T_1=\smallOPs{P_{VZ}}{n^{-1/2}}$.

\emph{Term }$\ba T_2.\,\,$ By the mean-value theorem there exists $(\gammaVtld(c), \tld p_{V_{\mhit{2}}}(c), \tld\expderiv)$ between $(\gammaV(c), p_{V_{\mhit{2}}}(c), \expderiv)$ and $(\gammaVchk(c), \chk p_{V_{\mhit{2}}}(c), \expderivchk)$ such that
\begin{align*}
\ba T_2(v,x)=&\,\fr{\indic{v_{\mhit{2}}=c}}{\chk p_{V_{\mhit{2}}}(c)}(g(v,x)-\gammaVchk(c))\expderivchk -\fr{\indic{v_{\mhit{2}}=c}}{p_{V_{\mhit{2}}}(c)}(g(v,x)-\gammaV(c))\expderiv \\
=&\,-\fr{\indic{v_{\mhit{2}}=c}}{\tld p_{V_{\mhit{2}}}(c)}\tld\expderiv (\gammaVchk(c)-\gammaV(c)) \\
&-\fr{\indic{v_{\mhit{2}}=c}}{\tld p_{V_{\mhit{2}}}(c)^2}(g(v,x)-\gammaVtld(c))\tld\expderiv(\chk p_{V_{\mhit{2}}}(c)-p_{V_{\mhit{2}}}(c)) \\
&+\fr{\indic{v_{\mhit{2}}=c}}{\tld p_{V_{\mhit{2}}}(c)}(g(v,x)-\gammaVtld(c))(\expderivchk-\expderiv).
\end{align*}
The standard central limit theorem applies to the i.i.d.\ sequence
$$\bigg((\linopQXinv(v,x)\mapsto\indic{v_{\mhit{2}}=c})(V_i,Z_i)\bigg)_{i\in[n]},$$
hence $\sqrt{n}(\Probnb-P_{VZ})((\linopQXinv(v,x)\mapsto\indic{v_{\mhit{2}}=c})(V,Z))=\bigOPs{P_{VZ}}{1}$. By the linearity of the process $\sqrt{n}(\Probnb-P_{VZ}) \linopQXinv$,
\begin{align*}
\sqrt{n}(\Probnb-P_{VZ})\big\{\big[\linopQXinv(v,x)\mapsto\indic{v_{\mhit{2}}=c}g(v,x)-\indic{v_{\mhit{2}}=c}\gammaVtld(c)\big](V,Z)\big\} \\
= \sqrt{n}(\Probnb-P_{VZ})\big\{\big[\linopQXinv (v,x)\mapsto \indic{v_{\mhit{2}}=c}g(v,x)\big](V,Z)\big\} \\
- \gammaVtld(c)\sqrt{n}(\Probnb-P_{VZ})\big\{\big[\linopQXinv (v,x)\mapsto\indic{v_{\mhit{2}}=c}\big](V,Z)\big\} \\
=(1-\gammaVtld(c))\bigOPs{P_{VZ}}{1}=\bigOPs{P_{VZ}}{1}
\end{align*} 
again by the standard central limit theorem and \eqref{priv:eq:consistency_gammaVhat_vz}. Suppose that $\expderivchk-\expderiv=\smallOPs{P_{VZ}}{1}$, which we show later. Then by \eqref{priv:eq:consistency_gammaVhat_vz} and \eqref{priv:eq:consistency_phat_vz}, $(\Probnb-P_{VZ})\linopQXinv \ba T_2=\smallOPs{P_{VZ}}{n^{-1/2}}$.

\emph{Term }$\ba T_3.\,\,$ Recall that $\ba T_3(v,x)=f(v,x, \muXchk, \gammaVchk(c))-f(v,x, \muX, \gammaV(c))$. By the consistency of $\gammaVchk$ and $\muXchk$ (\eqref{priv:eq:consistency_gammaVhat_vz} and \eqref{priv:eq:consistency_mu_supn_vz} or \eqref{priv:eq:consistency_mu_l2_vz}), we have $\norm{\linopQXinvSpa}{\ba T_3}=\smallOPs{P_{VZ}}{1}$ by the continuous mapping theorem and \eqref{priv:eq:f_sq_continu_priv}. Conclude by \eqref{priv:eq:meansq_empproc_vz} and \eqref{priv:eq:meansq_empproc_vz_qinv}  that  $(\Probnb-P_{VZ})\linopQXinv \ba T_3=\smallOPs{P_{VZ}}{n^{-1/2}}$.

\emph{Consistency of  $\expderivchk.\,\,$} By the definition of $\expderiv,\expderivchk$,
\begin{align*}
\expderivchk-\expderiv =&\, \Probnb''\partial_\gamma \ba f(V,Z,\muXchk,\gammaVchk(c)) - P_{VZ}\partial_\gamma \ba f(V,Z,\muX,\gammaV(c))   \\
=&\, \Probnb''\partial_\gamma \ba f(V,Z,\muXchk,\gammaVchk(c))-P_{VZ}\partial_\gamma \ba f(V,Z,\muXchk,\gammaVchk(c)) \\
&+ P_{VZ}\partial_\gamma \ba f(V,Z,\muXchk,\gammaVchk(c))- P_{VZ}\partial_\gamma f(V,Z,\muX,\gammaV(c)) \\
=&\, (\Probnb''-P_{VZ})\partial_\gamma \ba f(V,Z,\muXchk,\gammaVchk(c)) \\
&+ P_{VZ}\big[\partial_\gamma \ba f(V,Z,\muXchk,\gammaVchk(c))-\partial_\gamma \ba f(V,Z,\muX,\gammaV(c))  \big] \\
=&\, (\Probnb''-P_{VZ})\partial_\gamma \ba f(V,Z,\muX,\gammaV(c)) \\
&+(\Probnb''-P_{VZ})\big[\partial_\gamma \ba f(V,Z,\muXchk,\gammaVchk(c))-\partial_\gamma \ba f(V,Z,\muX,\gammaV(c))\big] \\
&+ P_{VZ}\big[\partial_\gamma \ba f(V,Z,\muXchk,\gammaVchk(c))-\partial_\gamma \ba f(V,Z,\muX,\gammaV(c))  \big]
\end{align*}
Here, the first term is $\bigOPs{P_{VZ}}{n^{-1/2}}=\smallOPs{P_{VZ}}{1}$ by the standard central limit theorem, and the second and third term are $\smallOPs{P_{VZ}}{1}$ by the continuity \eqref{priv:eq:f_diffcontinu_priv} of $\partial_\gamma f$ using that $\partial_\gamma \ba f =\linopQXinv \partial_\gamma f$ along the same arguments concerning $\ba T_3$ above.
\end{proof}

\begin{proof}[Proof of \Cref{priv:cor:psihat_efficiency}]
Follows from \Cref{priv:lem:chi_eff_empprocess_priv,priv:thm:dr}, noting that, for $\ba\expderiv'$ in \eqref{priv:eq:derivhatprime_vz}, 
$$\expderivchk-\ba\expderiv'=(\Probnb''-P_{VZ})\partial_\gamma \ba f(V,Z,\muXhat,\gammaVhat(c))$$ is $\smallOPs{P_{VZ}}{1}$ by the consistency proof of $\expderivchk$ in \Cref{priv:lem:chi_eff_empprocess_priv}.
\end{proof}

\subsection{Auxiliary Results}
\label{priv:app:priv:subsec:aux}

In this section, auxiliary results underpinning the proofs in \Cref{priv:app:priv} are derived: in \Cref{priv:lem:linopqx_properties}, the properties of the linear operator $\linopQX$; in \Cref{priv:lem:qtc_implications}, the distributions of the (partly) unobserved data $(V,X,Z)$ under mechanism \eqref{priv:eq:qtv}; in \Cref{priv:lem:norms}, norms under $P_{VX}$ and $P_{VZ}$, and consequent continuity of $\linopQX,\linopQXinv$.

The following properties of $\linopQX$ play an essential role in the derivation of the tangent set and the efficient influence function.
\begin{lemma}[Properties of $\linopQX$ in \eqref{priv:eq:linopQX}]
\label{priv:lem:linopqx_properties}
\begin{enumerate}[label=(\roman*)]
\item \label{priv:lem:linopqx_properties_exp} The operator $\linopQX$ is the conditional expectation operator $(\linopQX k)(v,x)=\Ebc{k(V,Z)}{V=v,X=x}$, $(v,x)\in\Vimagesp\times\Ximagesp$.
\item \label{priv:lem:linopqx_properties_adj}  The operator $\linopQX$ has adjoint $\linopQXadj:L_2(P_{VX})\to L_2(P_{VZ})$, 
$(\linopQXadj h)(v,z)=\Ebc{h(V,X)}{V=v,Z=z}$, $(v,z)\in\Vimagesp\times\Zimagesp$.
\item \label{priv:lem:linopqx_properties_change} Change of measure: 
$P_{VZ}k = P_{VX}\linopQX k$ for all $k\in L_2(P_{VZ})$. In particular, if $\linopQX: L_2(P_{VZ})\to S \subset L_2(P_{VX})$ has an inverse $\linopQXinv: S\to L_2(P_{VZ})$ so that  $\linopQX\linopQXinv h=h$ for all $h\in S$, then $k\ceq \linopQXinv h$ yields $P_{VZ}\linopQXinv h=P_{VX}\linopQX \linopQXinv h=P_{VX}h.$

\item \label{priv:lem:linopqx_properties_exp_inv_disc} If $X$ is distributed on a finite set with $|\Zimagesp|=|\Ximagesp|=J$, then the operator $\linopQX: L_2(P_{VZ})\to L_2(P_{VX})$ of \eqref{priv:eq:linopQX} can be represented in the matrix notation \eqref{priv:eq:q_as_matrix} as, for all $v\in\Vimagesp$,
\begin{align*}
\begin{bmatrix}
(\linopQX k)(v,x_1) \\
(\linopQX k)(v,x_2) \\
\vdots \\
(\linopQX k)(v,x_J)
\end{bmatrix}
=Q^\intercal \begin{bmatrix}
k(v,z_1) \\
k(v,z_2) \\
\vdots \\
k(v,z_J)
\end{bmatrix},
\end{align*}
and it has inverse $\linopQXinv:L_2(P_{VX})\to L_2(P_{VZ})$  if and only if $Q$ is invertible, given by, for all $v\in\Vimagesp$,
\begin{align*}
\begin{bmatrix}
k(v,z_1) \\
k(v,z_2) \\
\vdots \\
k(v,z_J)
\end{bmatrix} =
(Q^\intercal)^{-1}
\begin{bmatrix}
h(v,x_1) \\
h(v,x_2) \\
\vdots \\
h(v,x_J)
\end{bmatrix} 
\end{align*}
with $(Q^\intercal)^{-1}=(Q^{-1})^\intercal$.
\item \label{priv:lem:linopqx_properties_exp_qtc} Let $Q\in\setQgen$ in \eqref{priv:eq:qtv}. Then $(\linopQX k)(v,x)=\alpha k(v,x)+(1-\alpha)\int_\Xsupp k(v,z)\Qbar(\deriv z)$; moreover, $\linopQX:L_2(P_{VZ})\to \linopQXinvDom$ is a bounded, hence continuous, linear operator for the norm
\begin{align}
\norm{\linopQXinvDom}{h}\ceq\norm{L_2(P_{VX})}{h}+\norm{L_2(P_V\otimes \Qbar)}{h}\label{priv:eq:normqbar}
\end{align}
on $\linopQXinvDom$, where $P_V\otimes \Qbar$ is the distribution of a random element $(V,\ba Z)$ with independent coordinates $V\sim P_V$ and $\ba Z\sim\Qbar$.
\item \label{priv:lem:linopqx_properties_inv_qtc} Let $Q\in\setQgen$ in \eqref{priv:eq:qtv}. The inverse of $\linopQX: L_2(P_{VZ})\to \linopQXinvDom$ exists, and is, as in \citet[Section 4.9-1., Equation 1]{polanin_handbook_1998}, 
$$(\linopQXinv h)(v,z)=\fr{1}{\alpha}h(v,z)-\fr{1-\alpha}{\alpha}\int_\Ximagesp h(v,x)\Qbar (\deriv x)$$ for all $h \in \linopQXinvDom.$
That is, $\linopQX\linopQXinv h=h$ and $\linopQXinv\linopQX k$ $=k$ for all $h\in \linopQXinvDom$ and all $k\in L_2(P_{VZ})$. Moreover $\linopQXinv:\linopQXinvDom\to L_2(P_{VZ})$ is a bounded, hence continuous, linear operator for the norm \eqref{priv:eq:normqbar}.
\item \label{priv:lem:linopqx_properties_bounds} Let $Q\in\setQgen$ in \eqref{priv:eq:qtv}. Then for all $h\in \linopQXinvDom$,
\begin{align*}
P_{VX}h^2+\fr{1-\alpha}{\alpha}\left((P_V\otimes\Qbar)h^2- P_V\left(\int h(V,x)\Qbar(\deriv x)\right)^2\right) \leq P_{VZ}(\linopQXinv h)^2 \\ \leq \fr{2-\alpha}{\alpha} P_{VX}\tld\chi^2+2\fr{(1-\alpha)(2-\alpha)}{\alpha^2}(P_V\otimes\Qbar)h^2.
\end{align*}
The second term in the lower bound is nonnegative due to Jensen's inequality.
\end{enumerate}
\end{lemma}

\begin{proof}[Proof of \Cref{priv:lem:linopqx_properties}]
Assertion \ref{priv:lem:linopqx_properties_exp}. Follows directly from $Z\mid(V,X)\sim Q(\cdot\cond X)$.

Assertion \ref{priv:lem:linopqx_properties_adj}. By definition, the adjoint  $\linopQXadj$ of $\linopQX$ satisfies $P_{VX}[(\linopQX k) h]=P_{VZ}[k\linopQXadj h]$. By the tower property of expectation we can verify that, for $\linopQXadj$ given in \ref{priv:lem:linopqx_properties_adj},
\begin{align*}
P_{VX}[(\linopQX k) h]=\Eb{\Ebc{k(V,Z)}{V,X}h(V,X)}=\Eb{\Ebc{k(V,Z)h(V,X)}{V,X}} \\
=\Eb{\Ebc{k(V,Z)h(V,X)}{V,Z}}=\Eb{k(V,Z)\Ebc{h(V,X)}{V,Z}}=P_{VZ}[k\linopQXadj h].
\end{align*}

Assertion \ref{priv:lem:linopqx_properties_change}. By the tower property of expectation, 
$$P_{VZ}k=\E k(V,Z) =\E \Ebc{k(V,Z)}{V,X} = P_{VX}\linopQX k.$$

Assertion \ref{priv:lem:linopqx_properties_exp_inv_disc}. Under the discrete model for $X$, 
$$\int_\Ximagesp k(v,z)Q(\deriv z\cond x)=\sum_{z\in\Ximagesp}k(v,z)Q(\set{z}\cond x),$$
from which the assertion directly follows.

Assertion \ref{priv:lem:linopqx_properties_exp_qtc}. Under the mechanism $Q$ in \eqref{priv:eq:qtv}, we have
\begin{align*}
(\linopQX k)(v,x)= \int_\Ximagesp k(v,z)Q(\deriv z\cond x) = \int_\Ximagesp k(v,z)\left[\alpha\delta_x(\deriv z)+(1-\alpha)\Qbar(\deriv z)\right] \\
= \alpha k(v,x)+(1-\alpha)\int_\Ximagesp k(v,z) \Qbar(\deriv z).
\end{align*}
Next we show that $\linopQX: L_2(P_{VZ})\to \linopQXinvDom$ is bounded. Take some $k\in L_2(P_{VZ})$, and let $(\linopQX k)(v,x)= \alpha k(v,x)+(1-\alpha)\int_\Ximagesp k(v,z) \Qbar(\deriv z)\eqc \phi_1(v,x)+\phi_2(v)$. By Minkowski's inequality, 
$$\norm{\linopQXinvDom}{\linopQX k}\leq \norm{\linopQXinvDom}{\phi_1}+\norm{\linopQXinvDom}{\phi_2},$$
where, by definition \eqref{priv:eq:normqbar},
$$\norm{\linopQXinvDom}{\phi_1}=\norm{L_2(P_{VX})}{\phi_1}+\norm{L_2(P_V\otimes \Qbar)}{\phi_1}$$
with $\norm{L_2(P_{VX})}{\phi_1}=\alpha\norm{L_2(P_{VX})}{k}\leq \sqrt{\alpha}\norm{L_2(P_{VZ})}{k}$ by \Cref{priv:lem:qtc_implications}\ref{priv:lem:qtc_implications_norms}, and
$$\norm{\linopQXinvDom}{\phi_2}=2\norm{L_2(P_{V})}{\phi_2}.$$
By Jensen's inequality, 
$$\left(\int_\Ximagesp k(v,z) \Qbar(\deriv z)\right)^2\leq \int_\Ximagesp k(v,z)^2 \Qbar(\deriv z),$$
thus
\begin{align*}
 \norm{L_2(P_{V})}{\phi_2}\leq (1-\alpha)\sqrt{\int_\Vimagesp \int_\Ximagesp k(v,z)^2 \Qbar(\deriv z)P_V(\deriv v)}\\ =(1-\alpha)\sqrt{\int_\Vimagesp \int_\Ximagesp k(v,z)^2 \qbar(z) \mesX(\deriv z) p_V(v)\mesV(\deriv v)} \\
 \leq \sqrt{(1-\alpha)} \sqrt{\int_\Vimagesp \int_\Ximagesp k(v,z)^2 p_{VZ}(v,z) \mesX(\deriv z)\mesV(\deriv v)} =  \sqrt{(1-\alpha)} \norm{L_2(P_{VZ})}{k},
\end{align*}
where we used that by \Cref{priv:lem:qtc_implications} \ref{priv:lem:qtc_implications_dens}, 
$$\qbar(z)p_V(v)=\fr{1}{1-\alpha}p_{VZ}(v,z)-\fr{\alpha}{1-\alpha}p_{VX}(v,z)\leq \fr{1}{1-\alpha}p_{VZ}(v,z)$$ for $\alpha \in(0,1)$. 
Finally, by the previous display, $\norm{L_2(P_V\otimes \Qbar)}{\phi_1}\leq \sqrt{\fr{1}{1-\alpha}}\norm{L_2(P_{VZ})}{\phi_1}$, yielding 
\begin{align*}
\norm{\linopQXinvDom}{\linopQX k} \\
 \leq \sqrt{\alpha}\norm{L_2(P_{VZ})}{k} + \fr{\alpha}{\sqrt{1-\alpha}}\norm{L_2(P_{VZ})}{k}+2 \sqrt{(1-\alpha)} \norm{L_2(P_{VZ})}{k}.
\end{align*}

Assertion \ref{priv:lem:linopqx_properties_inv_qtc}. It suffices to show $\linopQX(\linopQXinv h)=h$ for all $h\in \linopQXinvDom$. Under \eqref{priv:eq:qtv},
\begin{align*}
(\linopQX(\linopQXinv h))(v,x) = \alpha (\linopQXinv h)(v,x)+(1-\alpha) \int_{\Ximagesp} (\linopQXinv h)(v,z)\Qbar(\deriv z).
\end{align*}
Here, the first term is
\begin{align*}
\alpha (\linopQXinv h)(v,x) &= \alpha \left\{ \fr{1}{\alpha}h(v,x)-\fr{1-\alpha}{\alpha}\int_\Ximagesp h(v,t)\Qbar (\deriv t)\right\}  \\
&=h(v,x) -(1-\alpha)\int_\Ximagesp h(v,t)\Qbar (\deriv t),
\end{align*}
while the second term is
\begin{align*}
(1-\alpha) \int_{\Ximagesp} (\linopQXinv h)(v,z)\Qbar(\deriv z)\\
 = (1-\alpha) \int_{\Ximagesp} \left\{\fr{1}{\alpha}h(v,z)-\fr{1-\alpha}{\alpha}\int_\Ximagesp h(v,t)\Qbar (\deriv t)\right\} \Qbar(\deriv z) \\
=(1-\alpha)\left\{ \fr{1}{\alpha}\int_{\Ximagesp} h(v,z)\Qbar(\deriv z) - \fr{1-\alpha}{\alpha}\int_\Ximagesp h(v,t)\Qbar (\deriv t) \right\} \\
= (1-\alpha)\int_\Ximagesp h(v,t)\Qbar (\deriv t),
\end{align*}
where we used that $\Qbar$ is a probability measure with $\int_{\Ximagesp}\Qbar(\deriv z)=1$. Collecting terms, conclude that $\linopQX(\linopQXinv h)=h$. 
We showed that $\linopQXinv$ is the inverse of $\linopQX$, which is by  \ref{priv:lem:linopqx_properties_exp_qtc} a bounded operator. 

To see that $\linopQXinv: \linopQXinvDom\to L_2(P_{VZ})$ is bounded, take some $h\in\linopQXinvDom$. As $(\linopQXinv h)(v,z)=\fr{1}{\alpha}h(v,z)-\fr{1-\alpha}{\alpha}\int_\Ximagesp h(v,x)\Qbar(\deriv x)$, we have
\begin{align*}
\norm{L_2(P_{VZ})}{\linopQXinv h}\leq \fr{1}{\alpha}\norm{L_2(P_{VZ})}{h}+\fr{1-\alpha}{\alpha}\norm{L_2(P_V)}{\int_\Ximagesp h(\cdot,x)\Qbar(\deriv x)},
\end{align*}
where, by \Cref{priv:lem:qtc_implications} \ref{priv:lem:qtc_implications_jointobs},
\begin{align*}
\norm{L_2(P_{VZ})}{h}=\sqrt{\int h^2\dderiv [\alpha P_{VX}+(1-\alpha) P_V \otimes \Qbar]} \\
\leq \sqrt{\norm{L_2(P_{VX})}{h}^2 + \norm{L_2(P_V \otimes \Qbar)}{h}^2}  \leq \sqrt{(\norm{L_2(P_{VX})}{h} + \norm{L_2(P_V \otimes \Qbar)}{h})^2} \\ \leq \norm{L_2(P_{VX})}{h} + \norm{L_2(P_V \otimes \Qbar)}{h} = \norm{\linopQXinvDom}{h}.
\end{align*}
By Jensen's inequality, 
\begin{align*}
\norm{L_2(P_V)}{\int_\Ximagesp h(\cdot,x)\Qbar(\deriv x)} \\ \leq \sqrt{\int h^2 \dderiv P_V\otimes\Qbar}=\norm{L_2(P_V \otimes \Qbar)}{h}\leq \norm{\linopQXinvDom}{h},
\end{align*}
whereby
\begin{align*}
\norm{L_2(P_{VZ})}{\linopQXinv h}\leq \fr{1}{\alpha}\norm{\linopQXinvDom}{h} +\fr{1-\alpha}{\alpha}\norm{\linopQXinvDom}{h}.
\end{align*}

Assertion \ref{priv:lem:linopqx_properties_bounds}. By \ref{priv:lem:linopqx_properties_inv_qtc},
\begin{align}
(\linopQXinv h)^2(v,z)= \fr{1}{\alpha^2}h^2(v,z)-2\fr{1-\alpha}{\alpha^2} h(v,z)\int h(v,x)\Qbar(\deriv x)+\left(\fr{1-\alpha}{\alpha}\int h(v,x)\Qbar(\deriv x)\right)^2. \label{priv:eq:qinv_sqr}
\end{align}
Here, because $xy\leq (x^2+y^2)/2$ for all $x,y\in\real$, 
\begin{align}
\bigg\vert h(v,z)\int h(v,x)\Qbar(\deriv x) \bigg\vert \leq \fr{1}{2}\left( h^2(v,z)+\left(\int h(v,x)\Qbar(\deriv x)\right)^2\right). \label{priv:eq:hproduct_bound}
\end{align}
Then \eqref{priv:eq:qinv_sqr} yields the lower bound
\begin{align*}
P_{VZ}(\linopQXinv h)^2\geq \fr{1}{\alpha^2}P_{VZ}h^2- \fr{1-\alpha}{\alpha^2}\left(P_{VZ}h^2+P_V \left(\int h(V,x)\Qbar(\deriv x)\right)^2 \right) \\
+P_V\left(\fr{1-\alpha}{\alpha}\int h(V,x)\Qbar(\deriv x)\right)^2 = \fr{1}{\alpha}P_{VZ}h^2-\fr{1-\alpha}{\alpha}P_V \left(\int h(V,x)\Qbar(\deriv x)\right)^2
\end{align*}
Here, $P_{VZ}=\alpha P_{VX}+(1-\alpha)P_V\otimes\Qbar$ by \Cref{priv:lem:qtc_implications}, which proves the lower bound. For the upper bound, \eqref{priv:eq:qinv_sqr} and \eqref{priv:eq:hproduct_bound} give
\begin{align*}
P_{VZ}(\linopQXinv h)^2\leq \fr{1}{\alpha^2}P_{VZ}h^2+ \fr{1-\alpha}{\alpha^2}\left(P_{VZ}h^2+P_V \left(\int h(V,x)\Qbar(\deriv x)\right)^2 \right) \\
+P_V\left(\fr{1-\alpha}{\alpha}\int h(V,x)\Qbar(\deriv x)\right)^2 = \fr{2-\alpha}{\alpha^2} P_{VZ}h^2 + \fr{(2-\alpha)(1-\alpha)}{\alpha^2} P_V \left(\int h(V,x)\Qbar(\deriv x)\right)^2.
\end{align*}
Then $P_{VZ}=\alpha P_{VX}+(1-\alpha)P_V\otimes\Qbar$ and Jensen's inequality $P_V \left(\int h(v,x)\Qbar(\deriv x)\right)^2\leq (P_V\otimes\Qbar)h^2$ prove the upper bound.
\end{proof}


By the construction of $Z_i$ in \Cref{priv:sec:priv:subsec:priv_mech}, the sequence $((V_i,X_i,Z_i))_{i\in[n]}$ is an i.i.d.\ sample from the distribution of the partly unobserved data $(V,X,Z)$,
\begin{align}
P_{VXZ}(\bsetV, \bsetX, \bsetZ) &= \int \indic{(v,x)\in \bsetV\times \bsetX}Q(\bsetZ\cond x) \dderiv P_{VX}(v,x)\label{priv:eq:distribution_vxz}
\end{align}
for $\bsetV\in\Vsigma, \bsetX \in\Xsigma, \bsetZ\in\Zsigma$. When $Q\in\setQgen$ in \eqref{priv:eq:qtv}, this distribution is as follows.

\begin{lemma}[Distributions under \eqref{priv:eq:qtv}]
\label{priv:lem:qtc_implications}
If the mechanism $Q\in\setQgen$ in \eqref{priv:eq:qtv}, then:
\begin{enumerate}[label=(\roman*)]
\item\label{priv:lem:qtc_implications_jointall} The joint distribution of $(V,X,Z)$ is $P_{VXZ}(\bsetV,\bsetX,\bsetZ)=\alpha P_{VX}(\bsetV,\bsetX\cap \bsetZ)+(1-\alpha)P_{VX}(\bsetV,\bsetX)\Qbar(\bsetZ)$ for $\bsetV\in\Vsigma, \bsetX \in\Xsigma, \bsetZ\in\Zsigma$.

\item\label{priv:lem:qtc_implications_jointobs} The joint distribution of $(V,Z)$ is $P_{VZ}(\bsetV,\bsetZ)=\alpha P_{VX}(\bsetV,\bsetZ)+(1-\alpha)P_{V}(\bsetV)\Qbar(\bsetZ)$ for $\bsetV\in\Vsigma, \bsetZ\in\Zsigma$. Hence, we have the absolute-continuity relations $P_{VX}\ll P_{VZ}$, so $P_X\ll P_Z$, and $\Qbar\ll P_Z$.

\item \label{priv:lem:qtc_implications_dens} If $P_{VX}$ has density $p_{VX}$ with respect to some measure $\mesV\times\mesX$ and $\ba Q$ has density $\qbar$ with respect to $\mesX$, then $P_{VZ}$ has density $p_{VZ}(v,z)\ceq \alpha p_{VX}(v,z)+(1-\alpha)p_{V}(v) \qbar(z) $ for $(v,z)\in\Vimagesp\times\Ximagesp$ with respect to $\mesV\times\mesX$.

\item \label{priv:lem:qtc_implications_norms} For all $h:\Vimagesp\times\Ximagesp\to\real$ and all $p\in[1,\infty]$, $\norm{L_p(P_{VX})}{h}\leq \alpha^{-1/p}\norm{L_p(P_{VZ})}{h}$.

\item \label{priv:lem:qtc_implications_cond_z} The Markov kernel 
\begin{align*}
P_{X\mid Z}(B\cond z) &\ceq \alpha\fr{\deriv P_X}{\deriv P_{Z}}(z)\delta_z(B)+(1-\alpha)P_X(B)\fr{\deriv \ba Q}{\deriv P_{Z}}(z),\quad B\in\Xsigma, z\in\Ximagesp,
\end{align*}
is the conditional distribution of $X$ given $Z=z$, where the Radom-Nykod\'ym derivatives $\fr{\deriv P_X}{\deriv P_{Z}}$, $\fr{\deriv \ba Q}{\deriv P_{Z}}$ exist by \ref{priv:lem:qtc_implications_jointobs}. If $P_{VX}$ and $\ba Q$ have densities $p_{VX}$ and $\qbar$ with respect to $\mesV\times\mesX$ and $\mesX$, respectively, then for $p_Z$ induced by \ref{priv:lem:qtc_implications_jointobs}, the last display is equal to
\begin{align*}
P_{X\mid Z}(B\cond z)=\alpha\fr{p_X(z)}{p_Z(z)}\delta_z(B)+(1-\alpha)\fr{\qbar(z)}{p_Z(z)}P_X(B), \quad B\in\Xsigma, z\in\Ximagesp.
\end{align*}

\item \label{priv:lem:qtc_implications_cond_v} Suppose that $X$ given $V=v$, $v\in\Vimagesp$, admits a conditional distribution $P_{X\cond V}(\cdot\cond v)$. Then the Markov kernel
\begin{align*}
P_{Z\cond V}(B\cond v)\ceq \alpha P_{X\cond V}(B\cond v)+(1-\alpha)\ba Q(B),\quad B\in\Xsigma, v\in\Vimagesp,
\end{align*}
is the conditional distribution of $Z$ given $V=v$. Hence, $\Qbar\ll P_{Z\cond V}(\cdot\cond v)$ for any $v\in\Vimagesp$.

\item \label{priv:lem:qtc_implications_cond_vz} Suppose that $X$ given $V=v$, $v\in\Vimagesp$, admits a conditional distribution $P_{X\cond V}(\cdot\cond v)$. For each $v\in\Vimagesp$, let $\qbar(z\cond v)\ceq \fr{\deriv \Qbar}{\deriv P_{Z\cond V}(\cdot\cond v)}(z)$ be the Radom-Nykod\'ym derivative of $\Qbar$ with respect to $P_{Z\cond V}(\cdot\cond v)$, which exists by \ref{priv:lem:qtc_implications_cond_v}. Then the Markov kernel
\begin{align*}
P_{X\cond VZ}(B\cond v,z)\ceq \alpha \fr{\deriv P_{VX}}{\deriv P_{VZ}}(v,z)\delta_z(B) + (1-\alpha) \qbar(z\cond v) P_{X\cond V}(B\cond v),
\end{align*}
for $B\in\Xsigma, (v,z)\in\Vimagesp\times\Ximagesp$, is the conditional distribution of $X$ given $(V,Z)=(v,z)$.

\end{enumerate}
\end{lemma}

\begin{proof}[Proof of \Cref{priv:lem:qtc_implications}]

Assertion \ref{priv:lem:qtc_implications_jointall}. Plug \eqref{priv:eq:qtv} into \eqref{priv:eq:distribution_vxz}, using the definition of the Dirac measure $\delta_x(B)=\indic{x\in B}$, to find
\begin{align*}
P_{VXZ}(\bsetV,\bsetX,\bsetZ)=&\ \alpha \int \indic{(v,x)\in \bsetV\times \bsetX} \indic{x\in \bsetZ} \dderiv P_{VX}(v,x) \\
&+(1-\alpha) \Qbar(\bsetZ)\int\indic{(v,x)\in \bsetV\times \bsetX} \dderiv P_{VX}(v,x) \\
=&\, \alpha \int\indic{(v,x)\in \bsetV\times (\bsetX\cap\bsetZ)}\dderiv P_{VX}(v,x)\\
& +(1-\alpha) \Qbar(\bsetZ)\int \indic{(v,x)\in \bsetV\times \bsetX} \dderiv P_{VX}(v,x) \\
 =&\, \alpha P_{VX}(\bsetV,\bsetX\cap \bsetZ)+(1-\alpha)\Qbar(\bsetZ)P_{VX}(\bsetV,\bsetX).
\end{align*}

Assertion \ref{priv:lem:qtc_implications_jointobs}. Follows from \ref{priv:lem:qtc_implications_jointall} as the marginal distribution by setting $\bsetX\ceq\Ximagesp$.

Assertion \ref{priv:lem:qtc_implications_dens}. A convex combination of two densities, $p_{VZ}$ is nonnegative. From \ref{priv:lem:qtc_implications_jointobs}, write $P_{VZ}(\bsetV,\bsetZ)$ as
\begin{align*}
\alpha \int_{\bsetV}\int_{\bsetZ} p_{VX}(v,x)\dderiv \mesX(x) \dderiv \mesV(v)+(1-\alpha)\int_{\bsetZ}\qbar(z)\dderiv \mesX(x)\int_{\bsetV} p_V(v)\dderiv\mesV(v) \\
= \int_{\bsetV}\int_{\bsetZ} \big\{ \alpha p_{VX}(v,x)+(1-\alpha)\qbar(x) p_V(v) \big\}\dderiv \mesX(x)\dderiv\mesV(v).
\end{align*}

Assertion \ref{priv:lem:qtc_implications_norms}. By \ref{priv:lem:qtc_implications_jointobs}, $P_{VX}\leq \alpha^{-1}P_{VZ}$. Hence, 
\begin{align*}
\norm{L_p(P_{VX})}{h}= \left(\int_{\Ximagesp} \int_{\Vimagesp} |h(v,x)|^p \dderiv P_{VX}(v,x) \right)^{1/p}\leq \alpha^{-1/p} \norm{L_p(P_{VZ})}{h}.
\end{align*}

Assertion \ref{priv:lem:qtc_implications_cond_z}. It is sufficient and necessary to verify that for the $P_{X|Z}$ given in \ref{priv:lem:qtc_implications_cond_z}, $P_{XZ}(\bsetX,\bsetZ)=\int_{\bsetZ} P_{X|Z}(\bsetX\cond z)\dderiv P_Z(z)$. From the right,
\begin{align*}
\int_{\bsetZ} P_{X|Z}(\bsetX\cond z) \dderiv P_Z(z) \\
 = \int_{\bsetZ} \bigg\{ \alpha\fr{\deriv P_X}{\deriv P_{Z}}(z)\delta_z(\bsetX)+(1-\alpha)P_X(\bsetX)\fr{\deriv \ba Q}{\deriv P_{Z}}(z) \bigg\}\dderiv P_Z(z) \\
 = \alpha \int_{\bsetZ\cap\bsetX}\fr{\deriv P_X}{\deriv P_{Z}}(z) \dderiv P_Z(z) +(1-\alpha)P_X(\bsetX)\int_{\bsetZ}\fr{\deriv \ba Q}{\deriv P_{Z}}(z)\dderiv P_Z(z) \\
 = \alpha P_X(\bsetX\cap\bsetZ)+(1-\alpha)P_X(\bsetX) \Qbar(\bsetZ),
\end{align*}
in which we recognise $P_{XZ}(\bsetX,\bsetZ)$ by \ref{priv:lem:qtc_implications_jointall}.

Assertion \ref{priv:lem:qtc_implications_cond_vz}. Follows from the same arguments as \ref{priv:lem:qtc_implications_cond_z}.

Assertion \ref{priv:lem:qtc_implications_cond_vz}. Follows from the same arguments as \ref{priv:lem:qtc_implications_cond_z}, by verifying that for the $P_{X\cond VZ}$ given in the assertion, we have $P_{VXZ}(\bsetV,B,\bsetZ)=\int_{\bsetV\times \bsetZ}P_{X\cond VZ}(B\cond v,z)\dderiv P_{VZ}(v,z)$ for $P_{VXZ}$ in \ref{priv:lem:qtc_implications_jointall} and for all $\bsetV\in\Vsigma$, $B\in\Xsigma$, $\bsetZ\in\Xsigma$. Specifically, noting that $\deriv P_{VZ}(v,z)=P_{Z\cond V}(\deriv z\cond v)\deriv P_{V}(v)$ by definition of the conditional distribution $P_{Z\cond V}$, we have from the right,
\begin{align*}
\int_{\bsetV\times \bsetZ} \left\{ \alpha \fr{\deriv P_{VX}}{\deriv P_{VZ}}(v,z)\delta_z(B) + (1-\alpha) \qbar(z\cond v) P_{X\cond V}(B\cond v) \right\} \dderiv P_{VZ}(v,z) \\
= \alpha \int_{\bsetV\times (\bsetZ\cap B)}\fr{\deriv P_{VX}}{\deriv P_{VZ}}(v,z) \dderiv P_{VZ}(v,z)\\
+(1-\alpha)\int_{\bsetV}P_{X\cond V}(B\cond v)\int_{\bsetZ}\fr{\deriv \Qbar}{\deriv P_{Z\cond V}(\cdot\cond v)}(z)P_{Z\cond V}(\deriv z\cond v)\deriv P_{V}(v)\\
=\alpha P_{VX}(\bsetV,\bsetZ\cap B)+(1-\alpha)\Qbar(\bsetZ)\int_{\bsetV}P_{X\cond V}(B\cond v)\deriv P_V(v) \\
=\alpha P_{VX}(\bsetV,\bsetZ\cap B)+(1-\alpha)\Qbar(\bsetZ) P_{VX}(\bsetV,B),
\end{align*}
in which we recognise $P_{VXZ}(\bsetV,B,\bsetZ)$ in \ref{priv:lem:qtc_implications_jointall}.
\end{proof}


\Cref{priv:lem:norms} is key in proving \Cref{priv:lem:chi_eff_empprocess_priv}.

\begin{lemma}[Norms and continuity of $\linopQX,\linopQXinv$]
\label{priv:lem:norms}
With $c_Q$, we denote strictly positive constants which depend only on $Q$ and whose value may differ in every display in this lemma.

If $Q\in\setQgen$ in \eqref{priv:eq:qtv}, or $Q\in \setQdisc$ in \eqref{priv:eq:discreteQdef} with $\min_{(z,x)\in\Ximagesp^2}Q(\set{z}\cond x)>0$, then for all $h:\Vimagesp\times\Ximagesp\to\real$ and all $p\in[1,\infty]$,
\begin{align}
\norm{L_p(P_{VX})}{h}\leq c_Q\norm{L_p(P_{VZ})}{h}. \label{priv:lem:norms:eq:lp}
\end{align}

If $Q\in\setQdiscinv$ in \eqref{priv:eq:discreteQinv} and $0<\inf_{(v,x)\in\Vimagesp\times\Ximagesp}p_{VX}(v,x)\leq\sup_{(v,x)\in\Vimagesp\times\Ximagesp}p_{VX}(v,x)<\infty$, then for all $h\in L_2(P_{VX})$,
\begin{align}
\norm{L_2(P_{VZ})}{\linopQXinv h}&\leq c_Q\norm{L_2(P_{VX})}{h}. \label{priv:lem:norms:eq:q2_2_disc}
\end{align}

If $Q\in\setQgen$ in \eqref{priv:eq:qtv}, we have for all $h\in \linopQXinvDom$, 
\begin{align}
\norm{L_2(P_{VZ})}{\linopQXinv h}&\leq c_Q\norm{\linopQXinvDom}{h}. \label{priv:lem:norms:eq:q2_2_gen}
\end{align}

For any $Q\in\setQ(\Ximagesp\to\Zimagesp)$, we have for all $k\in L_2(P_{VZ})$,
\begin{align}
\norm{L_2(P_{VX})}{\linopQX k}&\leq \norm{L_2(P_{VZ})}{k}.  \label{priv:lem:norms:eq:q2_1}
\end{align}

\end{lemma}
\begin{proof}[Proof of \Cref{priv:lem:norms}]
Suppose that $Q\in\setQgen$. Then \eqref{priv:lem:norms:eq:lp} is \Cref{priv:lem:qtc_implications} \ref{priv:lem:qtc_implications_norms}; \eqref{priv:lem:norms:eq:q2_2_gen} is \Cref{priv:lem:linopqx_properties} \ref{priv:lem:linopqx_properties_inv_qtc}.

Suppose now that $X$ is distributed on a finite set $\Ximagesp=\Zimagesp$ and $P_{VX}$ has density $p_{VX}$. First, we show \eqref{priv:lem:norms:eq:lp}. We have, for any $c>0$,
\begin{align*}
\norm{L_p(P_{VZ})}{h}^p - \fr{1}{c}\norm{L_p(P_{VX})}{h}^p  \\
= \int_\Vimagesp\sum_{x\in\Ximagesp}|h(v,x)|^p \left(p_{VZ}(v,x)-\fr{1}{c}p_{VX}(v,x)\right) \dderiv\mesV(v).
\end{align*}
By assumption, 
$$\underline{q}\ceq \min_{(x,\ba x)\in\Xsupp^2}Q(\set{x}\cond \ba x)>0.$$ 
Since $p_{VZ}(v,x)=\sum_{\ba x \in \Ximagesp} Q(\set{x}\cond \ba x)p_{VX}(v,\ba x)$, we have
\begin{align*}
p_{VZ}(v,x)-\fr{1}{c}p_{VX}(v,x) \geq \underline{q}\sum_{\ba x\in\Ximagesp}p_{VX}(v,\ba x) - \fr{1}{c} p_{VX}(v,x) \\ =
 \underline{q}\sum_{\ba x\in\Ximagesp:\ba x\neq x}p_{VX}(v,\ba x)+\left(\underline{q}-\fr{1}{c}\right)p_{VX}(v,x).
\end{align*}
Hence, setting $c\ceq 1/\underline{q}$ implies $p_{VZ}(v,x)-\fr{1}{c}p_{VX}(v,x)\geq 0$. Thus 
$$\norm{L_p(P_{VZ})}{h}^p - \fr{1}{c}\norm{L_p(P_{VX})}{h}^p=\norm{L_p(P_{VZ})}{h}^p - \underline{q}\norm{L_p(P_{VX})}{h}^p\geq 0.$$

Second, we show \eqref{priv:lem:norms:eq:q2_2_disc}. Consider the matrix representation of $\linopQXinv$ in \Cref{priv:lem:linopqx_properties} \ref{priv:lem:linopqx_properties_exp_inv_disc}. The linear operator (matrix) $Q:\real^{J\times J}\to \real^{J\times 1}$ has inverse $Q^{-1}$ because $Q\in\setQident$. As $(Q^{-1})^\intercal:\real^{J\times J}\to \real^{J\times 1}$ is a linear operator on a finite-dimensional space, it is a continuous and bounded linear operator (e.g.\ \Citet[Theorem 2.4]{kress_linear_2014}). Whence, $\norm{2}{(Q^{-1})^\intercal t}\lesssim \norm{2}{t}$ for all $t\in\real^{J\times 1}$, where $\norm{2}{.}$ is the Euclidean norm on $\real^{J\times 1}$. The assertion follows by 
\begin{align*}
\norm{L_2(P_{VZ})}{\linopQXinv h}^2= \int_\Vimagesp\sum_{z\in\Ximagesp} \left[(\linopQXinv h)(v,z) \right]^2 p_{VZ}(v,z)\dderiv\mesV(v) \\ \leq \supnorm{p_{VZ}}\int_\Vimagesp\sum_{z\in\Ximagesp} \left[(\linopQXinv h)(v,z) \right]^2 \dderiv\mesV(v)  
\lesssim \int_\Vimagesp\sum_{z\in\Ximagesp} \left[ \sum_{x\in\Ximagesp}((Q^{-1})^\intercal)_{z,x} h(v,x) \right]^2\dderiv\mesV(v) \\ = \int_\Vimagesp \norm{2}{(Q^{-1})^\intercal (h(v,x))_{x\in\Xsupp}}^2\dderiv\mesV(v)  \lesssim \int_\Vimagesp \norm{2}{(h(v,x))_{x\in\Xsupp}}^2\dderiv\mesV(v)  \\ = \int_\Vimagesp \sum_{x\in\Ximagesp} h(v,x)^2\dderiv\mesV(v)  \simeq \int_\Vimagesp \sum_{x\in\Ximagesp} h(v,x)^2 p_{VX}(v,x)\dderiv\mesV(v) = \norm{L_2(P_{VX})}{h},
\end{align*}
since $\supnorm{p_{VZ}}<\infty$ by the assumption $\supnorm{p_{VX}}<\infty$, and $\inf_{(v,x)\in\Vimagesp\times\Ximagesp} p_{VX}(v,x)>0$ by assumption.

Finally, we show \eqref{priv:lem:norms:eq:q2_1}. By \Cref{priv:lem:linopqx_properties} \ref{priv:lem:linopqx_properties_exp} and Jensen's inequality, 
$$((\linopQX k)(v,x))^2\leq \Ebc{k(V,Z)^2}{V=v,X=x}.$$ But then 
\begin{align*}
\norm{L_2(P_{VX})}{\linopQX k}^2= \Eb{((\linopQX k)(V,X))^2} \\
\leq \Eb{\Ebc{k(V,Z)^2}{V,X}}=\E k(V,Z)^2 = \norm{L_2(P_{VZ})}{k}^2.
\end{align*}
\end{proof}

\subsection{Total-Variational and Differential Privacy}\label{priv:app:priv:subsec:relation_to_diff_priv}

Recall the definition of differential privacy, which is one of the most common privacy guarantees.

\begin{definition}[Local $(\alpha,\delta)$-Differential Privacy: $(\alpha,\delta)$-LDP \citep{hutchison_calibrating_2006}]
\label{priv:def:local_diff_priv}
For $\alpha,\delta\geq 0$, the privacy mechanism $Q\in\setQ(\Ximagesp\to\Zimagesp)$ is locally $(\alpha,\delta)$-differentially private if $Q(B\cond x)\leq e^\alpha Q(B\cond x')+\delta$ for all $B\in\Zsigma$ and for all $x,x'\in\Ximagesp$.
\end{definition}

\Cref{priv:lem:privacy_relations} shows how $\alpha$-LTVP in \Cref{priv:def:local_tv_priv} and $(\alpha,\delta)$-LDP are related. An $\alpha$-LTVP is always an $(.,\delta=\alpha)$-LDP, which provides weaker privacy guarantee than an $(\alpha,\delta=0)$-LDP. Conversely, when $(\alpha,\delta)$ is small enough --- so that privacy is strict ---, $(\alpha,\delta)$-LDP is an $(e^\alpha-1+\delta)$-LTVP. For example, an $(\alpha,\delta=0)$-LDP is also an $(e^\alpha-1)$-LTVP for $\alpha\leq \log(2)\approx 0.69$. 

\begin{lemma}[$(\alpha,\delta)$-LDP and $\alpha$-LTVP]
\label{priv:lem:privacy_relations}
For all $0\leq \alpha\leq 1$, every $\alpha$-LTVP mechanism is $(\tld\alpha,\alpha)$-LDP for any $\tld\alpha\geq 0$.
For all $\alpha,\delta\geq 0$ such that $e^\alpha-1+\delta\leq 1$, every $(\alpha,\delta)$-LDP mechanism is $(e^\alpha-1+\delta)$-LTVP.
\end{lemma}
\begin{proof}
Let $Q$ be $\alpha$-LTVP. As $|Q|=Q$, we have for all $B\in\Zsigma$,
\begin{align*}
Q(B\cond x)=|Q(B\cond x)|=|Q(B\cond x)-Q(B\cond x')+Q(B\cond x')| \\
\leq |Q(B\cond x')|+|Q(B\cond x)-Q(B\cond x')| \leq Q(B\cond x')+\alpha\leq e^{\tld\alpha}Q(B\cond x')+\alpha 
\end{align*}
for all $x,x'\in\Ximagesp$ and for any $\tld\alpha\geq 0$. Hence, $Q$ is $(\tld\alpha,\alpha)$-LDP.

Now let $Q$ be an $(\alpha,\delta)$-LDP mechanism. Then for all $x,x'\in\Ximagesp$ and $B\in\Zsigma$,
\begin{align*}
|Q(B \cond x)-Q(B\cond x')|=
\begin{cases}
Q(B\cond x)-Q(B\cond x')\quad&\text{ if }\, Q(B \cond x)-Q(B\cond x')\geq 0 \\
Q(B\cond x')-Q(B\cond x)\quad&\text{ if }\, Q(B \cond x)-Q(B\cond x')< 0
\end{cases} \\
\leq  
\begin{cases}
e^\alpha Q(B\cond x')+\delta-Q(B\cond x')\quad&\text{ if }\, Q(B \cond x)-Q(B\cond x')\geq 0 \\
e^\alpha Q(B\cond x)+\delta-Q(B\cond x)\quad&\text{ if }\, Q(B \cond x)-Q(B\cond x')<0
\end{cases} \\
=
\begin{cases}
(e^\alpha-1) Q(B\cond x')+\delta\quad&\text{ if }\, Q(B \cond x)-Q(B\cond x')\geq 0 \\
(e^\alpha-1) Q(B\cond x)+\delta\quad&\text{ if }\, Q(B \cond x)-Q(B\cond x')< 0
\end{cases} \\
\leq e^\alpha-1+\delta,
\end{align*}
because $Q$ maps to $[0,1]$ and $e^\alpha-1\geq 0$ for all $\alpha\geq 0$. Hence, if $e^\alpha-1+\delta
\leq 1$, $Q$ is $(e^\alpha-1+\delta)$-LTVP.
\end{proof}

\subsection{Invertible Privacy Mechanisms}\label{priv:app:priv:subsec:invertible_mechanisms}

In \Cref{priv:sec:priv:subsec:priv_mech}, we saw that for a generic $X$, a mechanism $Q\in\setQgen$ implies the existence of $\linopL$. In some special cases, choices different from $\setQgen$ also imply the existence of $\linopL$ in \eqref{priv:eq:linopL}.

The special cases arise when $Q(\cdot\cond x)=\Qbar_{\mathrm{c}}(\cdot\cond x)$, where $\Qbar_{\mathrm{c}}(\cdot\cond x)$ admits a $\mesX$-density $\qbar_{\mathrm{c}}(\cdot\cond x)$ (as is also the case for the discretely distributed covariates). Then inverting \eqref{priv:eq:distribution_vz} is equivalent to solving a Fredholm integral equation of the first kind (see e.g.\ \cite{polanin_handbook_1998}) where $\qbar_{\mathrm{c}}$ induces a linear integral operator. Such equations are usually ill-posed \citep{polanin_handbook_1998}, but in a few favourable cases they do admit a unique solution which also retains the nonparametric model class. 

One such favourable case occurs when $\Zimagesp=\Ximagesp=\real^K$ and $\qbar_{\mathrm{c}}$ corresponds to, for example, the Laplace mechanism satisfying $\alpha$-LDP, adding Laplace noise from the Laplace density $p_\varepsilon$. Then for the Fourier transform $\mathcal{F}$ and its inverse $\mathcal{F}^{-1}$, we have $p_{VX}(v,x)=(\mathcal{F}^{-1}(w\mapsto \fr{(\mathcal{F}p_{VZ}(v,\cdot))(w)}{(\mathcal{F}p_\varepsilon)(w)}))(x)$ by the convolution theorem.  Another favourable case occurs when $\Zimagesp=\Ximagesp$ and $\qbar_{\mathrm{c}}$ induces a compact integral operator. Then the equation admits a unique and smooth solution (\Citet[Theorem 3.4]{kress_linear_2014}). One class of compact linear operators are Hilbert-Schmidt operators (\Citet[Chapter VI.6]{reed_methods_1972}, \Citet[Chapter 3]{kress_linear_2014}). The $\qbar_{\mathrm{c}}$ induces a Hilbert-Schmidt operator if 
$$\int_{\Ximagesp\times\Ximagesp}\qbar_{\mathrm{c}}(z\cond x)^2\dderiv\mesX(z)\dderiv\mesX(x)<\infty$$
(\Citet[Theorem VI.23]{reed_methods_1972}), or if $(z,x)\mapsto \qbar_{\mathrm{c}}(z\cond x)$ was continuous and the domain $\Zimagesp=\Ximagesp\subset\real^K$ was compact (\Citet[Theorem 2.27]{kress_linear_2014}).

The first, convolution case is specific to privacy achieved by \emph{additive} noise and image space $\real^K$, which is not suitable for a generic covariate under our consideration, for example when $X$ also contains coordinates distributed on a finite set. The second, compact case also places restrictions on the domain, or require $\int_{\Ximagesp\times\Ximagesp}\qbar_{\mathrm{c}}(z\cond x)^2\dderiv\mesX(z)\dderiv\mesX(x)<\infty$, which we do not expect to hold unless $\Ximagesp$ is compact (for example, when $\Ximagesp=\real$ and the mechanism is the additive Laplace noise, then the last integral is infinite). Moreover, these cases assume the existence of a density. In contrast to these, our mechanism \eqref{priv:eq:qtv} allows for more generic covariate types and space $\Zimagesp=\Ximagesp$ handled smoothly by a single mechanism.



\section{Estimation of Nuisance Parameters}\label{priv:app:priv_nuisance}

This section proves \Cref{priv:prop:rate_mm_priv} and \Cref{priv:prop:priv_bound} of \Cref{priv:sec:estim_nuisance_priv}, and details their application to the estimation of the regression and of the Riesz representer. \Cref{priv:app:priv_nuisance:subsec:finite_dim} contains results on finite-dimensional models, while \Cref{priv:app:priv_nuisance:subsec:infinite_dim} on infinite-dimensional models.


\subsection{Finite-Dimensional Models}\label{priv:app:priv_nuisance:subsec:finite_dim}

In \Cref{priv:app:priv_nuisance:subsec:finite_dim:subsubsec:mm}, we prove \Cref{priv:prop:rate_mm_priv}, and we apply it to the estimation of the regression and the Riesz representer in \Cref{priv:app:priv_nuisance:subsec:finite_dim:subsubsec:regression,priv:app:priv_nuisance:subsec:finite_dim:subsubsec:riesz}, respectively.

\subsubsection{Method-of-Moments}\label{priv:app:priv_nuisance:subsec:finite_dim:subsubsec:mm}

Before proving \Cref{priv:prop:rate_mm_priv} on the private method-of-moments, it is helpful to recall the nonprivate version.

\begin{lemma}[Nonprivate Generalised method-of-moments (\cite{hansen_large_1982} and \cite{newey_chapter_1994})]
\label{priv:lemma:rate_mm}
Suppose that $\theta_0\in\Theta\subset\real^K$ is the unique minimiser
$$\theta_0=\arg\min_{\theta\in\Theta} P_{VX}\Xi_\theta,$$
for a fixed $\Xi_{\theta}:\Vimagesp\times\Ximagesp\to\real$, $\theta\in\Theta$. Further suppose that the derivative $\Deriv_\theta \Xi_{\tld\theta}(v,x)$ of $\theta\mapsto \Xi_\theta(v,x)$ exists at all $\tld\theta\in\neighTheta$, where $\neighTheta$ is a neighbourhood of $\theta_0$, and for all $(v,x)\in\Vimagesp\times\Ximagesp$. 
Assume that
\begin{enumerate}[label=(\roman*)]
\item The value $\theta_0$ is in the interior of the compact $\Theta$.

\item Let 
$$\phi_{\tld\theta}(v,x)\ceq \Deriv_\theta \Xi_{\tld\theta}(v,x)^\intercal,\quad (v,x)\in\Vimagesp\times\Ximagesp,$$
be the $\real^{K\times 1}$-valued derivative of $\theta\mapsto\Xi_\theta(v,x)$ at $\tld\theta$. The $\phi_{\theta}$ satisfies 
\begin{align}
\norm{L_1(P_{VX})}{\norm{2}{\phi_{\theta_0}}^2}<\infty, \label{priv:lemma:rate_mm_phi_bound}
\end{align}
where, for a fixed $(v,x)\in\Vimagesp\times\Ximagesp$, $\norm{2}{\phi_{\theta_0}}^2(v,x)$ is sum of the $K$ squared entries of $\phi_{\theta_0}(v,x)$.

\item The $\real^{K\times K}$-valued derivative $\dot\phi_{\tld\theta}(v,x)\ceq \Deriv_\theta \phi_{\tld\theta}(v,x)$ of $\theta\mapsto \phi_{\theta}(v,x)$ at $\tld\theta$ exist at all $\tld\theta\in\neighTheta$ and all $(v,x)\in\Vimagesp\times\Ximagesp$. The map $\theta\mapsto \dot\phi_{\theta}(v,x)$ is continuous at all $\theta\in\neighTheta$ for all $(v,x)\in\Vimagesp\times\Ximagesp$. The expectation of $\dot\phi_{\theta_0}$ exists, and the matrix $P_{VX}\dot\phi_{\theta_0}$ is invertible. Furthermore, 
\begin{align}
\norm{L_1(P_{VX})}{\sup_{\theta\in\neighTheta}\norm{1}{\dot \phi_\theta}}<\infty,\label{priv:lemma:rate_mm_phideriv_bound}
\end{align}
where, for a fixed $(v,x)\in\Vimagesp\times\Ximagesp$, $\norm{1}{\dot \phi_\theta}(v,x)$ is the sum of the absolute values of the $K^2$ entries of $\dot\phi_\theta(v,x)$.
\end{enumerate}

Let $A_n\in\real^{K\times K}$ be an arbitrary sequence of (possibly random and then $\sigma(\mathcal{S}')$-measurable) matrices with $A_n\os{P_{VX}}{\to}A_0$ as $n\to\infty$ for a symmetric positive definite $A_0$. Then the solution $\hat\theta$ to 
$$\tld\theta\mapsto \Lambda_n(\tld\theta)\ceq \big(\Probn'\phi_{\tld\theta}^\intercal \big)A_n \big(\Probn'\phi_{\tld\theta}\big)\equiv 0$$ 
up to $\Lambda_n(\hat\theta)=\smallOPs{P_{VX}}{n^{-1/2}}$ satisfies $\sqrt{n}(\hat\theta-\theta_0)\overset{P_{VX}}{\rsquig} \mathcal{N}(0,\Sigma)$ as $n\to\infty$, where
\begin{align*}
\Sigma \ceq ({\dot\Phi^\intercal}A_0{\dot\Phi})^{-1}{\dot\Phi^\intercal}A_0\Phi A_0{\dot\Phi}({\dot\Phi^\intercal}A_0{\dot\Phi})^{-1},\quad \dot \Phi \ceq P_{VX} \dot \phi_{\theta_0}, \quad \Phi \ceq P_{VX}\phi_{\theta_0}\phi_{\theta_0}^\intercal.
\end{align*}

Further, let $\xi_{\theta}:\Vimagesp\times\Ximagesp\to\real$, $\theta\in\Theta$, be (possibly random and then $\sigma(\mathcal{S}')$-measurable) functions with the derivative of $\theta\mapsto \xi_\theta(v,x)$ at $\tld\theta$, $\Deriv_\theta\xi_{\tld\theta}(v,x)$, existent at all $\tld\theta\in\neighTheta$ for all $(v,x)\in\Vimagesp\times\Ximagesp$. If
\begin{align}
\norm{L_2(P_{VX})}{\sup_{\tld\theta\in\neighTheta}\norm{2}{\Deriv_\theta \xi_{\tld\theta}}}=\bigOPs{P_{VX}}{1},\label{priv:lemma:rate_mm_xideriv_bound}
\end{align}
where, for a fixed $(v,x)\in\Vimagesp\times\Ximagesp$, $\norm{2}{\Deriv_\theta \xi_{\tld\theta}}^2(v,x)$ is the sum of squared entries of the vector $\Deriv_\theta \xi_{\tld\theta}(v,x)$, then $\norm{L_2(P_{VX})}{\xi_{\hat\theta}-\xi_{\theta_0}}=\bigOPs{P_{VX}}{n^{-1/2}}$.
\end{lemma}

\begin{proof}[Proof of \Cref{priv:lemma:rate_mm}]
We have the identification $P_{VX}\phi_{\theta_0}=0_{K}$. The conditions of the lemma are versions of  Assumptions 2.1--2.5 and 3.1--3.6 in \cite{hansen_large_1982} adapted to our setting, whereby Theorems 2.1 and 3.2 \emph{ibid} apply, giving $\sqrt{n}(\hat\theta-\theta_0)\overset{P_{VX}}{\rsquig}\mathcal{N}(0,\Sigma)$ (see also \citet[Theorems 2.6, 3.4]{newey_chapter_1994}). Then by the mean-value theorem and \eqref{priv:lemma:rate_mm_xideriv_bound}, 
$$\norm{L_2(P_{VX})}{\xi_{\hat\theta}-\xi_{\theta_0}}\leq \norm{2}{\hat\theta-\theta_0}\norm{L_2(P_{VX})}{\sup_{\tld\theta\in\neighTheta}\norm{2}{\Deriv_\theta \xi_{\tld\theta}}}=\bigOPs{P_{VX}}{n^{-1/2}}.$$
\end{proof}

\Cref{priv:prop:rate_mm_priv} is modelled after \Cref{priv:lemma:rate_mm} but the $\norm{L_1(P_{VX})}{\cdot}$ in conditions \eqref{priv:lemma:rate_mm_phi_bound} and \eqref{priv:lemma:rate_mm_phideriv_bound} replaced with $\norm{L_1(P_{VZ})}{\cdot}$ in \Cref{priv:prop:rate_mm_priv}. The restriction of $\norm{L_1(P_{VZ})}{\cdot}$-bounds in \eqref{priv:lemma:rate_mm_phi_bound} and \eqref{priv:lemma:rate_mm_phideriv_bound} instead of the $\norm{L_1(P_{VX})}{\cdot}$-bounds in the nonprivate case of \Cref{priv:lemma:rate_mm} appears to be unavoidable in our construction: we need to control $\norm{L_1(P_{VZ})}{\cdot}$, and while \Cref{priv:lem:linopqx_properties} \ref{priv:lem:qtc_implications_norms} shows that $\norm{L_p(P_{VX})}{h}\lesssim\norm{L_p(P_{VZ})}{h}$ when $Q\in\setQgen$ in \eqref{priv:eq:qtv}, the converse $\norm{L_p(P_{VZ})}{h}\lesssim\norm{L_p(P_{VX})}{h}$  may fail. One can replace $\norm{L_p(P_{VZ})}{\cdot}$ with $\supnorm{\cdot}$-bounds, which could be easier to verify. For example, in estimating $\muX=\muXs{\theta_0}$, if $m$ is bounded and $\muXs{\theta}$ is a generalised linear model with continuous second derivative, then the $\supnorm{.}$-bounds hold provided $\Vimagesp_\mhit{1}\times\Ximagesp$ is compact.

\begin{proof}[Proof of \Cref{priv:prop:rate_mm_priv}]
Because $\linopQXinv$ and differentiation with respect to $\theta$ commutes,
\begin{align*}
\ba \phi_{\tld\theta}^\intercal=\Deriv_\theta \ba\Xi_{\tld\theta}=\Deriv_\theta \linopQXinv \Xi_{\tld\theta}= \linopQXinv \Deriv_\theta \Xi_{\tld\theta}=\linopQXinv\phi_{\tld\theta}^\intercal,
\end{align*}
and then also
\begin{align*}
 \Deriv_\theta \ba\phi_{\tld\theta}=\Deriv_\theta \linopQXinv \phi_{\tld\theta}= \linopQXinv \Deriv_\theta\phi_{\tld\theta}=\linopQXinv \dot \phi_{\tld\theta}.
\end{align*}
(Here, we denote with $\linopQXinv \phi_{\tld\theta}^\intercal$ the vector and with $\linopQXinv \dot \phi_{\tld\theta}$ the matrix where $\linopQXinv$ is applied element-wise to the coordinate functions of $\phi_{\tld\theta}^\intercal$ and $\dot \phi_{\tld\theta}$, respectively.) 
\Cref{priv:lem:linopqx_properties} \ref{priv:lem:linopqx_properties_change} implies the identification 
$P_{VZ}\ba \phi_{\theta_0}^\intercal=P_{VX}\phi_{\theta_0}^\intercal=0_{K}$,
as $P_{VX}\phi_{\theta_0}^\intercal=0_{K}$ by the identifiability conditions of \Cref{priv:lemma:rate_mm}; and also $P_{VZ}\Deriv_\theta\ba\phi_{\tld\theta}=P_{VZ} \linopQXinv \dot \phi_{\tld\theta}=P_{VX}\dot\phi_{\tld\theta}$. By this, and the strengthening of $\norm{L_p(P_{VX})}{\cdot}$ to $\norm{L_p(P_{VZ})}{\cdot}$ in \eqref{priv:lemma:rate_mm_xideriv_bound},\eqref{priv:lemma:rate_mm_phi_bound}, and \eqref{priv:lemma:rate_mm_phideriv_bound}, all assumptions in \Cref{priv:lemma:rate_mm} that hold under $(\Probn', P_{VX})$ continue to hold under $(\Probnb', P_{VZ})$. Thus, the assertion follows with $\dot\Phi$ being the same as in \Cref{priv:lemma:rate_mm}.
\end{proof}

\subsubsection{Regression $\muX$}\label{priv:app:priv_nuisance:subsec:finite_dim:subsubsec:regression}

Suppose that $\muX=\muXs{\theta_0}$ belonging to the model 
\begin{align}
\muXmodel\ceq \set{\muXs{\theta}:\theta\in\Theta\subset\real^{K}} \label{priv:eq:muXmodel}
\end{align}
for a fixed $K\geq 1$.

With \Cref{priv:lemma:rate_mm} and \Cref{priv:prop:rate_mm_priv}, we can attain the parametric root-$n$ rate for the regression in both the nonprivate and the private setting.

\begin{corollary}[Regression --- Rate of Nonprivate Estimator $\muXs{\hat\theta}$]
\label{priv:cor:muXpara_rate}
Assume that $\muX=\muXs{\theta_0}$ belonging to the model in \eqref{priv:eq:muXmodel}.
For $(v,x,\tld\theta)\in\Vimagesp\times\Ximagesp\times\Theta$, let $\Xi_{\tld\theta}(v,x)\ceq \Delta_{\muXs{\tld\theta}}^2(v,x)=(m(v,x) - \muXs{\tld\theta}(v_{\mhit{1}},x))^2$,
\begin{align*}
\phi_{\tld\theta}(v,x)\ceq \Deriv_{\theta}\Delta_{\muXs{\tld\theta}}^2(v,x)^\intercal=2(m(v,x)-\muXs{\tld\theta}(v_\mhit{1},x))\Deriv_\theta \muXs{\tld\theta}(v_\mhit{1},x)^\intercal, \quad  \xi_{\tld\theta}\ceq \muXs{\tld\theta}.
\end{align*} 

Take a sequence of (possibly random and then $\sigma(\mathcal{S}')$-measurable) matrices $A_n\in\real^{K\times K}$, and
let $\hat\theta\in\Theta$ be the solution to the estimating equation 
$\tld\theta\mapsto \Lambda_n(\tld\theta)\ceq \big(\Probn' \phi_{\tld\theta}^\intercal\big)A_n\big(\Probn' \phi_{\tld\theta}\big)\equiv 0$ 
up to $\Lambda_n(\hat\theta)=\smallOPs{P_{VX}}{n^{-1/2}}$.

If $(\Xi_\theta, \phi_{\theta},A_n,\xi_\theta)$ satisfy the conditions of \Cref{priv:lemma:rate_mm}, then $\sqrt{n}(\hat\theta-\theta_0)\overset{P_{VX}}{\rsquig} \mathcal{N}(0,\Sigma)$ as $n\to\infty$ for some $\Sigma$, and 
$\norm{L_2(P_{VX})}{\muXs{\hat\theta}-\muX}=\bigOPs{P_{VX}}{n^{-1/2}}$.
\end{corollary}

\begin{corollary}[Regression --- Rate of Private Estimator $\muXs{\chk\theta}$]
\label{priv:cor:muXpara_rate_priv}
Assume that $P_{VZ}\in\mathcal{P}_{VZ}(Q, \mathcal{P}_{VX})$, for a fixed $Q\in\setQident$ and $P_{VX}\in\mathcal{P}_{VX}$ subject to the parametric assumption \eqref{priv:eq:muXmodel}. Let $(\Xi_\theta, \phi_{\theta},\xi_\theta)$, $\theta\in\Theta$, be as defined in \Cref{priv:cor:muXpara_rate}.

Take a sequence of (possibly random and then $\sigma(\bar{\mathcal{S}}')$-measurable) matrices $\ba A_n\in\real^{K\times K}$, and let $\chk\theta\in\Theta$ be the solution to the estimating equation 
$\tld\theta\mapsto \ba\Lambda_n(\tld\theta)\ceq \big(\Probnb' \ba\phi_{\tld\theta}^\intercal\big)\ba A_n \big(\Probnb' \ba\phi_{\tld\theta}\big)\equiv 0$
up to $\ba\Lambda_n(\chk\theta)=\smallOPs{P_{VZ}}{n^{-1/2}}$, where $\ba\phi_{\tld\theta}\ceq \Deriv_{\theta}\ba\Xi_{\tld\theta}$ with $\ba\Xi_{\theta}\ceq\linopQXinv \Xi_{\theta}$ for the inverse $\linopQXinv$ of $\linopQX$ in \eqref{priv:eq:linopQX}.

If $(\Xi_\theta, \phi_{\theta}, \ba A_n, \xi_\theta)$ satisfy the conditions of \Cref{priv:prop:rate_mm_priv}, then $\sqrt{n}(\chk\theta-\theta_0)\overset{P_{VZ}}{\rsquig} \mathcal{N}(0,\ba\Sigma)$ as $n\to\infty$ for some $\ba\Sigma$, and $\norm{L_2(P_{VX})}{\muXs{\chk\theta}-\muX}=\bigOPs{P_{VZ}}{n^{-1/2}}$.
\end{corollary}

The proofs of \Cref{priv:cor:muXpara_rate,priv:cor:muXpara_rate_priv} are omitted.

\subsubsection{Riesz Representer $\rieszX$}\label{priv:app:priv_nuisance:subsec:finite_dim:subsubsec:riesz}

In this section, we consider the estimation of the Riesz representer $\rieszX$.  We begin with an identification result, which is not limited to finite-dimensional models (see also {\citet[Theorem 2 (v)]{rotnitzky_characterization_2021}}).

\begin{lemma}[Identification of $\rieszX$]
\label{priv:lem:riesz_maxim}
For all $\gamma\in\Gamma$, the Riesz representer $\rieszX_\gamma$ of 
$L_2(P_{V_\mhit{1}}X)\ni\mu\mapsto P_{VX} f(V,X,\mu, \gamma)$ satisfies \eqref{priv:eq:risz_crit}.
\end{lemma}
\begin{proof}[Proof of \Cref{priv:lem:riesz_maxim}]
By the definition of $\rieszX_\gamma$, 
\begin{align*}
P_{VX}f(V,X,h,\gamma)=P_{VX}\rieszX_\gamma(V_\mhit{1},X)h(V_\mhit{1},X) \text{ for all }  h \in L_2(P_{V_\mhit{1}}X). 
\end{align*}
Then 
\begin{align*}
P_{VX}\Upsilon_{\gamma,h}(V,X)=P_{VX}\left[h^2-2f\right]=P_{VX}\left[h^2-2h\rieszX_\gamma\right]=P_{VX}\left[(h-\rieszX_\gamma)^2 - \rieszX_\gamma^2 \right],
\end{align*}
whereby $\arg\min_{h\in L_2(P_{V_\mhit{1}}X)}P_{VX}\Upsilon_{\gamma,h}=\arg\min_{h\in L_2(P_{V_\mhit{1}}X)}P_{VX}\left[(h-\rieszX_\gamma)^2\right]=\rieszX_\gamma$.
\end{proof}

Suppose that the Riesz representer $\rieszXs{\gamma}$ of $\mu\mapsto P_{VX}f(V,X,\mu,\gamma)$ is $\rieszXs{\gamma}=\rieszXs{\gamma,\theta_0}$  for each $\gamma\in\Gamma$ \emph{uniformly}, where $\rieszX_{\gamma,\theta_0}$ belongs to the model 
\begin{align}
\rieszXmodelgamma\ceq \set{\rieszXs{\gamma,\theta}:\theta\in\Theta\subset\real^{K}}, \label{priv:eq:rieszXmodelgamma}
\end{align}
with the map $(\gamma,\theta)\mapsto \rieszXs{\gamma,\theta}$ known. A  prototypical example is a parametric model for the propensity score in inferring the average treatment effect on the treated.

\begin{example}[name={Average Treatment Effect on the Treated},continues=priv:ex:att_dr]
Recall that for $\gammaV(c)\ceq \ba\gamma\ceq \E D=p_1$, the Riesz representer for $\Ebc{Y^0}{D=1}$ is $\rieszX(d,x)=\fr{1-d}{\gammaV(c)}\fr{\piX(1\cond x)}{1-\piX(1\cond x)}$, and for $\Ebc{Y^1-Y^0}{D=1}$, it is $\rieszX(d,x)=\fr{d}{\gammaV(c)}-\fr{1-d}{\gammaV(c)}\fr{\piX(1\cond x)}{1-\piX(1\cond x)}$. Suppose that the propensity score $\piX(1\cond x)=\pi_{\theta_0}(x)$ for some $\pi_{\theta_0}$ in the presumed model $\set{\pi_{\theta}:\theta\in\Theta\subset\real^{K}}$, for instance, the logistic model $\pi_\theta(x)=(1+\exp(-\theta^\intercal x))^{-1}$. Then
\begin{align*}
\rieszXs{\gamma,\theta_0}(d,x) =\fr{1-d}{\gamma}\fr{\pi_{\theta_0}(x)}{1-\pi_{\theta_0}(x)},\quad
\rieszXs{\gamma,\theta_0}(d,x) =\fr{d}{\gamma}-\fr{1-d}{\gamma}\fr{\pi_{\theta_0}(x)}{1-\pi_{\theta_0}(x)}.
\end{align*}
are the Riesz representers for $\Ebc{Y^0}{D=1}$ and $\Ebc{Y^1-Y^0}{D=1}$, respectively, for $\gamma=\gammaV(c)$ $=\ba\gamma=\E D$. Conversely, $\rieszXs{\gammaV(c),\theta}$ is not a Riesz representer unless $\theta=\theta_0$. Let $\lambda_{0,\theta}(d,x)\ceq (1-d)\fr{\pi_\theta(x)}{1-\pi_\theta(x)}$ and $\lambda_{1,\theta}(d,x)\ceq d- (1-d)\fr{\pi_\theta(x)}{1-\pi_\theta(x)}$. Then the models for the Riesz representers are
\begin{align*}
\rieszXmodelgamma =\set{\gamma^{-1}\lambda_{0,\theta}:\theta\in\Theta}, \quad \rieszXmodelgamma =\set{\gamma^{-1}\lambda_{1,\theta}:\theta\in\Theta} 
\end{align*}
 for $\Ebc{Y^0}{D=1}$ and $\Ebc{Y^1-Y^0}{D=1}$, respectively.
\end{example}

In the parametric model \eqref{priv:eq:rieszXmodelgamma}, the Riesz representer $\rieszXs{\gamma,\theta_0}$ is identified by \Cref{priv:lem:riesz_maxim} as 
\begin{align}
\rieszX_{\gamma,\theta_0}&=  \arg\min_{\rho \in\rieszXmodelgamma} P_{VX}\Upsilon_{\gamma,\rho}(V,X) \nonumber \\
&= \arg\min_{\rho\in\rieszXmodelgamma}P_{VX}\bigg[\rho(v_\mhit{1},x)^2-2f(V,X,\rho,\gamma)\bigg]. \label{priv:eq:rieszgamma_para}
\end{align}
Importantly, for all $\gamma\in\Gamma$ uniformly, $\theta=\theta_0$, and only $\theta=\theta_0$, gives the Riesz representer in the model $\rieszXmodelgamma$. An implication is that for \emph{any} $\gamma\in\Gamma$,
\begin{align}
\theta_0 = \arg\min_{\theta\in\Theta}P_{VX}\Upsilon_{\gamma,\rieszXs{\gamma,\theta}}(V,X)= \arg\min_{\theta\in\Theta}P_{VX}\bigg[\rieszXs{\gamma,\theta}(v_\mhit{1},x)^2-2f(V,X,\rieszXs{\gamma,\theta},\gamma)\bigg].\label{priv:eq:rieszgamma_para_theta}
\end{align}
This permits us to infer $\theta_0$ without suffering any bias from the unknown $\gammaV(c)$. Take an arbitrary $\gamma_0\in\Gamma$. In the nonprivate setting, \eqref{priv:eq:rieszgamma_para_theta} can directly be used to construct a method-of-moments estimator $\hat\theta_{\gamma_0}$ of $\theta_0$, and next estimate the representer as $\rieszXhat\ceq \rieszXs{\gammaVhat(c),\hat\theta_{\gamma_0}}$ for $\gammaVhat(c)$ of \eqref{priv:eq:gammaVhat}. In the private setting, \Cref{priv:prop:rate_mm_priv} can be used to construct a private estimator $\chk\theta_{\gamma_0}$ of $\theta_0$, and then estimate the representer as $\rieszXchk\ceq \rieszXs{\gammaVchk(c),\chk\theta_{\gamma_0}}$ for $\gammaVchk(c)$ of \eqref{priv:eq:gammaVchk}. If $(\gamma,\theta)\mapsto\rieszXs{\gamma,\theta}$ is smooth enough, then analogues to \Cref{priv:cor:muXpara_rate,priv:cor:muXpara_rate_priv} hold, giving root-$n$ rates.

\begin{corollary}[Riesz Representer --- Rate of Nonprivate Estimator $\rieszXs{\gammaVhat(c),\hat\theta_{\gamma_0}}$]
\label{priv:cor:rieszXpara_rate}
Suppose that $\rieszXs{\gamma}=\rieszXs{\gamma,\theta_0}$ for the model $\rieszXmodelgamma$ of \eqref{priv:eq:rieszXmodelgamma}. Assume that
\begin{enumerate}[label=(\roman*)]
\item \label{priv:cor:rieszXpara_rate:diff_riesz} The map $(\gamma,\theta)\mapsto \rieszXs{\gamma,\theta}(v_\mhit{1},x)$ is differentiable at all $(\gamma,\theta)\in\neighGamma\times\neighTheta$, where the $\neigh(w)$ are neighbourhoods of $w$, for all $(v_\mhit{1},x)\in\Vimagesp_\mhit{1}\times\Ximagesp$ with partial derivatives $\partial_\gamma \rieszX_{\tld\gamma,\tld\theta}(v_\mhit{1},x)$ with respect to $\gamma$,  and $\Deriv_\theta \rieszX_{\tld\gamma,\tld\theta}(v_\mhit{1},x)$ with respect to $\theta$, at $(\tld\gamma,\tld\theta)\in\neighGamma\times\neighTheta$.

\item \label{priv:cor:rieszXpara_rate:diff_upsilon} The map $\theta\mapsto \Upsilon_{\gamma,\rieszXs{\gamma,\theta}}(v,x)=\rieszXs{\gamma,\theta}(v_\mhit{1},x)^2-2f(v,x,\rieszXs{\gamma,\theta},\gamma)$ is differentiable at all $\theta\in\neighTheta$ for all $(v,x,\gamma)\in\Vimagesp\times\Ximagesp\times\Gamma$ with $\real^{ K\times 1}$-valued derivative  $\Deriv_\theta \Upsilon_{\gamma,\rieszXs{\gamma,\tld\theta}}(v,x)^\intercal\eqc \phi_{\gamma,\tld\theta}(v,x)$ at $\tld\theta\in\neighTheta$.
\end{enumerate}
Fix an arbitrary $\gamma_0\in\Gamma$, and take a sequence of (possibly random and then $\sigma(\mathcal{S}')$-measurable) matrices $A_n\in\real^{K\times K}$. Let $\hat\theta_{\gamma_0}\in\Theta$ be the solution to the estimating equation
$\tld\theta\mapsto \Lambda_n(\tld\theta)\ceq \big(\Probn' \phi_{\gamma_0,\tld\theta}^\intercal\big)A_n\big(\Probn' \phi_{\gamma_0,\tld\theta}\big) \equiv 0$
up to $\Lambda_n(\hat\theta_{\gamma_0})=\smallOPs{P_{VX}}{n^{-1/2}}$. If $\Xi_\theta\ceq \Upsilon_{\gamma_0,\rieszXs{\gamma_0,\theta}}$, $\phi_{\theta}\ceq \phi_{\gamma_0,\theta}$, and $A_n$ satisfy the conditions of \Cref{priv:lemma:rate_mm} pertaining to $(\Xi_\theta,\phi_\theta, A_n)$ therein, then $\sqrt{n}(\hat\theta_{\gamma_0}-\theta_0)\overset{P_{VX}}{\rsquig} \mathcal{N}(0,\Sigma_{\gamma_0})$ as $n\to\infty$ for some $\Sigma_{\gamma_0}$. If, in addition,
\begin{align}
\norm{L_2(P_{VX})}{\partial_\gamma r_{\gammaVtld(c),\tld\theta_{\gamma_0}}}&=\bigOPs{P_{VX}}{1}, \label{priv:cor:rieszXpara_rate:diff_riesz_bound_gamma} \\
\norm{L_2(P_{VX})}{\norm{2}{\Deriv_\theta r_{\gammaVtld(c),\tld\theta_{\gamma_0}}}}&=\bigOPs{P_{VX}}{1} \label{priv:cor:rieszXpara_rate:diff_riesz_bound_theta} 
\end{align}
for some $(\gammaVtld(c),\tld\theta_{\gamma_0})$ between $(\gammaV(c),\theta_0)$ and $(\gammaVhat(c),\hat\theta_{\gamma_0})$,
then 
$\norm{L_2(P_{VX})}{\rieszXhat-\rieszX}=\bigOPs{P_{VX}}{n^{-1/2}}$,
where $\rieszXhat\ceq\rieszXs{\gammaVhat(c),\hat\theta_{\gamma_0}}$, $\rieszX=\rieszXs{\gammaV(c),\theta_0}$  for the model  in \eqref{priv:eq:rieszXmodelgamma} and $\gammaVhat(c)$ in \eqref{priv:eq:gammaVhat}.
\end{corollary}
\begin{proof}[Proof of \Cref{priv:cor:rieszXpara_rate}]
The asymptotic normality of $\hat\theta_{\gamma_0}$ follows directly from \Cref{priv:lemma:rate_mm} via the identification \eqref{priv:eq:rieszgamma_para_theta}, whereby $P_{VX}\phi_{\gamma,\theta_0}=0_{K }$ for any fixed $\gamma\in\Gamma$. A mean-value expansion of $\rieszXs{\gammaVhat(c),\hat\theta_{\gamma_0}}$ via \ref{priv:cor:rieszXpara_rate:diff_riesz} in combination with \eqref{priv:cor:rieszXpara_rate:diff_riesz_bound_gamma} and \eqref{priv:cor:rieszXpara_rate:diff_riesz_bound_theta} yields the second assertion as $\gammaVhat(c)-\gammaV(c)=\bigOPs{P_{VX}}{n^{-1/2}}$, and $\hat\theta_{\gamma_0}-\theta_0=\bigOPs{P_{VX}}{n^{-1/2}}$ by the first assertion.
\end{proof}

\begin{corollary}[Riesz Representer --- Rate of Private Estimator $\rieszXs{\gammaVchk(c),\chk\theta}$]
\label{priv:cor:rieszXpara_rate_priv}
Assume that $P_{VZ}\in\mathcal{P}_{VZ}(Q, \mathcal{P}_{VX})$, for a fixed $Q\in\setQident$ and $P_{VX}\in\mathcal{P}_{VX}$ subject to the parametric assumption of \eqref{priv:eq:rieszXmodelgamma}. Suppose that \ref{priv:cor:rieszXpara_rate:diff_riesz} and \ref{priv:cor:rieszXpara_rate:diff_upsilon} of \Cref{priv:cor:rieszXpara_rate} hold, and let $\Upsilon_{\gamma,\rieszXs{\gamma,\theta}}, \phi_{\gamma,\theta}$, $(\gamma,\theta)\in\Gamma\times\Theta$, be as they are defined therein.

Fix an arbitrary $\gamma_0\in\Gamma$, and take a (possibly random and then $\sigma(\bar{\mathcal{S}}')$-measurable) sequence of matrices $\ba A_n\in\real^{K\times K}$. Let $\chk\theta_{\gamma_0}\in\Theta$ be the solution to the estimating equation $\tld\theta\mapsto \ba\Lambda_n(\tld\theta)\ceq \big(\Probnb' \ba\phi_{\gamma_0,\tld\theta}^\intercal\big)\ba A_n \big(\Probnb' \ba\phi_{\gamma_0,\tld\theta}\big)\equiv 0$, 
up to $\ba\Lambda_n(\chk\theta_{\gamma_0})=\smallOPs{P_{VZ}}{n^{-1/2}}$, with $\ba\phi_{\gamma,\tld\theta}\ceq \Deriv_\theta \ba\Upsilon_{\gamma,\rieszXs{\gamma,\tld\theta}}$, $\ba\Upsilon_{\gamma,\rieszXs{\gamma,\tld\theta}} \ceq \linopQXinv \Upsilon_{\gamma,\rieszXs{\gamma,\tld\theta}}$, for the inverse $\linopQXinv$ of $\linopQX$ in \eqref{priv:eq:linopQX}.

If $\Xi_\theta\ceq\Upsilon_{\gamma_0,\rieszXs{\gamma_0,\theta}}$, $\phi_\theta\ceq\phi_{\gamma_0,\theta}$, and $\ba A_n$ satisfy the conditions of \Cref{priv:prop:rate_mm_priv} pertaining to $(\Xi_\theta, \phi_\theta, \ba A_n)$ therein, then $\sqrt{n}(\chk\theta_{\gamma_0}-\theta_0)\overset{P_{VZ}}{\rsquig} \mathcal{N}(0,\ba\Sigma_{\gamma_0})$ as $n\to\infty$ for some $\ba\Sigma_{\gamma_0}$. If, in addition, $\norm{L_2(P_{VX})}{\partial_\gamma r_{\gammaVtld(c),\tld\theta_{\gamma_0}}}=\bigOPs{P_{VZ}}{1}$ and $\norm{L_2(P_{VX})}{\norm{2}{\Deriv_\theta r_{\gammaVtld(c),\tld\theta_{\gamma_0}}}}=\bigOPs{P_{VZ}}{1}$ for some $(\gammaVtld(c),\tld\theta_{\gamma_0})$ between $(\gammaV(c),\theta_0)$ and $(\gammaVchk(c),\chk\theta_{\gamma_0})$, then 
$\norm{L_2(P_{VX})}{\rieszXchk-\rieszX}=\bigOPs{P_{VZ}}{n^{-1/2}}$,
where $\rieszXchk\ceq \rieszXs{\gammaVchk(c),\chk\theta_{\gamma_0}}$, $\rieszX=\rieszXs{\gammaV(c),\theta_0}$ for the model in \eqref{priv:eq:rieszXmodelgamma} and $\gammaVchk(c)$ in \eqref{priv:eq:gammaVchk}.
\end{corollary} 
\begin{proof}[Proof of \Cref{priv:cor:rieszXpara_rate_priv}]
Follows from \Cref{priv:prop:rate_mm_priv} as \Cref{priv:cor:rieszXpara_rate} follows from \Cref{priv:lemma:rate_mm}.
\end{proof}

\subsection{Infinite-Dimensional Models}\label{priv:app:priv_nuisance:subsec:infinite_dim}

\begin{proof}[Proof of \Cref{priv:prop:priv_bound}]
Write, for $T_n$ in \eqref{priv:eq:secnubound},
\begin{align*}
\infnuparchk(v_\mhit{1},x)-\infnupar(v_\mhit{1},x) =&\, T_n(v_\mhit{1}, x) + M_{n}(v_\mhit{1}, x) + B_{n}(v_\mhit{1}, x), \\
 M_{n}(v_\mhit{1}, x) \ceq&\, \nsumn{i}\left\{\ba\nweight_{n,i}(v_\mhit{1},x,V_i',Z_i',\secnu)-P_{VZ}[\ba\nweight_{n,i}(v_\mhit{1},x,V,Z,\secnu)]\right\}, \\
 B_{n}(v_\mhit{1}, x) \ceq&\, \nsumn{i}P_{VZ}[\ba\nweight_{n,i}(v_\mhit{1},x,V,Z,\secnu)]-\infnupar(v_\mhit{1},x),
\end{align*}
for $(v_\mhit{1},x)\in \Vimagesp_{\mhit{1}}\times\Ximagesp$,  so that
\begin{align}
P_{VX}(\infnuparchk-\infnupar)^2 \leq 2\left\{P_{VX}T_n^2 + 2(P_{VX}M_{n}^2+P_{VX}B_{n}^2)\right\},
\end{align}
since $(x+y)^2\leq2(x^2+y^2)$ for all $x,y\in\real$. For $\ba\nweight_{ni}$ in \eqref{priv:eq:nonparaestim_priv}, we have by \Cref{priv:lem:linopqx_properties}  \ref{priv:lem:linopqx_properties_change} that $P_{VZ}[\ba\nweight_{n,i}(v_\mhit{1},x,V,Z,\secnu)]=P_{VX}[\nweight_{n,i}(v_\mhit{1},x,V,X,\secnu)]$. Then, in the light of \eqref{priv:eq:secnubound}, the bound \eqref{priv:eq:nonparaestim_priv_bound} follows once we show
\begin{align}
P_{VX}M_n^2=\bigOPs{P_{VZ}}{\fr{1}{n^2}\sumn{i}\int \Vb{\ba\nweight_{n,i}(v_\mhit{1},x,V,Z,\secnu)} \dderiv P_{V_{\mhit{1}}X}(v_{\mhit{1}},x)}. \label{priv:eq:varianceorder}
\end{align}
By Markov's inequality, for all $K>0$,
\begin{align*}
\prob{P_{VX}M_{n}^2>K}\leq \fr{1}{K}\Eb{\int  M_{n}(v_\mhit{1},x)^2\dderiv P_{V_\mhit{1}X}(v_\mhit{1},x)}\\
 = \fr{1}{K}\int \Eb{M_{n}(v_\mhit{1},x)^2}\dderiv P_{V_\mhit{1}X}(v_\mhit{1},x).
\end{align*}
For any fixed $(v_\mhit{1},x)$, $\Eb{M_{n}(v_\mhit{1},x)^2}=\fr{1}{n^2}\sumn{i}\Vb{\ba\nweight_{n,i}(v_\mhit{1},x,V,Z,\secnu)}$, because the $(V_i',Z_i)$ are i.i.d.. Hence \eqref{priv:eq:varianceorder} holds.

Bound \eqref{priv:eq:sigmacondbound}. If $Q\in\setQgen$, then, by \Cref{priv:lem:linopqx_properties}, we have for $h\in L_2(P_{VZ})$,
\begin{align*}
\Vb{\linopQXinv h(V,Z)}\leq 2 \left\{\fr{1}{\alpha^2}\Vb{h(V,Z)}+\left(\fr{1-\alpha}{\alpha}\right)^2\Vb{\int h(V,z)\Qbar(\deriv z)} \right\},
\end{align*}
since for random variables $X,Y$, we have $\Vb{X+Y}\leq 2(\Vb{X}+\Vb{Y})$. Clearly, $\Vb{h(V,Z)}\leq P_{VZ}h^2$, and for the same reason, $\Vb{\int h(V,z)\Qbar(\deriv z)}\leq (P_V \otimes \Qbar)h^2$  by Jensen's inequality. But $P_V \otimes \Qbar\leq \fr{1}{1-\alpha}P_{VZ}$ by \Cref{priv:lem:qtc_implications}, so \eqref{priv:eq:sigmacondbound} holds.
\end{proof}

%

In the following, we apply \Cref{priv:prop:priv_bound} to various estimators. The regression $\muX$ is estimated by kernel and orthogonal series estimators in \Cref{priv:ex:regression_nw_kernel,priv:ex:regression_ortho}, respectively. The Riesz representer $\rieszX$ is estimated by orthogonal series in \Cref{priv:ex:riesz_ortho}.

\begin{example}[name={Regression $\muX$ --- Nadaraya--Watson Kernel Estimator}]
\label{priv:ex:regression_nw_kernel}
Suppose that $\Vimagesp_{\mhit{1}}\times\Ximagesp\subset \real^{d_{\mhit{1}}}\times\real^{d_\Xsupp}$ for fixed positive integers $d_\mhit{1},d_\Xsupp$ and that the mechanism $Q\in\setQgen$ in \eqref{priv:eq:qtv}. To estimate $\muX$, consider the nonprivate estimator
\begin{align*}
\muXhat(v_\mhit{1},x)\ceq \nsumn{i} &\fr{K_h(v_\mhit{1},x,V_{\mhit{1}i}',X_i')m(V_i',X_i')}{\nsumn{j}K_h(v_\mhit{1},x,V_{\mhit{1}j'},X_j')},\\
 K_h(v_\mhit{1},x,V_{\mhit{1}i},X_i)&\ceq K\left(\fr{1}{h}\left((v_\mhit{1},x)-(V_{\mhit{1}i},X_i)\right)\right)
\end{align*}
for a bandwidth $h>0$, and a kernel $K:\real^d\to\real$ with $d\ceq d_{\mhit{1}}+d_\Xsupp$. For the density estimator $\hat p_{V_\mhit{1}X}(v_\mhit{1},x)\ceq \fr{1}{nh^d}\sumn{i}K_h(v_\mhit{1},x,V_{\mhit{1}i}',X_i')$, $\muXhat$ is of the form \eqref{priv:eq:nonparaestim} for $\nweight_{n,i}(v_\mhit{1}, x, V',X',\secnu)=\nweight_{n,i}(v_\mhit{1}, x, V',X', p_{V_\mhit{1}X})\ceq \fr{K_h(v_\mhit{1}, x, V_\mhit{1}',X')m(V',X')}{h^d p_{V_\mhit{1}X}(v_\mhit{1},x)}$. 

The private estimator $\muXchk$ is defined as in \eqref{priv:eq:nonparaestim_priv}, with $\check p_{V_\mhit{1}X}$ an arbitrary estimator of $p_{V_\mhit{1}X}$ constructed from $\bar{\mathcal{S}}'=((V_i',Z_i'))_{i\in[n]}$. For instance, if $\Qbar$ of $Q$ in \eqref{priv:eq:qtv} admits a density $\qbar$ with respect to the same dominating measure as $P_X$ does, then \Cref{priv:lem:qtc_implications} \ref{priv:lem:qtc_implications_dens} motivates the choice
\begin{align} 
\label{priv:eq:estim_pvx_priv}
\check p_{V_\mhit{1}X}(v_\mhit{1},x)\ceq \fr{1}{\alpha}\chk p_{V_\mhit{1}Z}(v_\mhit{1},x)-\fr{1-\alpha}{\alpha}\check p_{V_\mhit{1}}(v)\qbar(x), 
\end{align}
where $(\check p_{V_\mhit{1}Z},\check p_{V_\mhit{1}})$ is an arbitrary estimator of 
$(p_{V_\mhit{1}Z}, p_{V_\mhit{1}})$ constructed from $\bar{\mathcal{S}}'$; for example
\begin{align}
\label{priv:eq:kernel_pvz}
\begin{aligned} 
\check p_{V_\mhit{1}Z}(v_\mhit{1},z)\ceq \fr{1}{nh^d}\sumn{i}K_h(v_\mhit{1},x,V_{\mhit{1}i}',Z_i'), \\ 
\check p_{V_\mhit{1}}(v_\mhit{1})\ceq \fr{1}{nh_\mhit{1}^{d_\mhit{1}}}\sumn{i}K_{\mhit{1},h_\mhit{1}}(v_\mhit{1},V_{\mhit{1}i}'),
\end{aligned} 
\end{align}
where $K_{\mhit{1},h}(v_\mhit{1},V_\mhit{1})\ceq K_\mhit{1}((v_\mhit{1}-V_\mhit{1})/h_\mhit{1})$ for a kernel $K_\mhit{1}:\real^{d_\mhit{1}}\to\real$ and bandwidth $h_\mhit{1}>0$. Next, we study $\muXchk$ in relation to \Cref{priv:prop:priv_bound}.

Term $T_n$ in \eqref{priv:eq:secnubound}. Let $g_h(v_\mhit{1},x,v',x')\ceq K_h(v_\mhit{1},x,v_\mhit{1}',x')m(v',x')$ and $\ba g_h(v_\mhit{1},x,v',z')\ceq(\linopQXinv(v',x')\mapsto g_h(v_\mhit{1},x,v',x'))(v',z')$. Then
\begin{align*}
T_n(v_\mhit{1},x)=\fr{p_{V_\mhit{1}X}(v_\mhit{1},x)-\check p_{V_\mhit{1}X}(v_\mhit{1},x)}{p_{V_\mhit{1}X}(v_\mhit{1},x)}\left(\fr{1}{\check p_{V_\mhit{1}X}(v_\mhit{1},x)}\fr{1}{nh^d}\sumn{i}\ba g_h(v_\mhit{1},x,V_i',Z_i')\right).
\end{align*}
Assume that the $(v_\mhit{1},x)$-supremum of the bracketed factor is $\bigOPs{P_{VZ}}{1}$ independent of $h$, and $\supnorm{\fr{1}{p_{V_\mhit{1}X}}}<\infty$. Then $P_{VX}T_n^2$ is dominated by $P_{VX}(p_{V_\mhit{1}X}-\check p_{V_\mhit{1}X})^2$. Suppose that $\check p_{V_\mhit{1}X}(v_\mhit{1},x)$ is defined as \eqref{priv:eq:estim_pvx_priv} and \eqref{priv:eq:kernel_pvz}, and that $p_{V_\mhit{1}Z}(v_\mhit{1},z)=\alpha p_{V_\mhit{1}X}(v_\mhit{1},z)+(1-\alpha)p_{V_\mhit{1}}(v_\mhit{1})\qbar(z)$ has $\beta$ continuous derivatives which are uniformly bounded. Further suppose that the kernel $K$ is chosen suitably, so that $K$ is of order $\beta$, satisfying $\int K(u)\dderiv u=1$, $\int u^j K(u)\dderiv u = 0$ if $|j|<\beta$ and $\int u^j K(u)\dderiv u \neq 0$ if $|j|=\beta$, where $u^j$ is the multi-index notation $u^j\ceq u_1^{j_1}u_2^{j_1}\cdots u_d^{j_d}$ for nonnegative integers $j_1,j_2,\ldots,j_d$ with $|j|\ceq\sum_{k=1}^d j_k$. Then we expect $P_{VX}(p_{V_\mhit{1}Z}-\check p_{V_\mhit{1}Z})^2=\bigOPs{P_{VZ}}{\fr{1}{nh^d}+h^{2\beta}}$, and, consequently, 
$$P_{VX}(p_{V_\mhit{1}X}-\check p_{V_\mhit{1}X})^2=\bigOPs{P_{VZ}}{\fr{1}{\alpha^2}\left(\fr{1}{nh^d}+h^{2\beta}\right)}$$ 
under suitable kernel $K_\mhit{1}$ and bandwidth $h_\mhit{1}$, because the error of $\check p_{V_\mhit{1}X}$ is dominated by the error of $\check p_{V_\mhit{1}Z}$ as opposed to that of $\check p_{V_\mhit{1}}$, since, clearly, $V_\mhit{1}$ is lower dimensional than $(V_\mhit{1},Z)$ and $\qbar$ is known.

Turning to the variance term in \eqref{priv:eq:nonparaestim_priv_bound}, $\Vb{\ba\nweight_{n,i}(v_\mhit{1},x,V,Z)}=\fr{1}{h^{2d}p_{V_\mhit{1}X}^2(v_\mhit{1},x)}\Vb{\ba g_h(v_\mhit{1},x,V,Z)}$. By \eqref{priv:eq:sigmacondbound},  
\begin{align*}
(\alpha^2/4)\Vb{\ba g_h(v_\mhit{1},x,V,Z)}\leq \int \Eb{K_h^2(v_\mhit{1},x,V_\mhit{1},Z)m^2(V,Z)}\dderiv P_{V_{\mhit{1}}X}(v_{\mhit{1}},x) \\
 = \Eb{m^2(V,Z)\int {K_h^2(v_\mhit{1},x,V_\mhit{1},Z)}\dderiv P_{V_{\mhit{1}}X}(v_{\mhit{1}},x)} = h^d \E m^2(V,Z)K^2(V_\mhit{1},Z),
\end{align*}
where the last step is by a change of variables and using that $p_{V_{\mhit{1}}X}$ integrates to one.
Hence, if $\supnorm{\fr{1}{p_{V_\mhit{1}X}}}<\infty$ and $P_{VZ} (mK)^2<\infty$, then $\fr{1}{n^2}\sumn{i}\int \ba \sigma_i^2(v_\mhit{1},x)  \dderiv P_{V_{\mhit{1}}X}(v_{\mhit{1}},x)\lesssim\fr{1}{\alpha^2 nh^d}$.

The bias term in \eqref{priv:eq:nonparaestim_priv_bound} satisfies
\begin{align*}
\nsumn{i}P_{VX}[\nweight_{n,i}(v_\mhit{1},x,V,X,\secnu)]-\infnupar(v_\mhit{1},x)=\fr{P_{VX}K_h(v_\mhit{1},x, V_\mhit{1}, X)m(V,X)}{h^d p_{V_\mhit{1}X}(v_\mhit{1},x)}-\muX(v_\mhit{1},x) \\
=\fr{P_{VX}K_h(v_\mhit{1},x, V_\mhit{1}, X)\muX(V_\mhit{1},X)-h^d p_{V_\mhit{1}X}(v_\mhit{1},x)\muX(v_\mhit{1},x)}{h^d p_{V_\mhit{1}X}(v_\mhit{1},x)},
\end{align*}
by the tower property of expectations and the definition of $\muX$. Let $g(w)\ceq p_{V_\mhit{1}X}(w)\muX(w)$ for $w\ceq (v_\mhit{1},x)$. If $g$ has $\beta'$ continuous derivatives which are uniformly bounded, then a change-of-variables and a Taylor-expansion argument show that, if $\supnorm{\fr{1}{p_{V_\mhit{1}X}}}<\infty$, then the $P_{VX}$-integrated square of the last display is of the order $h^{2\beta'}$. Then the bias term in \eqref{priv:eq:nonparaestim_priv_bound} is $\bigOPs{P_{VZ}}{h^{2\beta'}}$.

Conclude that
\begin{align*}
P_{VX}(\muXchk-\muX)^2=\bigOPs{P_{VZ}}{\fr{1}{\alpha^2}\left(\fr{1}{nh^d}+h^{2\beta}\right)} + \bigOPs{P_{VZ}}{\fr{1}{\alpha^2 n h^d}}+\bigOPs{P_{VZ}}{h^{2\beta'}},
\end{align*}
which, apart from the $\alpha^{-2}$ factor deriving from privacy, is the usual nonprivate error rate of kernel estimators for $\beta=\beta'$. The privacy mechanism does introduce a bias through the estimation of the secondary-nuisance parameter $p_{V_\mhit{1}X}$ because of the $\alpha^{-1}$ factors in \eqref{priv:eq:estim_pvx_priv}.
\end{example}

\begin{example}[name={Regression $\muX$ --- Orthogonal Series Estimator}]
\label{priv:ex:regression_ortho}
Let $(\varphi_j)$, $j=1,2,\ldots$, be an orthonormal basis in $L_2(P_{V_\mhit{1}X})$, that is, $P_{V_\mhit{1}X}(\varphi_j\varphi_k)=\indic{k=j}$ for all $j,k\geq1$ and $\mu = \sum_{j=1}^\infty c_j \varphi_j$ for projection coefficients  $c_j\ceq P_{V_\mhit{1}X}(\mu\varphi_j)$ for all $\mu\in L_2(P_{V_\mhit{1}X})$. If $(\varphi_j)$ were known, one could estimate $\muX$ in the nonprivate setting by $\sum_{j=1}^J \hat c_j \varphi_j$ with $ \hat c_j \ceq \nsumn{i} m(V_i',X_i')\varphi_j(V_{\mhit{1}i}', X_i')$ and a positive integer $J=J_n$ tending to infinity. However, we cannot construct a basis $(\varphi_j)$ in $L_2(P_{V_\mhit{1}X})$, because $P_{V_\mhit{1}X}$ itself is unknown. A remedy to this is to assume that $\muX$ is in the smaller space $L_2(\mesVXsub{\mhit{1}})\subseteq L_2(P_{V_\mhit{1}X})$, where $\mesVXsub{\mhit{1}}$ is the dominating measure of $P_{V_\mhit{1}X}$. For instance, with $(V_\mhit{1},X)$ taking values in a subspace of $\real^d$ and $\mesVXsub{\mhit{1}}$ the Lebesgue measure, this assumption necessitates that $|\muX(v_\mhit{1},x)|$ decay to zero as $v_\mhit{1}$ or $x$ tends  to infinity. Under this assumption, we can write $\muX=\sum_{j=1}^\infty a_j \phi_j$ for $a_j\ceq \mesVXsub{\mhit{1}}(\mu\phi_j)$ and an orthonormal basis $(\phi_j)$ in $L_2(\mesVXsub{\mhit{1}})$. Correspondingly, if ${\mathcal{S}}'=((V_i',X_i'))_{i\in[n]}$ were observed, we could estimate $\muX$ by
\begin{align}
\muXhat(v_\mhit{1},x)\ceq \sum_{j=1}^J \hat a_j \phi_j(v_\mhit{1},x), \quad \hat a_j\ceq \nsumn{i} \fr{m(V_i',X_i')\phi_j(V_{\mhit{1}i}', X_i')}{\hat p_{V_\mhit{1}X}(V_{\mhit{1}i}',X_i')},
\end{align}
for $J=J_n$ tending to infinity with $n$. We inversely weight with the estimated density $\hat p_{V_\mhit{1}X}$ to correct for having the basis in $L_2(\mesVXsub{\mhit{1}})$ but using the empirical version of $P_{V_\mhit{1}X}$ to compute the projection coefficients, since $P_{V_\mhit{1}X}\left(\fr{\mu\phi_j}{p_{V_\mhit{1}X}}\right)= \mesVXsub{\mhit{1}}(\mu\phi_j)=a_j$. We can rewrite $\muXhat$ in the form \eqref{priv:eq:nonparaestim} as
\begin{align*}
\muXhat(v_\mhit{1},x)&= \nsumn{i} \nweight_{n,i}(v_\mhit{1}, x, V_{i}',X_i',\hat p_{V_\mhit{1}X}), \\
 \nweight_{n,i}(v_\mhit{1}, x, v',x', p_{V_\mhit{1}X})&\ceq \sum_{j=1}^J \fr{m(v',x')\phi_j(v_\mhit{1}',x')}{p_{V_\mhit{1}X}(v_\mhit{1}',x')}\phi_j(v_\mhit{1},x),
\end{align*}
and modify it according to \eqref{priv:eq:nonparaestim_priv} for private estimation from $\bar{\mathcal{S}}'=((V_i',Z_i'))_{i\in[n]}$ as
\begin{align*}
\muXchk(v_\mhit{1},x)&\ceq \nsumn{i} \bar\nweight_{n,i}(v_\mhit{1}, x, V_{i}',Z_i',\chk p_{V_\mhit{1}X}), \\
\bar\nweight_{n,i}(v_\mhit{1}, x, v',z',\chk p_{V_\mhit{1}X}) &= \left(\linopQXinv (v',x')\mapsto \sum_{j=1}^J \fr{m(v',x')\phi_j(v_\mhit{1}',x')}{\chk p_{V_\mhit{1}X}(v_\mhit{1}',x')}\phi_j(v_\mhit{1},x) \right)(v',z') \\
&=\sum_{j=1}^J \left(\linopQXinv (v',x')\mapsto \fr{m(v',x')\phi_j(v_\mhit{1}',x')}{\chk p_{V_\mhit{1}X}(v_\mhit{1}',x')}\right)(v',z')\phi_j(v_\mhit{1},x), \\
&\eqc \sum_{j=1}^J (\linopQXinv\chk\varrho_j)(v',z')\phi_j(v_\mhit{1},x),
\end{align*}
by the linearity of $\linopQXinv$, where $\check p_{V_\mhit{1}X}$ is an arbitrary estimator of $p_{V_\mhit{1}X}$ computed from $\bar{\mathcal{S}}'$. Next, we study $\muXchk$ in relation to \Cref{priv:prop:priv_bound}, assuming that $Q\in\setQgen$.

The term $T_{n}$ in \eqref{priv:eq:secnubound} is 
\begin{align*}
T_n(v_\mhit{1},x)&=\nsumn{i} \bar\nweight_{n,i}(v_\mhit{1}, x, V_{i}',Z_i',\chk p_{V_\mhit{1}X})-\bar\nweight_{n,i}(v_\mhit{1}, x, V_{i}',Z_i',p_{V_\mhit{1}X}) \nonumber \\
&=\sum_{j=1}^J \left(\nsumn{i}(\linopQXinv(\chk\varrho_j-\varrho_j))(V_i',Z_i')\right) \phi_j(v_\mhit{1},x)\eqc \sum_{j=1}^J \Delta_{n,j}\phi_j(v_\mhit{1},x), 
\end{align*}
where $\varrho_j(v',x')\ceq\fr{m(v',x')\phi_j(v_\mhit{1}',x')}{p_{V_\mhit{1}X}(v_\mhit{1}',x')}$. The $\Delta_{n,j}$ do not depend on $(v_\mhit{1},x)$. Because $T_n^2$ is nonnegative, 
\begin{align*}
P_{VX}T_n^2\leq \supnorm{p_{V_\mhit{1}X}}\sum_{j=1}^J\sum_{k=1}^J\Delta_{n,j}\Delta_{k,n}P_{VX}\left(\fr{\phi_j\phi_k}{p_{V_\mhit{1}X}}\right) = \supnorm{p_{V_\mhit{1}X}}\sum_{j=1}^J \Delta_{n,j}^2
\end{align*}
as $P_{VX}\left(\fr{\phi_j\phi_k}{p_{V_\mhit{1}X}}\right)=\mesVXsub{\mhit{1}}(\phi_j\phi_k)=\indic{j=k}$ by the orthonormality of the $(\phi_j)$. Assume that 
\begin{align*}
\supnorm{\fr{m}{p_{V_\mhit{1}X}}}\supnorm{p_{V_\mhit{1}X}}<\infty, \quad \sum_{j=1}^J\supnorm{\phi_j}^2=\bigO{J}.
\end{align*}
As $\supnorm{\linopQXinv h}\leq \fr{2}{\alpha} \supnorm{h}$, $P_{VX}T_n^2=\bigO{\fr{J}{\alpha^2}\supnorm{\fr{\chk p_{V_\mhit{1}X} - p_{V_\mhit{1}X}}{\chk p_{V_\mhit{1}X}}}^2}$.

To bound the variance term in \eqref{priv:eq:nonparaestim_priv_bound}, we use \eqref{priv:eq:sigmacondbound} under $Q\in\setQgen$:
\begin{align*}
\Vb{\ba \nweight_{n,i}(v_\mhit{1}, x,V,Z,p_{V_\mhit{1}X})}\leq \E \nweight_{n,i}^2(v_\mhit{1}, x,V,Z,p_{V_\mhit{1}X}) \\
=\Eb{ \fr{m^2(V,Z)p_{V_\mhit{1}Z}(V_\mhit{1}, Z)p_{V_\mhit{1}X}(v_\mhit{1}, x)}{p_{V_\mhit{1}X}^2(V_\mhit{1}, Z)}\fr{\left(\sum_{j=1}^J \phi_j(V_\mhit{1}, Z)\phi_j(v_\mhit{1}, x)\right)^2}{p_{V_\mhit{1}Z}(V_\mhit{1}, Z)p_{V_\mhit{1}X}(v_\mhit{1}, x)}} \\
\leq \supnorm{\fr{m^2p_{V_\mhit{1}Z}}{p_{V_\mhit{1}X}^2}}\supnorm{p_{V_\mhit{1}X}}\sum_{j=1}^J\sum_{k=1}^J \Eb{\fr{\phi_j\phi_k}{p_{V_\mhit{1}Z}}(V,Z)}\fr{\phi_j\phi_k}{p_{V_\mhit{1}X}}(v_\mhit{1}, x).
\end{align*}
Here the expectation equals $\indic{j=k}$, and likewise $P_{V_\mhit{1}X}\left(\fr{\phi_j\phi_k}{p_{V_\mhit{1}X}}\right)=\indic{j=k}$. Hence, if $\supnorm{\fr{m^2p_{V_\mhit{1}Z}}{p_{V_\mhit{1}X}^2}}\supnorm{p_{V_\mhit{1}X}}<\infty$, then $\int \Vb{\ba \nweight_{n,i}(v_\mhit{1}, x,V,Z,p_{V_\mhit{1}X})} \dderiv P_{V_\mhit{1}X}(v_\mhit{1}, x)\lesssim J$.

To bound the bias in \eqref{priv:eq:nonparaestim_priv_bound}, first write by the tower property of expectations, expanding $\muX=\sum_{j=1}^\infty a_j\phi_j$ for $a_j=P_{VX}(\muX\phi_j)$,
\begin{align*}
\nsumn{i}P_{VX}[\nweight_{n,i}(v_\mhit{1},x,V,X,\secnu)] &= P_{VX}\left[\sum_{j=1}^J \fr{m(V,X)\phi_j(V_\mhit{1},X)}{p_{V_\mhit{1}X}(V_\mhit{1}X)}\phi_j(v_\mhit{1},x)\right] \\
&=P_{VX}\left[\left(\sum_{j=1}^J \fr{\phi_j(V_\mhit{1},X)}{p_{V_\mhit{1}X}(V_\mhit{1}X)}\phi_j(v_\mhit{1},x)\right)\sum_{j=1}^\infty a_j \phi_j(V_\mhit{1},X)\right] \\
&=\sum_{j=1}^J a_j \phi_j(v_\mhit{1},x),
\end{align*}
where we used the orthonormality of  $(\phi_j)$. Hence, $$\left(\sum_{j=1}^J a_j \phi_j(v_\mhit{1},x)-\muX(v_\mhit{1},x)\right)^2=\left(\sum_{j=J+1}^\infty a_j \phi_j(v_\mhit{1},x)\right)^2.$$ A nonnegative function, its integral is bounded by
\begin{align*}
P_{VX}\left(\sum_{j=J+1}^\infty a_j \phi_j(V_\mhit{1},X)\right)^2\leq \supnorm{p_{V_\mhit{1}X}}\mesVXsub{\mhit{1}}\left(\sum_{j=J+1}^\infty a_j \phi_j(V_\mhit{1},X)\right)^2 \leq \supnorm{p_{V_\mhit{1}X}}\sum_{j=J+1}^\infty a_j^2,
\end{align*}
again by the orthonormality of $(\phi_j)$. Thus, the bias in \eqref{priv:eq:nonparaestim_priv_bound} is bounded by $\bigO{\sum_{j=J+1}^\infty a_j^2}$.

Conclude that 
\begin{align}
P_{VX}(\muXchk-\muX)^2=\bigO{\fr{J}{\alpha^2}\supnorm{\fr{\chk p_{V_\mhit{1}X} - p_{V_\mhit{1}X}}{\chk p_{V_\mhit{1}X}}}^2} + \bigOPs{P_{VZ}}{\fr{J}{\alpha^2 n}} + \bigO{\sum_{j=J+1}^\infty a_j^2}, \label{priv:eq:ortseries_muxpriv_bound}
\end{align}
where the first term dominates the second one. Apart from the $\alpha^{-2}$ privacy factor in the second term, the last two term correspond to the usual bound for orthogonal series estimates under equidistant or uniformly distributed design $(V_\mhit{1},X)$. The first term derives from of the weighting correction to accommodate the unknown density $p_{V_\mhit{1}X}$. This first term usually results in a suboptimal rate. Indeed, suppose that $\muX$ only depends on $X$ and that $X_i=\fr{i}{n}$. If $\muX$ belongs to a Sobolev smoothness class $S(\beta,L)$ for an $L>0$ and integer $\beta>0$, so that the $(\beta-1)$-th derivative of $\muX$ is absolutely continuous and its integrated squared $\beta$-th derivative is bounded by $L$, then we expect $P_{VX}(\muXhat-\muX)^2=\bigO{\fr{J}{n}}+\bigO{J^{-2\beta}}$ for a trigonometric basis (see \citet[Chapter 1.7]{tsybakov_introduction_2009}). If $\beta$ is known, a suitable choice of $J\sim n^\fr{1}{2\beta+1}$ gives a rate $n^\fr{-2\beta}{2\beta+1}$. In contrast, if $\supnorm{\hat p_{X}- p_{X}}=\left(\fr{\log n}{n}\right)^{a}$ for some $0<a\leq\fr{1}{2}$ --- a common case for nonparametric density estimation under regularity conditions --- and $\supnorm{\fr{1}{\hat p_{X}}}<\infty$, then the choice of $J$ balancing the terms in \eqref{priv:eq:ortseries_muxpriv_bound} in the nonprivate setting is $J\sim \left[\left(\fr{\log n}{n}\right)^{2a}+\fr{1}{n}\right]^\fr{-1}{2\beta+1}$, yielding a rate of $\left[\left(\fr{\log n}{n}\right)^{2a}+\fr{1}{n}\right]^\fr{2\beta}{2\beta+1}$ for $\muXhat$. This is of the order $\left(\fr{\log n}{n}\right)^{2a\fr{2\beta}{2\beta+1}}$, so a loss of $n^{2a}$ is incurred. 

To prevent such a loss, \cite{kohler_multivariate_2008} constructs an orthonormal basis in $L_2(\Probn')$ for the empirical distribution $\Probn'$ of $((V_i',X_i'))_{i\in[n]}$. Such a basis is constructed using the whole sample $((V_{\mhit{1}i}',X_i'))_{i\in[n]}$ and for that reason it is unclear how it would lend itself for transformation to the private setting and analysis by \Cref{priv:prop:priv_bound}. 
\end{example}

\begin{example}[name={Riesz Representer $\rieszX$ --- Orthogonal Series Estimator}]
\label{priv:ex:riesz_ortho}
The Riesz representer $\rieszX=\rieszXs{\gammaV(c)}$ can be estimated by orthogonal series similarly to the regression in \Cref{priv:ex:regression_ortho}.  Assume that $\rieszX\in L_2(\mesVXsub{\mhit{1}})\subseteq L_2(P_{V_\mhit{1}X})$, and let  $(\phi_j)$, $j=1,2,\ldots$, be an orthonormal basis in $L_2(\mesVXsub{\mhit{1}})$. Then we can write $\rieszX=\sum_{j=1}^\infty a_j \phi_j$, where the projection coefficients $a_j\ceq \mesVXsub{\mhit{1}}(\rieszXs{\gammaV(c)}\phi_j)$ depend on $\gammaV(c)$. Assume that  $\phi_j/p_{V_\mhit{1}X}\in L_2(P_{V_\mhit{1}X})$ for all $j\geq 1$. Then, using the Riesz property ``backwards,'' we have 
$$a_j=\mesVXsub{\mhit{1}}(\rieszXs{\gammaV(c)}\phi_j)=P_{V_\mhit{1}X}\left(\rieszXs{\gammaV(c)}\fr{\phi_j}{p_{V_\mhit{1}X}}\right)=P_{VX}\left(f\left(V,X,\fr{\phi_j}{p_{V_\mhit{1}X}},\gammaV(c)\right)\right).$$
This is useful, because the form of $\rieszXs{\gammaV(c)}$ may be unknown, so constructing estimates $\hat a_j$ from the empirical version of $P_{VX}\left(\rieszXs{\gammaV(c)}\fr{\phi_j}{p_{V_\mhit{1}X}}\right)$ 
may not be feasible; but $f$ is known, so estimates can be derived from the rightmost side of the display. Indeed, we can estimate $\rieszX$ in the nonprivate setting in the form \eqref{priv:eq:nonparaestim} as
\begin{align*}
\rieszXhat(v_\mhit{1},x)&\ceq \nsumn{i} \nweight_{n,i}(v_\mhit{1}, x, V_{i}',X_i',\hat p_{V_\mhit{1}X}, \gammaVhat(c)), \\
 \nweight_{n,i}(v_\mhit{1}, x, v',x', p_{V_\mhit{1}X}, \gamma)&\ceq \sum_{j=1}^J f\left(v',x', \fr{\phi_j}{p_{V_\mhit{1}X}},\gamma \right)\phi_j(v_\mhit{1},x),
\end{align*} 
for a positive integer $J=J_n$ tending to infinity with $n$; some estimator $\gammaVhat(c)$ of $\gammaV(c)$, such as \eqref{priv:eq:gammaVhat}, and $\hat p_{V_\mhit{1}X}$, an estimator of $p_{V_\mhit{1}X}$, both computed from $\mathcal{S}'=((V_i',X_i'))_{i\in[n]}$.


Transferring $\rieszXhat$ to the private setting according to \eqref{priv:eq:nonparaestim_priv}, we obtain
\begin{align*}
\rieszXchk(v_\mhit{1},x)&\ceq \nsumn{i} \bar\nweight_{n,i}(v_\mhit{1}, x, V_{i}',Z_i',\chk p_{V_\mhit{1}X}, \gammaVchk(c)), \\
\bar\nweight_{n,i}(v_\mhit{1}, x, v',z', p_{V_\mhit{1}X}, \gamma) &= \sum_{j=1}^J \ba f\left(v',z',\fr{\phi_j}{ p_{V_\mhit{1}X}},\gamma\right)\phi_j(v_\mhit{1},x),
\end{align*}
where $\ba f(\cdot,\mu,\gamma)=(\linopQXinv (v,x)\mapsto f(v,x, \mu, \gamma))(\cdot)$, and $\gammaVchk(c)$ is an estimator of $\gammaV(c)$, such as \eqref{priv:eq:gammaVchk}, and $\chk p_{V_\mhit{1}X}$ of $p_{V_\mhit{1}X}$; all computed from $\bar{\mathcal{S}}'=((V_i',Z_i'))_{i\in[n]}$. Let us turn to the implications of \Cref{priv:prop:priv_bound}, which are similar to those in \Cref{priv:ex:regression_ortho}, except now we have two secondary-nuisance parameters $\secnu= (\gammaV(c),p_{V_\mhit{1}X})$. We assume that $Q\in\setQgen$ in \eqref{priv:eq:qtv}.

The term $T_n$ in \eqref{priv:eq:secnubound} is
\begin{align*}
T_{n}(v_\mhit{1}, x)=&\, \left\{\nsumn{i} \bar\nweight_{n,i}(v_\mhit{1}, x, V_{i}',Z_i',\chk p_{V_\mhit{1}X}, \gammaVchk(c)) - \nsumn{i} \bar\nweight_{n,i}(v_\mhit{1}, x, V_{i}',Z_i',\chk p_{V_\mhit{1}X}, \gammaV(c)) \right\} \\
&+ \left\{\nsumn{i} \bar\nweight_{n,i}(v_\mhit{1}, x, V_{i}',Z_i',\chk p_{V_\mhit{1}X}, \gammaV(c)) - \nsumn{i} \bar\nweight_{n,i}(v_\mhit{1}, x, V_{i}',Z_i', p_{V_\mhit{1}X}, \gammaV(c))\right\}  \\
\eqc&\, T_{n1}(v_\mhit{1}, x) + T_{n2}(v_\mhit{1}, x).
\end{align*}
By the mean-value theorem, $T_{n1}(v_\mhit{1}, x)=(\gammaVchk(c)-\gammaV(c))\sum_{j=1}^J D_{n,j} \phi_j(v_\mhit{1}, x)$, where the $D_{n,j}\ceq \nsumn{i} \partial_\gamma \ba f \left(V_i',Z_i', \fr{\phi_j}{\chk p_{V_\mhit{1}}X} ,\gammaVtld(c) \right)$ do not depend on $(v_\mhit{1}, x)$, and $\gammaVtld(c)$ is some value between $\gammaVchk(c)$ and $\gammaV(c)$. Because $V_\mhit{2}$ is discretely distributed, $\gammaVchk(c)-\gammaV(c)=\bigOPs{P_{VZ}}{n^{-1/2}}$ for any reasonable estimator such as \eqref{priv:eq:gammaVchk}. Then, along arguments in \Cref{priv:ex:regression_ortho},  $$P_{VX}T_{n1}^2 \leq  \bigOPs{P_{VZ}}{1/n} \supnorm{p_{V_\mhit{1}X}} \sum_{j=1}^J D_{n,j}^2$$ by the orthonormality of $(\phi_j)$. Assume $\sum_{j=1}^J D_{n,j}^2=\bigOPs{P_{VZ}}{\fr{J}{\alpha^2}}$, which is reasonable for $\partial_\gamma f$ bounded in all its arguments and privacy mechanism \eqref{priv:eq:qtv}. Then, by arguments in \Cref{priv:ex:regression_ortho}, $P_{VX}T_{n1}^2=\bigOPs{P_{VZ}}{\fr{J}{\alpha^2n}}$ provided $\supnorm{p_{V_\mhit{1}X}}<\infty$. 

By the linearity of $f$ and thus $\ba f$ in the regression argument, $T_{2n}(v_\mhit{1}, x)=\sum_{j=1}^J \Delta_{n,j} \phi_j(v_\mhit{1}, x)$ for $\Delta_{n,j}\ceq \nsumn{i} \ba f\left(V_i',Z_i',\phi_j\cdot\left(\fr{1}{ \chk p_{V_\mhit{1}X}}-\fr{1}{p_{V_\mhit{1}X}}\right),\gammaV(c)\right)$. Assuming $\supnorm{p_{V_\mhit{1}X}}<\infty$ and 
$$\sum_{j=1}^J \Delta_{n,j}^2=\bigOPs{P_{VZ}}{\fr{J}{\alpha^2}\supnorm{\fr{\chk p_{V_\mhit{1}X} - p_{V_\mhit{1}X}}{\chk p_{V_\mhit{1}X}}}^2},$$
we obtain 
$$P_{VX}T_n^2=\bigOPs{P_{VZ}}{\fr{J}{\alpha^2n}}+\bigOPs{P_{VZ}}{\fr{J}{\alpha^2}\supnorm{\chk p_{V_\mhit{1}X}-p_{V_\mhit{1}X}}^2}=\bigOPs{P_{VZ}}{\fr{J}{\alpha^2}\supnorm{\chk p_{V_\mhit{1}X}-p_{V_\mhit{1}X}}^2}.$$

We bound the variance term in \eqref{priv:eq:nonparaestim_priv_bound} by \eqref{priv:eq:sigmacondbound} under $Q\in\setQgen$. Assume that $\supnorm{p_{V_\mhit{1}X}}<\infty$ and
\begin{align*}
\sum_{j=1}^J \E f^2\left(V,Z,\fr{\phi_j}{p_{V_\mhit{1}X}},\gammaV(c)\right) = \bigO{J}.
\end{align*}
Then the arguments in \Cref{priv:ex:regression_ortho} give $\int \Vb{\ba \nweight_{n,i}(v_\mhit{1}, x,V,Z,p_{V_\mhit{1}X},\gammaV(c))} \dderiv P_{V_\mhit{1}X}(v_\mhit{1}, x)\lesssim J$.

The bias in \eqref{priv:eq:nonparaestim_priv_bound} can be bounded as in \Cref{priv:ex:regression_ortho}, yielding the bound $\bigO{\sum_{j=J+1}^\infty a_j^2}$ for $a_j= \mesVXsub{\mhit{1}}(\rieszXs{\gammaV(c)}\phi_j)$.

Conclude that 
\begin{align*}
P_{VZ}(\rieszXchk-\rieszX)^2=\bigOPs{P_{VZ}}{\fr{J}{\alpha^2}\supnorm{\fr{\chk p_{V_\mhit{1}X} - p_{V_\mhit{1}X}}{\chk p_{V_\mhit{1}X}}}^2}+ \bigOPs{P_{VZ}}{\fr{J}{\alpha^2n}} + \bigO{\sum_{j=J+1}^\infty a_j^2},
\end{align*}
where the first term dominates the second one. See \Cref{priv:ex:regression_ortho} for a discussion of these terms; in particular, on how the estimated secondary-nuisance $p_{V_\mhit{1}X}$ of the first term leads to a rate slower than what only the last two terms would imply.

\end{example}



\section{Nonprivate Estimation}
\label{priv:app:sec:nonprivate_estimation}

Our aim is private inference, but given that our parameter class is novel, we also present results on nonprivate inference in \Cref{priv:app:sec:nonprivate_estimation:results}, proven in \Cref{priv:app:sec:nonprivate_estimation:proofs}.

\subsection{Results}
\label{priv:app:sec:nonprivate_estimation:results}

In this section, we estimate $\chi(P_{VX})$ of \eqref{priv:eq:dr_para} in the nonprivate setting, from random samples from $(V,X)\sim P_{VX}$ in the nonparametric model \eqref{priv:eq:nonparamodel_set}. For estimation, we further require the approximability conditions
\begin{align}
\label{priv:eq:sq_continu}\tag{C.SC}
\begin{aligned}
\Eb{\{f(V,X,\mu,\gamma)-f(V,X,\muX,\gammaV(c))\}^2}&\to 0 \\
\Eb{\{\partial_\gamma f(V,X,\mu,\gamma)-\partial_\gamma f(V,X,\muX,\gammaV(c))\}^2}&\to 0
\end{aligned}
\end{align}
as $\rho((\mu,\gamma),(\muX,\gammaV(c)))\to0$. Also note that the linearity \eqref{priv:eq:f_lin} implies the $P_{VX}$-a.s.\ linearity of $\mu\mapsto \fr{\partial^j f}{\partial \gamma^j}(V,X,\mu,\ba\gamma)$ for any integer $j\geq 1$ for which the derivative exists. This follows directly from the definition of the derivative.

Recall the one-step estimator
\begin{align*}
\hat\chi_n= \chi(\hat P_{VX})+\Probn \hat{\tld\chi} =  \chi(\hat P_{VX})+\nsumn{i}\hat{\tld\chi}(V_i,X_i).
\end{align*}
from \eqref{priv:eq:bias_corrected_chihat}. Here, the estimator of the influence function \eqref{priv:eq:eif_chi} is
\begin{align}
\label{priv:eq:eif_chi_hat}
\begin{aligned}
\hat{\tld\chi}(v,x)\ceq&\, \rieszXhat(v_{\mhit{1}},x)(m(v,x)-\muXhat(v_{\mhit{1}},x))+\fr{\indic{v_{\mhit{2}}=c}}{\hat p_{V_{\mhit{2}}}(c)}(g(v,x)-\gammaVhat(c))\expderivhat \\
&+ f(v,x, \muXhat, \gammaVhat(c))-\chi(\hat P_{VX}), 
\end{aligned}
\end{align}
where $\rieszXhat,\muXhat$, taking values in $L_2(P_{V_{\mhit{1}}X})$, are some estimators of $\rieszX,\muX$, respectively; $\hat p_{V_{\mhit{2}}}(c)$, taking values in $\real$, is some estimator of $p_{V_{\mhit{2}}}(c)$; and
$\expderivhat$, taking values in $\real$, is some estimator of 
\begin{align}
\expderiv\ceq\E \partial_\gamma f(V,X, \muX, \gammaV(c)).\label{priv:eq:e_def}
\end{align}
Note that we use the same initial estimator $\chi(\hat P_{VX})$ in $\hat{\tld\chi}$, which we may set to
\begin{align}
\chi(\hat P_{VX})=\Probn f(V,X, \muXhat, \gammaVhat(c))=\nsumn{i} f(V_i,X_i, \muXhat, \gammaVhat(c)) \label{priv:eq:plugin_chihat}.
\end{align}
Therefore,
\begin{align*}
\hat\chi_n =&\, \chi(\hat P_{VX})+\Probn \hat{\tld\chi}  \\
	=&\, \nsumn{i}\left\{ \rieszXhat(v_{\mhit{1}i},X_i)(m(V_i,X_i)-\muXhat(v_{\mhit{1}i},X_i))+\fr{\indic{V_{\mhit{2}i}=c}}{\hat p_{V_{\mhit{2}}}(c)}(g(V_i,X_i)-\gammaVhat(c))\expderivhat \right. \\
	 &+\left. \vphantom{\fr{\indic{V_{\mhit{2}i}=c}}{\hat p_{V_{\mhit{2}}}(c)}}  f(V_i,X_i, \muXhat, \gammaVhat(c)) \right\}.
\end{align*}

We assume that the estimators 
\begin{align}
 \label{priv:eq:nuisance}
 \begin{aligned}
\hat\eta&\ceq (\rieszXhat,\muXhat,\gammaVhat(c),\hat p_{V_{\mhit{2}}}(c), \expderivhat)\in L_2(P_{V_{\mhit{1}}X})\times L_2(P_{V_{\mhit{1}}X}) \times\Gamma\times\real\times\real\,\text{ of } \\ 
\eta&\ceq (\rieszX,\muX,\gammaV(c), p_{V_{\mhit{2}}}(c), \expderiv)
\end{aligned}
\end{align}
are computed from random samples from $P_{VX}$ which are independent
of $$ \mathcal{S}\ceq\mathcal((V_i,X_i))_{i\in[n]}.$$ Specifically, we assume that there are two more random samples 
\begin{align*}
 \mathcal{S}'\ceq ((V_i',X_i'))_{i\in[n]}\,\text{ and }\,  \mathcal{S}''\ceq((V_i'',X_i''))_{i\in[n]}
\end{align*}
from $P_{VX}$, with $\mathcal{S},\mathcal{S}',\mathcal{S}''$ mutually independent, where $\mathcal{S'}$ is used for the estimation of $(\rieszX,\muX,\gammaV(c), p_{V_{\mhit{2}}}(c))$, and $\mathcal{S}',\mathcal{S}''$ are used for the estimation of $\expderiv$ in \eqref{priv:eq:e_def} as
\begin{align}
\expderivhat \ceq \Probn'' \partial_\gamma f(V,X,\muXhat,\gammaVhat(c))\ceq \nsumn{i} \partial_\gamma f(V_i'',X_i'',\muXhat,\gammaVhat(c)). \label{priv:eq:expderivhat}
\end{align}
With this estimation strategy, we can establish the consistency of $\expderivhat$, and hence of $\hat{\tld\chi}$ and $\hat\chi_n$ in turn, without additional regularity conditions. Decompose
\begin{align}
\sqrt{n}(\hat\chi_n - \chi(P_{VX})) &= \sqrt{n}\Probn \tld\chi+\sqrt{n}(\Probn-P_{VX})(\hat{\tld{\chi}}-\tld\chi)+\sqrt{n}R_n, \label{priv:eq:chi_eff_decomp} \\
R_n&\ceq \chi(\hat P_{VX})-\chi(P_{VX})+P_{VX} \hat{\tld\chi}, \label{priv:eq:chi_eff_decomp_bias}
\end{align}
where we used that $P_{VX}\tld\chi=0$ by $\tld\chi$ being the influence function. The term $\sqrt{n}\Probn \tld\chi\rsquigs{P_{VX}}\normaldist(0,P_{VX}\tld\chi^2)$ by the standard central limit theorem. The second term in \eqref{priv:eq:chi_eff_decomp} is called the empirical process term and is vanishing as $\smallOPs{P_{VX}}{1}$ under consistent estimators $\hat\eta$ and stochastic boundedness conditions.

\begin{assumption}[Consistent Estimators]
\label{priv:ass:consistent_nuisance}
It holds that 
\begin{align}
\norm{L_2(P_{VX})}{\rieszXhat-\rieszX}&=\smallOPs{P_{VX}}{1} \label{priv:eq:consistency_riesz_l2}, \\
\gammaVhat(c)-\gammaV(c)&=\smallOPs{P_{VX}}{1}, \label{priv:eq:consistency_gammaVhat} \\
\hat p_{V_2}(c)-p_{V_2}(c)&=\smallOPs{P_{VX}}{1}. \label{priv:eq:consistency_phat}
\end{align}
Further, it either holds that
\begin{align}
\supnorm{m-\muX}&=\bigO{1}, \label{priv:eq:consistency_bigobound} \\
\supnorm{\muX-\muXhat}&=\smallOPs{P_{VX}}{1},  \label{priv:eq:consistency_mu_supn}
\end{align}
or that
\begin{align}
\supnorm{m-\muXhat}&=\bigOPs{P_{VX}}{1}, \label{priv:eq:consistency_bigobound_hat} \\
\norm{L_2(P_{VX})}{\muX-\muXhat}&=\smallOPs{P_{VX}}{1},  \label{priv:eq:consistency_mu_l2} \\
P_{VX}\left(\set{(V_{\mhit{1}},X)\in \Vimagesp_{\mhit{1}}\times\Ximagesp: |r(V_{\mhit{1}},X)|>\ba R }\right)&=0  \label{priv:eq:consistency_riesz_bound}
\end{align}
for some constant $\ba R<\infty$, where \eqref{priv:eq:consistency_riesz_bound} may be replaced by $\supnorm{\rieszX}<\infty$.
\end{assumption}

\begin{lemma}[Vanishing Empirical Process Term]
\label{priv:lem:chi_eff_empprocess}
Assume that \Cref{priv:ass:consistent_nuisance} holds for $\hat\eta$ in \eqref{priv:eq:nuisance}. 
Then $\expderivhat-\expderiv=\smallOPs{P_{VX}}{1}$ and $(\Probn-P_{VX})(\hat{\tld{\chi}}-\tld\chi)=\smallOPs{P_{VX}}{n^{-1/2}}$. 
\end{lemma}

Hence, if the nuisance parameters in $\eta$ are consistently estimated and bound conditions apply, the behaviour of $\sqrt{n}(\hat\chi_n-\chi(P_{VX}))$ is governed by the second-order bias term $R_n$ in \eqref{priv:eq:chi_eff_decomp_bias}. We show that our class of parameters \eqref{priv:eq:dr_para} enjoys a rate-double-robustness property, therefore  $R_n$ exhibits a product-structure of estimation errors.
 

\Cref{priv:thm:dr} implies that the bias in \eqref{priv:eq:chi_eff_decomp_bias} is 
\begin{align}
 \label{priv:eq:bias_r}
 \begin{aligned}
R_n =&\ \chi(\hat P_{VX})-\chi(P_{VX})+P_{VX}\hat{\tld\chi} \\
=&\ -P_{VX}(\rieszX-\rieszXhat)(\muX-\muXhat) +(\gammaV(c)-\gammaVhat(c))\left(\fr{p_{V_{\mhit{2}}}(c)}{\hat p_{V_{\mhit{2}}}(c)}\expderivhat-\expderiv'\right) \\
&-(\gammaV(c)-\gammaVhat(c))^2\fr{P_{VX}\partial_\gamma^2 f(V,X,\muXhat,\gammaVtld(c))}{2},  
\end{aligned}
\end{align}
for some $\gammaVtld(c)$ between $\gammaV(c)$ and $\gammaVhat(c)$, and
\begin{align}
\expderiv' \ceq&\, P_{VX}\partial_\gamma f(V,X,\muXhat,\gammaVhat(c)). \label{priv:eq:derivhatprime}
\end{align}
Because $V_{\mhit{2}}$ is distributed on a finite set, any reasonable estimators of $\gammaV(c)$, $p_{V_{\mhit{2}}}(c)$ are $\sqrt{n}$-consistent; for instance
\begin{align}
\gammaVhat(c) \ceq \fr{1}{N_c}\sumn{i}\indic{V_{\mhit{2}i}'=c} g(V_i',X_i'), \quad \hat p_{V_{\mhit{2}}}(c)\ceq N_c/n, \quad N_c\ceq \sumn{i}\indic{V_{\mhit{2}i}'=c} \label{priv:eq:gammaVhat}
\end{align}
satisfy $\gammaVhat(c)-\gammaV(c)=\bigOPs{P_{VX}}{n^{-1/2}}$, $\hat p_{V_{\mhit{2}}}(c)-p_{V_{\mhit{2}}}(c)=\bigOPs{P_{VX}}{n^{-1/2}}$ by the standard central limit theorem. Suppose that $\expderivhat-\expderiv'=\smallOPs{P_{VX}}{1}$ and that $P_{VX} \partial_\gamma^2 f(V,X,\muXhat,\gammaVtld(c))=\bigOPs{P_{VX}}{1}$. It follows from \eqref{priv:eq:bias_r} that the bias is then
\begin{align}
R_n = -P_{VX}(\rieszX-\rieszXhat)(\muX-\muXhat) + \smallOPs{P_{VX}}{n^{-1/2}}. \label{priv:eq:chihat_dr}
\end{align}
Hence, $R_n$ is ultimately determined by the product of the estimation errors of the Riesz representer $\rieszX$ and the regression function $\muX$. 
For example, we find that the average treatment effect on the treated is double robust. This is aligned with \citet[Example 6]{chernozhukov_automatic_2022}, and is an improvement on \citet[Example 12]{rotnitzky_characterization_2021}, who too, establish asymptotic normality, but not efficiency, as this parameter is not natively included in their class.

Our results amounts to the asymptotic efficiency of $\hat\chi_n$ under fast enough estimation rates, boundedness conditions, and the estimation strategy using independent samples.

\begin{assumption}[Rates of Estimators]
\label{priv:ass:rates_nuisance}
For $\hat\eta$ in \eqref{priv:eq:nuisance},
\begin{align*}
P_{VX}(\rieszX-\rieszXhat)(\muX-\muXhat)&=\smallOPs{P_{VX}}{n^{-1/2}},  \\
\gammaVhat(c)-\gammaV(c)&=\bigOPs{P_{VX}}{n^{-1/2}}, \\
\hat p_{V_{\mhit{2}}}(c)-p_{V_{\mhit{2}}}(c)&=\bigOPs{P_{VX}}{n^{-1/2}}, \\
P_{VX} \partial_\gamma^2 f(V,X,\muXhat,\gammaVtld(c))&=\bigOPs{P_{VX}}{1}.
\end{align*} 
\end{assumption}

\begin{corollary}[Asymptotic Efficiency of $\hat\chi_n$]
\label{priv:cor:eff_chihat}
If \Cref{priv:ass:consistent_nuisance,priv:ass:rates_nuisance} hold, then
$\sqrt{n}(\hat\chi_n-\chi(P_{VX}))\rsquigs{P_{VX}}\normaldist(0, P_{VX}\tld\chi^2)$ as $n\to\infty$.
\end{corollary}

Conveniently and expectedly, when $(X,V_{\mhit{1}})$ is discretely distributed, the same limit is achievable without  the use of independent samples $\mathcal{S},\mathcal{S}',\mathcal{S}''$. Indeed, the plug-in estimator $\chi(\hat P_{VX})$ is asymptotically efficient when we compute the estimators $(\muXhat,\gammaVhat)$ and $\chi(\hat P_{VX})$ on the same sample $\mathcal{S}$, provided some boundedness conditions hold.

\begin{proposition}[Asymptotic Efficiency of the Plug-in Estimator for Discrete $(X,V_{\mhit{1}})$]
\label{priv:prop:asymeff_plugin_discretecovar}
Suppose that $\Ximagesp$ and $\Vimagesp_{\mhit{1}}$ are finite, and define
\begin{align*}
\chi(\hat P_{VX})\ceq&\, \nsumn{i} f(V_i,X_i,\muXhat,\gammaVhat(c)), \\
\muXhat(v_{\mhit{1}},x)\ceq&\, \fr{1}{N_{v_{\mhit{1}}x}}\sumn{i}\indic{V_{\mhit{1}i}=v_{\mhit{1}}, X_i=x}m(V_i,X_i), \quad N_{v_{\mhit{1}}x}\ceq \sumn{i}\indic{V_{\mhit{1}i}=v_{\mhit{1}}, X_i=x}, \\
\gammaVhat(c) \ceq&\, \fr{1}{N_c}\sumn{i}\indic{V_{\mhit{2}i}=c} g(V_i,X_i),  \quad N_c\ceq \sumn{i}\indic{V_{\mhit{2}i}=c}.
\end{align*} 
Assume that $\supnorm{\muX}<\infty$ and that there exists an $\rieszXhat:\Vimagesp_1\times\Ximagesp \to\real$ such that $\supnorm{\rieszXhat-\rieszX}=\smallOPs{P_{VX}}{1}$, where $r$ is the Riesz representer of $\mu\mapsto \E f(V,X,\mu,\gammaV(c))$ $=\chi(P_{VX})$ as before. Further assume that for every fixed $\mu\in L_2(P_{V_{\mhit{1}}X})$ and for every $\gammaVtld(c)$ between $\gammaV(c)$ and $\gammaVhat(c)$, the bound conditions
\begin{align}
(\Probn-P_{VX})\partial_\gamma f(V,X,\mu,\gammaVtld(c))&=\bigOPs{P_{VX}}{1},\label{priv:eq:emproc_discrete_bound1}  \\
(\Probn-P_{VX})\partial_\gamma^2 f(V,X,\mu,\gammaVtld(c))&=\bigOPs{P_{VX}}{1},\label{priv:eq:emproc_discrete_bound2} \\
P_{VX}\partial_\gamma f(V,X,\muX,\gammaVtld(c))&=\bigOPs{P_{VX}}{1}, \label{priv:eq:emproc_discrete_bound3} \\
P_{VX}\partial_\gamma^2 f(V,X,\muX,\gammaVtld(c))&=\bigOPs{P_{VX}}{1}, \label{priv:eq:emproc_discrete_bound4} \\
P_{VX}\partial_\gamma^2 f(V,X,\muXhat,\gammaVtld(c))&=\bigOPs{P_{VX}}{1} \label{priv:eq:emproc_discrete_bound5}
\end{align}
hold. Then $\sqrt{n}(\chi(\hat P_{VX})-\chi(P_{VX}))\rsquigs{P_{VX}}\normaldist(0, P_{VX}\tld\chi^2)$ as $n\to\infty$.
\end{proposition}

Note that for particular forms of $f$, the conditions of \Cref{priv:prop:asymeff_plugin_discretecovar} are easier to verify. Namely, if $f(v,x,\mu,\gamma)$ factors as
$$f(v,x,\mu,\gamma)=f_1(v,x,\mu)f_2(\gamma)$$
where $\fr{\partial ^2 f_2}{\partial\gamma^2}$ exists and is continuous and $(v,x)\mapsto f_1(v,x,\mu)$ belongs to $L_2(P_{VX})$ for any $\mu\in L_2(P_{V_{\mhit{1}}X})$, then the boundedness conditions \eqref{priv:eq:emproc_discrete_bound1},\eqref{priv:eq:emproc_discrete_bound2},\eqref{priv:eq:emproc_discrete_bound3}, \eqref{priv:eq:emproc_discrete_bound4} are easily met by the standard central limit theorem and the continuous mapping theorem. For instance, this includes the average treatment effect on the treated in \Cref{priv:ex:att_dr}.

\subsection{Proofs}
\label{priv:app:sec:nonprivate_estimation:proofs}

In this section, we prove the results in \Cref{priv:app:sec:nonprivate_estimation:results}: \Cref{priv:lem:chi_eff_empprocess,priv:cor:eff_chihat,priv:prop:asymeff_plugin_discretecovar}.

\begin{proof}[Proof of \Cref{priv:lem:chi_eff_empprocess}]
Throughout, we apply that if a $P_{VX}$-random function $\hat q \in L_2(P_{VX})$ is independent of the random sample generating the process $\Probn$, then
\begin{align}
\label{priv:eq:meansq_empproc}
\begin{aligned}
\int(\hat{q}(v,x)-q(v,x))^2 \dderiv P_{VX}(v,x)=\smallOPs{P_{VX}}{1} \text{ implies } \\ 
 \sqrt{n}(\Probn-P_{VX})(\hat q -q)=\smallOPs{P_{VX}}{1};
\end{aligned}
\end{align}
which follows from Markov's inequality and the dominated convergence theorem.

By the definitions \eqref{priv:eq:eif_chi} and \eqref{priv:eq:eif_chi_hat},
\begin{align}
\hat{\tld\chi}(v,x)-\tld\chi(v,x) =&\, T_1(v,x)+T_2(v,x)+T_3(v,x)+T_4, \label{priv:eq:emproc_decomp} \\
T_1(v,x)\ceq&\, \rieszXhat(v_{\mhit{1}},x)(m(v,x)-\muXhat(v_{\mhit{1}},x)) \nonumber \\
&-\rieszX(v_{\mhit{1}},x)(m(v,x)-\muX(v_{\mhit{1}},x)), \nonumber  \\
T_2(v,x)\ceq&\, \fr{\indic{v_{\mhit{2}}=c}}{\hat p_{V_{\mhit{2}}}(c)}(g(v,x)-\gammaVhat(c))\expderivhat -\fr{\indic{v_{\mhit{2}}=c}}{p_{V_{\mhit{2}}}(c)}(g(v,x)-\gammaV(c))\expderiv, \nonumber \\
T_3(v,x)\ceq&\, f(v,x, \muXhat, \gammaVhat(c)) -f(v,x, \muX, \gammaV(c)), \nonumber  \\
T_4\ceq&\, -\chi(\hat P_{VX})+\chi(P_{VX}). \nonumber 
\end{align}
As $T_4$ is constant, not depending on $(v,x)$, $(\Probn-P_{VX})T_4=0$. It remains to show $(\Probn-P_{VX})T_j=\smallOPs{P_{VX}}{n^{-1/2}}$ for $j=1,2,3$ by the linearity of the process $\Probn-P_{VX}$.

\emph{Term }$T_1.\,\,$ Suppressing the arguments, write
\begin{align}
T_1=\rieszXhat(m-\muXhat)-\rieszX(m-\muX) &= (\rieszXhat-\rieszX+\rieszX)(m-\muXhat)-\rieszX(m-\muX) \\
&= (\rieszXhat-\rieszX)(m-\muXhat) + \rieszX(\muX-\muXhat).
\end{align}
By \Cref{priv:ass:consistent_nuisance}, $\supnorm{m-\muXhat}=\bigOPs{P_{VX}}{1}$; either directly by 
\eqref{priv:eq:consistency_bigobound_hat}, or by \eqref{priv:eq:consistency_bigobound} and \eqref{priv:eq:consistency_mu_supn}, noting that $\supnorm{m-\muXhat}\leq \supnorm{m-\muX}+\supnorm{\muX-\muXhat}=\bigOPs{P_{VX}}{1}+\smallOPs{P_{VX}}{1}=\bigOPs{P_{VX}}{1}$.
Then the $L_2(P_{V_{\mhit{1}}X})$-convergence \eqref{priv:eq:consistency_riesz_l2} of $\rieszX$ implies that $(\Probn-P_{VX})((\rieszXhat-\rieszX)(m-\muXhat))=\smallOPs{P_{VX}}{n^{-1/2}}$ by \eqref{priv:eq:meansq_empproc} as
\begin{align*}
P_{VX}((\rieszXhat-\rieszX)^2(m-\muXhat)^2) \\
\leq\supnorm{m-\muXhat}^2 P_{VX}(\rieszXhat-\rieszX)^2=\bigOPs{P_{VX}}{1}\smallOPs{P_{VX}}{1}=\smallOPs{P_{VX}}{1}
\end{align*}
since $\supnorm{|q|^2}=\supnorm{q}^2$. 

By \Cref{priv:ass:consistent_nuisance}, either \eqref{priv:eq:consistency_mu_supn}, or \eqref{priv:eq:consistency_mu_l2} and \eqref{priv:eq:consistency_riesz_bound} hold. In the former case,
\begin{align*}
P_{VX}(\rieszX^2(\muX-\muXhat)^2)\leq\supnorm{\muX-\muXhat}^2 P_{VX}\rieszX^2=\smallOPs{P_{VX}}{1},
\end{align*}
by \eqref{priv:eq:meansq_empproc} because $\rieszX\in L_2(P_{V_{\mhit{1}}X})$. In the latter case, letting $$B\ceq \set{(V_{\mhit{1}},X)\in \Vimagesp_{\mhit{1}}\times\Ximagesp: |r(V_{\mhit{1}},X)|>\ba R }$$ with complement $B^C$, we have
\begin{align*}
P_{VX}(\rieszX^2(\muX-\muXhat)^2) &= \int_{B^C} r^2(\muX-\muXhat)^2 \dderiv P_{VX}+ \int_{B} r^2(\muX-\muXhat)^2 \dderiv P_{VX} \\
&\leq \ba R^2 \norm{L_2(P_{V_{\mhit{1}X}})}{\muX-\muXhat}^2+0=\smallOPs{P_{VX}}{1},
\end{align*}
since by \eqref{priv:eq:consistency_riesz_bound}, $P_{VX}(B)=0$ and $\muXhat$ is $L_2(P_{V_{\mhit{1}X}})$-convergent by \eqref{priv:eq:consistency_mu_l2}. Thus, $(\Probn-P_{VX})(\rieszX(\muX-\muXhat))=\smallOPs{P_{VX}}{n^{-1/2}}$ by \eqref{priv:eq:meansq_empproc}. Conclude that  $(\Probn-P_{VX})T_1=\smallOPs{P_{VX}}{n^{-1/2}}$.

\emph{Term }$T_2.\,\,$ By the mean-value theorem there exists $(\gammaVtld(c), \tld p_{V_{\mhit{2}}}(c), \tld\expderiv)$ between 
\begin{align*}
(\gammaV(c), p_{V_{\mhit{2}}}(c), \expderiv) \text{ and } (\gammaVhat(c), \hat p_{V_{\mhit{2}}}(c), \expderivhat)
\end{align*}
such that
\begin{align*}
T_2(v,x)=&\,\fr{\indic{v_{\mhit{2}}=c}}{\hat p_{V_{\mhit{2}}}(c)}(g(v,x)-\gammaVhat(c))\expderivhat -\fr{\indic{v_{\mhit{2}}=c}}{p_{V_{\mhit{2}}}(c)}(g(v,x)-\gammaV(c))\expderiv \\
=&\,-\fr{\indic{v_{\mhit{2}}=c}}{\tld p_{V_{\mhit{2}}}(c)}\tld\expderiv (\gammaVhat(c)-\gammaV(c)) \\
&-\fr{\indic{v_{\mhit{2}}=c}}{\tld p_{V_{\mhit{2}}}(c)^2}(g(v,x)-\gammaVtld(c))\tld\expderiv(\hat p_{V_{\mhit{2}}}(c)-p_{V_{\mhit{2}}}(c)) \\
&+\fr{\indic{v_{\mhit{2}}=c}}{\tld p_{V_{\mhit{2}}}(c)}(g(v,x)-\gammaVtld(c))(\expderivhat-\expderiv).
\end{align*}
By the standard central limit theorem, $\sqrt{n}(\Probn-P_{VX})\indic{V_{\mhit{2}}=c}=\bigOPs{P_{VX}}{1}$. By the linearity of the process $\Probn-P_{VX}$,
\begin{align*}
\sqrt{n}(\Probn-P_{VX})\big[\indic{V_{\mhit{2}}=c}g(V,X)-\indic{V_{\mhit{2}}=c}\gammaVtld(c)\big] \\
= \sqrt{n}(\Probn-P_{VX})\indic{V_{\mhit{2}}=c}g(V,X) - \gammaVtld(c)\sqrt{n}(\Probn-P_{VX})\indic{V_{\mhit{2}}=c} \\
=(1-\gammaVtld(c))\bigOPs{P_{VX}}{1}=\bigOPs{P_{VX}}{1}
\end{align*} 
again by the standard central limit theorem and \eqref{priv:eq:consistency_gammaVhat}. Suppose that $\expderivhat-\expderiv=\smallOPs{P_{VX}}{1}$, which we show later. Then by \eqref{priv:eq:consistency_gammaVhat} and \eqref{priv:eq:consistency_phat}, $(\Probn-P_{VX})T_2=\smallOPs{P_{VX}}{n^{-1/2}}$.

\emph{Term }$T_3.\,\,$ Recall that $T_3(v,x)=f(v,x, \muXhat, \gammaVhat(c))-f(v,x, \muX, \gammaV(c))$. The continuity \eqref{priv:eq:sq_continu}, together with the consistency 
of $\gammaVhat$ and $\muXhat$ (\eqref{priv:eq:consistency_gammaVhat} and \eqref{priv:eq:consistency_mu_supn} or \eqref{priv:eq:consistency_mu_l2}) imply that $\int T_3(v,x)^2\dderiv P_{VX}(v,x)=\smallOPs{P_{VX}}{1}$ by the continuous mapping theorem. Conclude by \eqref{priv:eq:meansq_empproc}  that  $(\Probn-P_{VX})T_3=\smallOPs{P_{VX}}{n^{-1/2}}$.

\emph{Consistency of } $\expderivhat.\,\,$ By the definition of $\expderiv$ and $\expderivhat$,
\begin{align*}
\expderivhat-\expderiv =&\, \Probn''\partial_\gamma f(V,X,\muXhat,\gammaVhat(c)) - P_{VX}\partial_\gamma f(V,X,\muX,\gammaV(c))   \\
=&\, \Probn''\partial_\gamma f(V,X,\muXhat,\gammaVhat(c))-P_{VX}\partial_\gamma f(V,X,\muXhat,\gammaVhat(c)) \\
&+ P_{VX}\partial_\gamma f(V,X,\muXhat,\gammaVhat(c))- P_{VX}\partial_\gamma f(V,X,\muX,\gammaV(c)) \\
=&\, (\Probn''-P_{VX})\partial_\gamma f(V,X,\muXhat,\gammaVhat(c)) \\
&+ P_{VX}\big[\partial_\gamma f(V,X,\muXhat,\gammaVhat(c))-\partial_\gamma f(V,X,\muX,\gammaV(c))  \big] \\
=&\, (\Probn''-P_{VX})\partial_\gamma f(V,X,\muX,\gammaV(c)) \\
&+(\Probn''-P_{VX})\big[\partial_\gamma f(V,X,\muXhat,\gammaVhat(c))-\partial_\gamma f(V,X,\muX,\gammaV(c))\big] \\
&+ P_{VX}\big[\partial_\gamma f(V,X,\muXhat,\gammaVhat(c))-\partial_\gamma f(V,X,\muX,\gammaV(c))  \big].
\end{align*}
Here, the first term is $\bigOPs{P_{VX}}{n^{-1/2}}=\smallOPs{P_{VX}}{1}$ by the standard central limit theorem, and the second and third term are $\smallOPs{P_{VX}}{1}$ by the continuity \eqref{priv:eq:sq_continu} along the same arguments concerning $T_3$ above.
\end{proof}

\begin{proof}[Proof of \Cref{priv:cor:eff_chihat}]
Follows from \Cref{priv:lem:chi_eff_empprocess,priv:thm:dr}, noting that, for $\expderiv'$ in \eqref{priv:eq:derivhatprime},
$$\expderivhat-\expderiv'=(\Probn''-P_{VX})\partial_\gamma f(V,X,\muXhat,\gammaVhat(c))$$ is $\smallOPs{P_{VX}}{1}$ by the consistency proof of $\expderivhat$ in \Cref{priv:lem:chi_eff_empprocess}.
\end{proof}

\begin{proof}[Proof of \Cref{priv:prop:asymeff_plugin_discretecovar}]
We start with an auxiliary result. Define
\begin{align*}
\hat{\tld\chi}(v,x)\ceq&\,  \rieszXhat(v_{\mhit{1}},x)(m(v,x)-\muXhat(v_{\mhit{1}},x))+\fr{\indic{v_{\mhit{2}}=c}}{\hat p_{V_{\mhit{2}}}(c)}(g(v,x)-\gammaVhat(c))\hat\expderiv\nonumber \\
&+ f(v,x, \muXhat, \gammaVhat(c))-\chi(\hat P_{VX}) \\
\hat p_{V_{\mhit{2}}}(c)\ceq&\, N_c/n.
\end{align*} 
First we show that 
\begin{align*}
\Probn\hat{\tld\chi} =\Probn \left[\rieszXhat(V_{\mhit{1}},X)(m(V,X)-\muXhat(V_{\mhit{1}},X)\right] \\
+ \fr{\expderivhat}{\hat p_{V_{\mhit{2}}}(c)}\Probn \left[\indic{V_{\mhit{2}}=c}g(V,X)-\gammaVhat(c)\indic{V_{\mhit{2}}=c}\right] 
+\Probn\left[f(V,X,\muXhat,\gammaVhat(c))-\chi(\hat P_{VX}))\right]
\end{align*}
is zero. We apply an `empirical tower property' to the first term to get
\begin{align*}
\Probn \rieszXhat(V_{\mhit{1}},X)(m(V,X)-\muXhat(V_{\mhit{1}},X)) = \nsumn{i}\rieszXhat(V_{\mhit{1}i},X_i)(m(V_i,X_i)-\muXhat(V_{\mhit{1}i},X_i))  \\
= \fr{1}{n}\sum_{(v_{\mhit{1}}, x)\in \Vimagesp_{\mhit{1}}\times\Ximagesp}\sum_{i: (V_{\mhit{1}i},X_i)=(v_{\mhit{1}}, x)}\rieszXhat(V_{\mhit{1}i},X_i)(m(V_i,X_i)-\muXhat(V_{\mhit{1}i},X_i))  \\
= \fr{1}{n}\sum_{(v_{\mhit{1}}, x)\in \Vimagesp_{\mhit{1}}\times\Ximagesp} \left\{\rieszXhat(v_{\mhit{1}}, x) \vphantom{ \left[\left(\sum_{i: (V_{\mhit{1}i},X_i)=(v_{\mhit{1}}, x)}m(V_i,X_i)\right) -\left( \sum_{i: (V_{\mhit{1}i},X_i)=(v_{\mhit{1}}, x)}\muXhat(v_{\mhit{1}}, x)\right) \right]} \right. \\
\left. \times \left[\left(\sum_{i: (V_{\mhit{1}i},X_i)=(v_{\mhit{1}}, x)}m(V_i,X_i)\right) -\left( \sum_{i: (V_{\mhit{1}i},X_i)=(v_{\mhit{1}}, x)}\muXhat(v_{\mhit{1}}, x)\right) \right]\right\} \\
= \fr{1}{n}\sum_{(v_{\mhit{1}}, x)\in \Vimagesp_{\mhit{1}}\times\Ximagesp}\rieszXhat(v_{\mhit{1}}, x)\left[N_{v_{\mhit{1}}x}\muXhat(v_{\mhit{1}}, x)-N_{v_{\mhit{1}}x}\muXhat(v_{\mhit{1}}, x) \right] = 0.
\end{align*}
The second term satisfies
\begin{align*}
\Probn \left[\indic{V_{\mhit{2}}=c}g(V,X)-\gammaVhat(c)\indic{V_{\mhit{2}}=c}\right]&=\Probn\indic{V_{\mhit{2}}=c}g(V,X) - \gammaVhat(c)\Probn \indic{V_{\mhit{2}}=c} \\
&= N_c \gammaVhat(c)-\gammaVhat(c)N_c=0
\end{align*}
by the definition of $\gammaVhat(c), N_c$. The last term satisfies $$\Probn \left[f(V,X,\muXhat,\gammaVhat(c))-\chi(\hat P_{VX}))\right]=\Probn f(V,X,\muXhat,\gammaVhat(c))-\chi(\hat P_{VX})=0$$ by the definition of $\chi(\hat P_{VX})$ as $\chi(\hat P_{VX})$ is constant with respect to the empirical measure $\Probn$. Conclude that $\Probn\hat{\tld\chi}=0$.

Now we show asymptotic efficiency. Using that $\Probn \hat{\tld\chi}=0$, write
\begin{align*}
\sqrt{n}(\chi(\hat P_{VX})-\chi(P_{VX})) &= \sqrt{n}\Probn \tld\chi + \sqrt{n}(\Probn-P_{VX})(\hat{\tld\chi}-\tld\chi)+\sqrt{n} R_n, \\
R_n &= \chi(\hat P_{VX})-\chi(P_{VX}) + P_{VX}\hat{\tld\chi},
\end{align*}
for $\tld\chi$ in \eqref{priv:eq:eif_chi}. The first term satisfies $\sqrt{n}\Probn \tld\chi \rsquigs{P_{VX}}\normaldist(0, P_{VX}\tld\chi^2)$. Because the arguments in \Cref{priv:thm:dr} also apply when $(\muXhat,\gammaVhat)$ and $\chi(\hat P_{VX})$ are computed from the same sample $\mathcal{S}$, we have $\sqrt{n}R_n=\smallOPs{P_{VX}}{1}$. This follows from \eqref{priv:eq:chihat_biasdecomp} of \Cref{priv:thm:dr},  because $\muXhat$, $\gammaVhat(c)$ and $\hat p_{V_{\mhit{2}}}(c)$ are asymptotically normal, $\rieszXhat$ is consistent , $P_{VX}\partial_\gamma^2 f(V,X,\muXhat,\gammaVtld(c)=\bigOPs{P_{VX}}{1}$ by \eqref{priv:eq:emproc_discrete_bound5}, and, as we show at the end of this proof, $\expderivhat-\expderiv'=\smallOPs{P_{VX}}{1}$ for $\expderiv'=P_{VX}\partial_\gamma f(V,X,\muXhat,\gammaVhat(c))$. 

It remains to show $\sqrt{n}(\Probn-P_{VX})(\hat{\tld\chi}-\tld\chi)=\smallOPs{P_{VX}}{1}$. Because \Cref{priv:lem:chi_eff_empprocess} assumes that $(\muXhat,\gammaVhat)$ are computed from a sample independent from what generates $\Probn$, we need to adapt the arguments therein. As in \eqref{priv:eq:emproc_decomp}, write
\begin{align}
\hat{\tld\chi}(v,x)-\tld\chi(v,x) =&\, T_1(v,x)+T_2(v,x)+T_3(v,x)+T_4, \label{priv:eq:emproc_decomp_discrete} \\
T_1(v,x)\ceq&\, \rieszXhat(v_{\mhit{1}},x)(m(v,x)-\muXhat(v_{\mhit{1}},x)) \nonumber \\
&-\rieszX(v_{\mhit{1}},x)(m(v,x)-\muX(v_{\mhit{1}},x)), \nonumber  \\
T_2(v,x)\ceq&\,\fr{\indic{v_{\mhit{2}}=c}}{\hat p_{V_{\mhit{2}}}(c)}(g(v,x)-\gammaVhat(c))\expderivhat -\fr{\indic{v_{\mhit{2}}=c}}{p_{V_{\mhit{2}}}(c)}(g(v,x)-\gammaV(c))\expderiv, \nonumber \\
T_3(v,x)\ceq&\, f(v,x, \muXhat, \gammaVhat(c)) -f(v,x, \muX, \gammaV(c)), \nonumber  \\
T_4\ceq&\, -\chi(\hat P_{VX})+\chi(P_{VX}). \nonumber 
\end{align}
Again, $T_4$ being a constant under $\Probn-P_{VX}$, $\sqrt{n}(\Probn-P_{VX})T_4=0$, so we need show $(\Probn-P_{VX})T_j=\smallOPs{P_{VX}}{n^{-1/2}}$ for $j=1,2,3$. 

\emph{Term }$T_1.\,\,$ Write
\begin{align}
\sqrt{n}(\Probn-P_{VX})T_1(V,X) = \sqrt{n}(\Probn-P_{VX})\left[m(V,X)(\rieszXhat(V_{\mhit{1}},X)-\rieszX(V_{\mhit{1}},X))\right] \nonumber \\
+\sqrt{n}(\Probn-P_{VX})\left[\muXhat(V_{\mhit{1}},X)(\rieszXhat(V_{\mhit{1}},X)-\rieszX(V_{\mhit{1}},X))\right] \label{priv:eq:emproc_decomp_discrete_t1_2}  \\
+\sqrt{n}(\Probn-P_{VX})\left[\rieszX(V_{\mhit{1}},X)(\muX(V_{\mhit{1}},X)-\muXhat(V_{\mhit{1}},X))\right].  \label{priv:eq:emproc_decomp_discrete_t1_3} 
\end{align}
Because $\Vimagesp_\mhit{1}\times\Ximagesp$ is finite, we can write 
\begin{align*}
\rieszX(v_{\mhit{1}},x)&=\sum_{(\ba v_{\mhit{1}},\ba x)\in \Vimagesp_\mhit{1}\times\Ximagesp}\varrho_{\ba v_{\mhit{1}},\ba x}\indic{(\ba v_{\mhit{1}},\ba x)}(v_{\mhit{1}},x), \quad \varrho_{\ba v_{\mhit{1}},\ba x}\ceq \rieszX(\ba v_{\mhit{1}},\ba x), \\
\rieszXhat(v_{\mhit{1}},x)&=\sum_{(\ba v_{\mhit{1}},\ba x)\in \Vimagesp_\mhit{1}\times\Ximagesp}\hat\varrho_{\ba v_{\mhit{1}},\ba x}\indic{(\ba v_{\mhit{1}},\ba x)}(v_{\mhit{1}},x), \quad \hat\varrho_{\ba v_{\mhit{1}},\ba x}\ceq \rieszXhat(\ba v_{\mhit{1}},\ba x).
\end{align*}
But then
\begin{align*}
 \sqrt{n}(\Probn-P_{VX})\left[m(V,X)(\rieszXhat(V_{\mhit{1}},X)-\rieszX(V_{\mhit{1}},X))\right] \\
 =\sqrt{n}(\Probn-P_{VX})\left[m(V,X)\left(\sum_{(\ba v_{\mhit{1}},\ba x)\in \Vimagesp_\mhit{1}\times\Ximagesp}(\hat\varrho_{\ba v_{\mhit{1}},\ba x}-\varrho_{\ba v_{\mhit{1}},\ba x})\indic{(\ba v_{\mhit{1}},\ba x)}(V_{\mhit{1}},X)\right)\right]  \\
 = \sum_{(\ba v_{\mhit{1}},\ba x)\in \Vimagesp_\mhit{1}\times\Ximagesp}(\hat\varrho_{\ba v_{\mhit{1}},\ba x}-\varrho_{\ba v_{\mhit{1}},\ba x})\sqrt{n}(\Probn-P_{VX})\left[ m(V,X)\indic{(\ba v_{\mhit{1}},\ba x)}(V_{\mhit{1}},X)\right],
\end{align*}
which is $\smallOPs{P_{VX}}{1}$, because $\rieszXhat$ is consistent, $$\sqrt{n}(\Probn-P_{VX})\left[ m(V,X)\indic{(\ba v_{\mhit{1}},\ba x)}(V_{\mhit{1}},X)\right]=\bigOPs{P_{VX}}{1}$$ by the standard central limit theorem, and $|\Vimagesp_\mhit{1}\times\Ximagesp|$ is finite. The terms \eqref{priv:eq:emproc_decomp_discrete_t1_2}, \eqref{priv:eq:emproc_decomp_discrete_t1_3} can be handled similarly because $\supnorm{\muXhat}\leq \supnorm{\muXhat-\muX}+\supnorm{\muX}=\smallOPs{P_{VX}}{1}+\bigO{1}=\bigOPs{P_{VX}}{1}$ by assumption. Thus $\sqrt{n}(\Probn-P_{VX})T_1=\smallOPs{P_{VX}}{1}$.

\emph{Term }$T_2.\,\,$ Same arguments as in \Cref{priv:lem:chi_eff_empprocess} apply, yielding $\sqrt{n}(\Probn-P_{VX})T_2$ $=\smallOPs{P_{VX}}{1}$, because $\expderivhat-\expderiv=\smallOPs{P_{VX}}{1}$ --- which we show at the end of this proof ---, and $\gammaVhat(c),\hat p_{V_{\mhit{2}}}(c)$ are consistent. 

\emph{Term }$T_3.\,\,$ Write
\begin{align}
T_3(v,x)=&\,f(v,x,\muXhat,\gammaVhat(c))-f(v,x,\muX,\gammaV(c)) \nonumber \\
	=&\,f(v,x,\muXhat,\gammaVhat(c))-f(v,x,\muX,\gammaVhat(c)) \nonumber \\
	&+f(v,x,\muX,\gammaVhat(c))-f(v,x,\muX,\gammaV(c)). \label{priv:eq:emproc_decomp_discrete_t3_2}
\end{align}
As with $\rieszX$ before, represent $\muX(v_{\mhit{1}},x)=\sum_{(\ba v_{\mhit{1}},\ba x)\in \Vimagesp_\mhit{1}\times\Ximagesp}\lambda_{\ba v_{\mhit{1}},\ba x}\indic{(\ba v_{\mhit{1}},\ba x)}(v_{\mhit{1}},x)$ for some parameter $\lambda\in\real^{|\Vimagesp_\mhit{1}\times\Ximagesp|}$ with $\lambda_{\ba v_{\mhit{1}},\ba x}\ceq \muX(\ba v_{\mhit{1}},\ba x)$; similarly, write 
\begin{align*}
\muXhat(v_{\mhit{1}},x)&=\sum_{(\ba v_{\mhit{1}},\ba x)\in \Vimagesp_\mhit{1}\times\Ximagesp}\hat\lambda_{\ba v_{\mhit{1}},\ba x}\indic{(\ba v_{\mhit{1}},\ba x)}(v_{\mhit{1}},x), \\
\hat\lambda_{\ba v_{\mhit{1}},\ba x}&\ceq \muXhat(\ba v_{\mhit{1}},\ba x).
\end{align*}
By the linearity \eqref{priv:eq:f_lin} of $f$ and the mean value theorem,
\begin{align*}
\sqrt{n}(\Probn-P_{VX})\left[f(V,X,\muXhat,\gammaVhat(c))-f(V,X,\muX,\gammaV(c))\right] \\
=\sqrt{n}(\Probn-P_{VX})\left[f(V,X,\muXhat-\muX,\gammaVhat(c))\right] \\
=\sqrt{n}(\Probn-P_{VX})\left[\sum_{(\ba v_{\mhit{1}},\ba x)\in \Vimagesp_\mhit{1}\times\Ximagesp}(\hat\lambda_{\ba v_{\mhit{1}},\ba x}-\lambda_{\ba v_{\mhit{1}},\ba x}) f(V,X,\indic{\ba v_{\mhit{1}},\ba x},\gammaVhat(c))\right] \\
= \sum_{(\ba v_{\mhit{1}},\ba x)\in \Vimagesp_\mhit{1}\times\Ximagesp}(\hat\lambda_{\ba v_{\mhit{1}},\ba x}-\lambda_{\ba v_{\mhit{1}},\ba x})\sqrt{n}(\Probn-P_{VX})\left[f(V,X,\indic{\ba v_{\mhit{1}},\ba x},\gammaVhat(c))\right] \\
= \sum_{(\ba v_{\mhit{1}},\ba x)\in \Vimagesp_\mhit{1}\times\Ximagesp}(\hat\lambda_{\ba v_{\mhit{1}},\ba x}-\lambda_{\ba v_{\mhit{1}},\ba x})\sqrt{n}(\Probn-P_{VX})\left[f(V,X,\indic{\ba v_{\mhit{1}},\ba x},\gammaV(c))\right] \\
+\sum_{(\ba v_{\mhit{1}},\ba x)\in \Vimagesp_\mhit{1}\times\Ximagesp}(\hat\lambda_{\ba v_{\mhit{1}},\ba x}-\lambda_{\ba v_{\mhit{1}},\ba x})(\Probn-P_{VX})\left[\partial_\gamma f(V,X,\indic{\ba v_{\mhit{1}},\ba x},\gammaVtld(c))\right] \\
\times \sqrt{n}(\gammaVhat(c)-\gammaV(c))
\end{align*}
for some $\gammaVtld(c)$ between $\gammaV(c)$ and $\gammaVhat(c)$. But this is $\smallOPs{P_{VX}}{1}$ by the standard central limit theorem, and because $\sqrt{n}(\gammaVhat(c)-\gammaV(c))=\bigOPs{P_{VX}}{1}$, $\muXhat$ being consistent and $(\Probn-P_{VX})\left[\partial_\gamma f(V,X,\indic{\ba v_{\mhit{1}},\ba x},\gammaVtld(c))\right]=\bigOPs{P_{VX}}{1}$ by \eqref{priv:eq:emproc_discrete_bound1}. For the term \eqref{priv:eq:emproc_decomp_discrete_t3_2}, apply the mean-value theorem twice to get
\begin{align*}
T_{3,1}(v,x)\ceq f(v,x,\muX,\gammaVhat(c))-f(v,x,\muX,\gammaV(c)) \\
=(\gammaVhat(c)-\gammaV(c))\left\{\partial_\gamma f(v,x,\muX,\gammaV(c))+(\gammaVtld(c)-\gammaV(c))\partial_\gamma^2 f(v,x,\muX,\gammaVtld'(c))\right\}
\end{align*}
for some $\gammaVtld'(c)$ between $\gammaV(c)$ and $\gammaVhat(c)$. Thus $\sqrt{n}(\Probn-P_{VX})T_{3,1}=\smallOPs{P_{VX}}{1}$ by the standard central limit theorem, $\sqrt{n}$-consistency of $\gammaVhat(c)$, and stochastic boundedness \eqref{priv:eq:emproc_discrete_bound2} of the second derivative. Hence, $\sqrt{n}(\Probn-P_{VX})T_3=\smallOPs{P_{VX}}{1}$.

\emph{Consistency of } $\expderivhat.\,\,$  Write
\begin{align}
\expderivhat - \expderiv =&\, \Probn \partial_\gamma f(V,X,\muXhat,\gammaVhat(c))-P_{VX}\partial_\gamma
f(V,X,\muXhat,\gammaVhat(c)) \nonumber \\
=&\, (\Probn - P_{VX}) \partial_\gamma f(V,X,\muXhat,\gammaVhat(c)) \label{priv:eq:emproc_decomp_discrete_e_1} \\
&+ P_{VX}\left[ \partial_\gamma f(V,X,\muXhat,\gammaVhat(c))- \partial_\gamma f(V,X,\muX,\gammaV(c))\right]. \label{priv:eq:emproc_decomp_discrete_e_2}
\end{align}
As $\mu\mapsto \partial_\gamma f(v,x,\mu,\ba\gamma)$ is linear, \eqref{priv:eq:emproc_decomp_discrete_e_1} is $\smallOPs{P_{VX}}{1}$ along similar arguments as $T_3$ above; because we only need consistency, we only need the existence and stochastic boundedness \eqref{priv:eq:emproc_discrete_bound2} of the second derivative; no need for higher order derivatives. For the term \eqref{priv:eq:emproc_decomp_discrete_e_2}, write it as
\begin{align*}
P_{VX}\left[ \partial_\gamma f(V,X,\muXhat,\gammaVhat(c))- \partial_\gamma f(V,X,\muX,\gammaVhat(c))\right] \\
+P_{VX}\left[ \partial_\gamma f(V,X,\muX,\gammaVhat(c))- \partial_\gamma f(V,X,\muX,\gammaV(c))\right].
\end{align*}
By the linearity of $\mu\mapsto \partial_\gamma f(v,x,\mu,\gammaVhat(c))$, the first term here is $\smallOPs{P_{VX}}{1}$, following a mean-value expansion, by the consistency of $\muXhat$ and the assumed boundedness \eqref{priv:eq:emproc_discrete_bound4} of $P_{VX}\partial_\gamma f(V,X,\muX,\gammaVhat(c))$. The second term is $\smallOPs{P_{VX}}{1}$ by the same arguments under assumption \eqref{priv:eq:emproc_discrete_bound4} for $P_{VX}\partial_\gamma^2 f(V,X,\muX,\gammaVhat(c))$. Hence, $\expderivhat-\expderiv=\smallOPs{P_{VX}}{1}$. Finally, note that $\expderivhat-\expderiv'=\smallOPs{P_{VX}}{1}$ too as we claimed, because it is equal to \eqref{priv:eq:emproc_decomp_discrete_e_1}.
\end{proof}

\end{appendices}

\end{document}